\documentclass[11pt]{article}
\usepackage{amsmath}
\usepackage{bbm}
\usepackage{amsfonts}
\usepackage{graphicx}
\usepackage{enumerate,booktabs}
\usepackage{mathtools}
\usepackage[sort,compress]{cite}
\usepackage{parskip}
\usepackage[papersize={8.5in,11in},margin=1in]{geometry}
\usepackage{caption}
\usepackage{capt-of}
\usepackage{algorithm,algorithmic}
\usepackage{subcaption}
\usepackage{xcolor}
\usepackage{hyperref}
\usepackage{multirow}
\usepackage{tikz}
\usetikzlibrary{positioning}
\usepackage{tabu}
\usepackage{colortbl}
\hypersetup{
  pdfcreator = {},
  pdfproducer = {}
}
\newtheorem{theorem}{Theorem}

\newtheorem{example}{Example}

\numberwithin{equation}{section}

\newenvironment{proof}[1][Proof]
{\noindent\textbf{#1.} }{\ \rule{0.5em}{0.5em}}

\newcommand*{\affaddr}[1]{#1} 
\newcommand*{\affmark}[1][*]{\textsuperscript{#1}}
\newcommand*{\email}[1]{\texttt{#1}}

\begin{document}

	\title{Neural network representation of the probability density function of diffusion processes}

\author{%
Wayne Isaac T. Uy\affmark[1,]\affmark[3], Mircea D. Grigoriu\affmark[1,]\affmark[2\thanks{both authors contributed equally to this work}]\\
\affaddr{
\affaddr{\affmark[1]Center for Applied Mathematics}\\
\affaddr{\affmark[2]Department of Civil and Environmental Engineering}\\
\affaddr{Cornell University} \\
\affmark[3]Courant Institute of Mathematical Sciences \\ New York University}\\
\email{wayne.uy@cims.nyu.edu, mdg12@cornell.edu}\\
}
\date{\vspace{-3ex}}

\maketitle

\begin{abstract}
Physics-informed neural networks are developed to characterize the state of dynamical systems in a random environment. The neural network approximates  the probability density function (pdf) or the characteristic function (chf) of the state of these systems which satisfy the Fokker-Planck equation or an integro-differential equation under Gaussian and/or Poisson white noises.
We examine analytically and numerically the advantages and disadvantages of solving each type of differential equation to characterize the state. It is also demonstrated how prior information of the dynamical system can be exploited to design and simplify the neural network architecture. Numerical examples show that: 1) the neural network solution can approximate the target solution even for partial integro-differential equations and system of PDEs describing the time evolution of the pdf/chf, 2)  solving either the Fokker-Planck equation or the chf differential equation  using neural networks yields similar pdfs of the state, and 3) the solution to these differential equations can be used to study the behavior of the state for different types of random forcings.
\end{abstract}


\section{Introduction}

Let $\boldsymbol{X}(t) \in \mathbb{R}^d$ be a random vector defined on  a probability space $(\Omega,\mathcal{F},P)$ whose dynamics are goverened by
$$d\boldsymbol{X}(t) = \boldsymbol{a}(\boldsymbol{X}(t)) \,dt + \boldsymbol{b}(\boldsymbol{X}(t)) \,d \boldsymbol{Y}(t), \,\,\, t \ge 0$$
where $\boldsymbol{Y}(t) \in \mathbb{R}^k$ is a stochastic process and $\boldsymbol{a} \in \mathbb{R}^{d \times 1}, \boldsymbol{b} \in \mathbb{R}^{d \times k}$. Such stochastic differential equations (SDE) are used to model complex systems that arise in various applications of science and engineering \cite{book:Grigoriu2002}. We are interested in computing the probability density function $f(\boldsymbol{x},t)$ of $\boldsymbol{X}(t)$ which can then be used to estimate statistics of $\boldsymbol{X}(t)$,  including its moments $E[\boldsymbol{X}(t)^p]$ and probabilities of events $P(\boldsymbol{X}(t) \in A)$, $A \in \mathcal{F}$.
If $\boldsymbol{Y}(t)$ is Brownian motion, $f(\boldsymbol{x},t)$ satisfies a partial differential equation (PDE) called the Fokker-Planck equation \cite{book:Risken1989}. Analytical solutions to this PDE are only available under particular conditions on the drift and diffusion matrices \cite{book:KloedenP1992,book:Risken1989} and moreover, numerical solutions via the finite element method become unstable if the state  $\boldsymbol{X}(t)$ has dimension $d > 3$
\cite{paper:PichlerMB2013,paper:MasudB2005,paper:WojtkiewiczB2000}.

Other methods have been sought to solve the Fokker-Planck equation in high dimensions for special cases. In general, they approximate solutions of PDEs on high-dimensional domains. We only provide a brief survey of recent work. In \cite{paper:ChoVK2016}, various algorithms are presented to solve kinetic PDEs in which the high-dimensional problem is transformed into a sequence of low-dimensional ones. Nonlinear high-dimensional PDEs are tackled in \cite{paper:DektorV2019} by decomposing the function space into lower-dimensional nested subspaces. The work by \cite{paper:ChenMT2018,paper:ChenM2018} study the Fokker-Planck equation for high-dimensional nonlinear turbulent dynamical systems which possess conditional Gaussian structures, i.e. a set of components of the state conditioned on the trajectory of the remaining components is a Gaussian process. Finally, \cite{paper:ChenR2018} reduces the high-dimensional Fokker-Planck equation into a 1 or 2-dimensional PDE and uses
the path integral solution to solve resulting the dimension-reduced PDE.

As an alternative to the above dimension-reduction approaches and the traditional methods for solving PDEs, neural networks have been proposed to solve nonlinear and/or high-dimensional PDEs in scientific, engineering, and financial applications \cite{paper:SirignanoS2018,paper:RaissiPK2019}. A neural network is used to represent the solution to the PDE while its parameters are obtained via optimization. The objective function enforces the neural network approximation to satisfy the governing equation of the PDE together with the initial and boundary conditions. This idea has been pursued in \cite{paper:RaissiPK2019} to solve the Schr\"odinger, Allen-Cahn, KdV, and Burger's equation in 1 dimension and the Navier Stokes equation in 2 dimensions which are nonlinear PDEs whose solution may exhibit nearly discontinuous behavior. The work\cite{paper:SirignanoS2018} successfully estimated free boundary PDE solutions on domains of up to 200 dimensions and solutions to the high-dimensional Hamilton-Jacobi-Bellman PDE.  

By utilizing the above methodology, a neural network-based approximation to the Fokker-Planck equation has been undertaken in \cite{paper:AlAradiCNJS2018,paper:AlAradiCNJS2020,paper:XuZLZLK2020}.  The 1-dimensional time-varying PDE was tackled in \cite{paper:AlAradiCNJS2018,paper:AlAradiCNJS2020} where it was noticed that it was necessary to incorporate the constraint that the Fokker-Planck solution integrates to 1 for all times in the optimization step. Otherwise, the authors showed that the neural network approximation could not recover the analytical solution. In \cite{paper:XuZLZLK2020}, the steady-state PDE of up to 3 dimensions was addressed and strategies were presented to account for the normalization constraint.

Building on existing work, we investigate how neural networks can be used to represent the pdf of a state vector satisfying a stochastic differential equation. In contrast to \cite{paper:XuZLZLK2020}, we seek the pdf over a time interval instead of the steady-state solution. In addition, while \cite{paper:AlAradiCNJS2018,paper:AlAradiCNJS2020,paper:XuZLZLK2020}  
are only concerned with the Fokker-Planck equation, we also consider an alternative differential equation which describes the time evolution of the characteristic function of the state. The pdf and chf offer identical information about $\boldsymbol{X}(t)$ such that both can be used to compute its statistics, however, the latter is complex-valued.  We study the advantages and disadvantages of solving the Fokker-Planck equation or the differential equation for the chf in order to approximate the pdf of  $\boldsymbol{X}(t)$ from an analytical and numerical perspective. In particular, we highlight situations in which solving the Fokker-Planck equation may not be favorable regardless of the solution method employed. Strategies are then outlined on how the neural network architecture can be designed and simplified by exploiting probabilistic information from the SDE. The numerical examples feature the capabilities of the neural network solution to match the target solution for various dynamical systems subject to different types of noise. This work serves as a proof of concept of our objectives and adapts the methodology of \cite{paper:RaissiPK2019}. Extensions to the high-dimensional situations can be accomplished following \cite{paper:SirignanoS2018}. The differential equations dealt with here are different from the examples presented in \cite{paper:RaissiPK2019}; for instance, the time evolution of the chf may be represented by a partial integro-differential equation in which the highest order of the partial derivative can exceed 2.

A brief survey of neural networks and its application to solving PDEs in the spirit of \cite{paper:RaissiPK2019} is presented in Section~\ref{sec:PINN}. In Section~\ref{sec:CharFunFPPDE}, we derive a differential equation for the characteristic function of the state subject to commonly used random forcing via stochastic analysis.  A comparison between this differential equation and the Fokker-Planck equation is also performed. A neural network-based solution for these differential equations is then presented in Section~\ref{sec:NNapproxPDE}. Finally, Section~\ref{sec:NumExp} showcases the approximation properties of neural networks for a variety of applications aligned with our objective.

\section{Physics-informed neural networks} \label{sec:PINN}

A brief survey of the physics-informed neural networks framework \cite{paper:RaissiPK2019} for approximating solutions to PDEs is outlined in this section. Section~\ref{subsec:ReviewNN} describes the components of the neural network architecture and training of its parameters as employed in machine learning. Section~\ref{subsec:reviewPINN} then elaborates how neural networks can be trained to represent solutions to PDEs.

\subsection{Review of neural networks} \label{subsec:ReviewNN}

Neural networks traditionally employed in machine learning construct an approximation $\widetilde{\boldsymbol{f}}(\boldsymbol{x},t)$ to an unknown mapping $(\boldsymbol{x},t) \in  \mathbb{R}^{d+1} \mapsto \boldsymbol{f}(\boldsymbol{x},t) \in \mathbb{R}^m$ from data on the input and the output. Several types of neural network architectures exist \cite{book:GoodfellowBC2016}; in this work, we only focus on feedforward neural networks. Denote the components of $\boldsymbol{x} \in \mathbb{R}^d$ and $\boldsymbol{f} \in \mathbb{R}^m$ by $\boldsymbol{x} = (x_1,\dots,x_d)$ and $\boldsymbol{f} = (f_1,\dots,f_m)$ respectively, and suppose that $N$ data points $\{(\boldsymbol{x}_i,t_i,\boldsymbol{f}(\boldsymbol{x}_i,t_i))\}_{i=1}^N$ of the unknown function $\boldsymbol{f}$ are available. The neural network architecture is comprised of an input layer with $m_0 \coloneqq d+1$ neurons corresponding to each input, an output layer with $m_{L+1} \coloneqq m$ neurons corresponding to each output, and $L$ hidden layers in between with $m_{\ell}, \ell = 1,\dots,L$ neurons each. An example of a neural network with 2 hidden layers is depicted in Figure~\ref{fig:NNarch}.

The output layer ($\ell = L+1$) and each of the hidden layers is associated with a function $\mathcal{H}_{\ell}: \mathbb{R}^{m_{\ell - 1}} \rightarrow \mathbb{R}^{m_{\ell}}, \ell = 1,\dots, L+1$  such that the approximation $\widetilde{\boldsymbol{f}}$ is a composition of these functions, i.e.
\begin{align*}
\widetilde{\boldsymbol{f}}(\boldsymbol{x},t) = \mathcal{H}_{L+1} \circ \cdots \circ \mathcal{H}_1 (\boldsymbol{x},t).
\end{align*} 
In particular, $\mathcal{H}_{\ell}$ is a possibly nonlinear transformation of an affine function expressed as $\mathcal{H}_{\ell}(\boldsymbol{z}) = \sigma_{\ell}(\boldsymbol{W}^{\ell} \boldsymbol{z} + \boldsymbol{b}^{\ell})$ in which $\boldsymbol{z} \in \mathbb{R}^{m_{\ell - 1}}$, $\sigma_{\ell}$ is an activation function that is applied to each component of its input argument, $\boldsymbol{W}^{\ell} \in \mathbb{R}^{m_{\ell} \times m_{\ell-1}}$ is the weight matrix, and $\boldsymbol{b}^{\ell} \in \mathbb{R}^{m_{\ell}}$ is a vector of biases. If the $(i,j)$-entry of $\boldsymbol{W}^{\ell}$ is 0, this signifies that there is no edge connecting the $j$-th neuron of the $(\ell-1)$-th layer to the $i$-th neuron of the $\ell$-th layer.

In the above formulation, the number of hidden layers, the number of neurons $m_{\ell}$ per hidden layer, and the activation function $\sigma_{\ell}$ have to be specified beforehand. Commonly used activation functions include the sigmoid $\sigma_{\ell}(z) = \frac{1}{1 + e^{-z}}$, the ReLU $\sigma_{\ell}(z) = \max(z,0)$, and the hyperbolic tangent $\sigma_{\ell}(z) = \tanh z$, the latter being used in our simulations. The weight matrices $\boldsymbol{W}^{\ell}$ and the bias vectors $\boldsymbol{b}^{\ell}$ are then estimated by minimizing a loss function $\mathcal{L}$ which measures the discrepancy between the available data and the prediction via $\widetilde{\boldsymbol{f}}(\boldsymbol{x},t)$. One such loss function is the mean squared error (MSE) given by $\mathcal{L} = \sum_{i=1}^N \|\widetilde{\boldsymbol{f}}(\boldsymbol{x}_i,t_i) - \boldsymbol{f}(\boldsymbol{x}_i,t_i)\|_2^2$. The loss is minimized via gradient descent wherein the gradients of $\mathcal{L}$ with respect to $\boldsymbol{W}^{\ell},\boldsymbol{b}^{\ell}$ are efficiently calculated through backpropagation.

We refer the reader to \cite{book:GoodfellowBC2016} for further details on choosing the activation and the loss functions, the various optimization algorithms for minimizing the loss, approaches to initializing the parameters, etc.

\begin{figure}[h!]
\tikzset{%
  every neuron/.style={
    circle,
    draw,
    minimum size=1cm
  },
  neuron missing/.style={
    draw=none, 
    scale=4,
    text height=0.333cm,
    execute at begin node=\color{black}$\vdots$
  },
}
\begin{center}
\begin{tikzpicture}[x=1.5cm, y=1.5cm, >=stealth]

\foreach \m/\l [count=\y] in {1,missing,2,3}
  \node [every neuron/.try, neuron \m/.try] (input-\m) at (0,2-\y) {};

\foreach \m [count=\y] in {1,2,missing,3,4}
  \node [every neuron/.try, neuron \m/.try ] (hidden1-\m) at (2,2-\y*0.85) {};
  
 \foreach \m [count=\y] in {1,2,missing,3,4}
  \node [every neuron/.try, neuron \m/.try ] (hidden2-\m) at (4,2-\y*0.85) {};

\foreach \m [count=\y] in {1,missing,2}
  \node [every neuron/.try, neuron \m/.try ] (output-\m) at (6,1.5-\y) {};

\foreach \l [count=\i] in {1,d}
  \draw [<-] (input-\i) -- ++(-1,0)
    node [above, midway] {$x_\l$};

 \draw [<-] (input-3) -- ++(-1,0)
    node [above, midway] {$t$};

%

\foreach \l [count=\i] in {1,m}
  \draw [->] (output-\i) -- ++(1,0)
    node [above = 2mm, midway] {$f_\l(\boldsymbol{x},t)$};

\foreach \i in {1,...,3}
  \foreach \j in {1,...,4}
    \draw [->] (input-\i) -- (hidden1-\j);

\foreach \i in {1,...,4}
  \foreach \j in {1,...,4}
    \draw [->] (hidden1-\i) -- (hidden2-\j);

\foreach \i in {1,...,4}
  \foreach \j in {1,...,2}
    \draw [->] (hidden2-\i) -- (output-\j);

\foreach \l [count=\x from 0] in {Input, Hidden, Hidden, Ouput}
  \node [align=center, above] at (\x*2,2) {\l \\ layer};

\end{tikzpicture}
\end{center}
\caption{An example of a feedforward neural network with 2 hidden layers.} \label{fig:NNarch}
\end{figure}
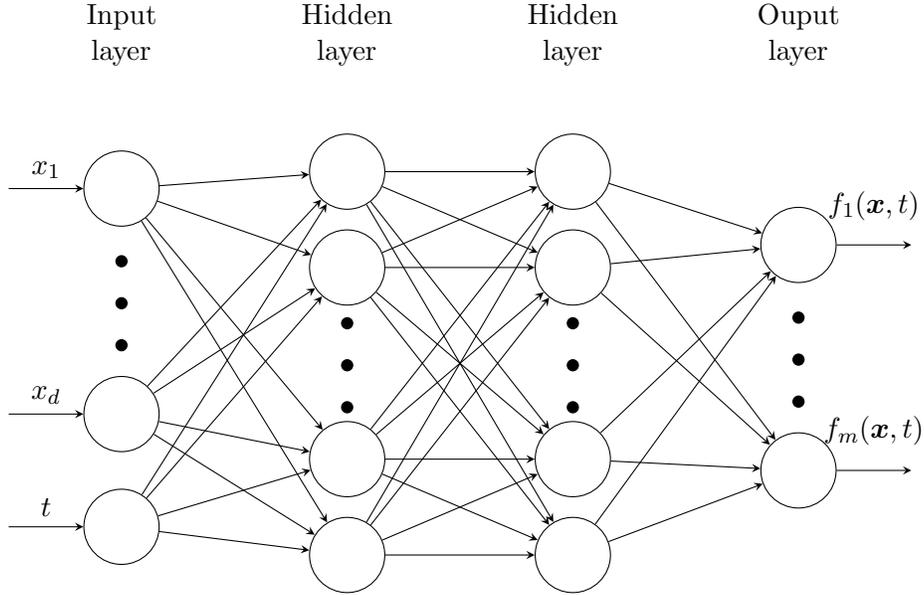

\subsection{Solving PDEs using neural networks} \label{subsec:reviewPINN}

Feedforward neural networks are universal approximators, i.e. they can sufficiently approximate any measurable function \cite{paper:Cybenko1989,paper:HornikSW1989} such as solutions to PDEs. Denote by $\boldsymbol{f}(\boldsymbol{x},t)$ the solution to the PDE given by
\begin{align} \label{eq:ModelPDE}
\mathcal{N}[\boldsymbol{f}(\boldsymbol{x},t)] & = \boldsymbol{0}, \hspace{3.5em} (\boldsymbol{x},t) \in D \times [0,T] \\
\boldsymbol{f}(\boldsymbol{x},0) & = \boldsymbol{g}(\boldsymbol{x}), \hspace{2em}  \boldsymbol{x} \in D \notag \\
\boldsymbol{f}(\boldsymbol{x},t) & = \boldsymbol{h}(\boldsymbol{x},t), \hspace{1em} (\boldsymbol{x},t) \in \partial D \times [0,T] \notag
\end{align}
for the operator $\mathcal{N}$, spatial domain $D$ and its boundary $\partial D$, end time $T$, and initial and boundary conditions $\boldsymbol{g}(\boldsymbol{x})$ and $\boldsymbol{h}(\boldsymbol{x},t)$, with $\boldsymbol{0}$ denoting a vector of zeros. Unlike the  setup traditionally utilized in machine learning wherein we desire the neural network to match available  data on the PDE solution $\{(\boldsymbol{x}_i,t_i,\boldsymbol{f}(\boldsymbol{x}_i,t_i))\}_{i=1}^N$,  \cite{paper:RaissiPK2019} pursues an approach based on physics-informed neural networks. Instead, collocation points in the domain $D \times [0,T]$ are selected to enforce the governing equation and the initial and boundary equations in \eqref{eq:ModelPDE} which are known. Let $\{(\boldsymbol{x}^{Op}_i,t^{Op}_i)\}_{i=1}^{N_{Op}}$, $\{(\boldsymbol{x}^{IC}_i,0)\}_{i=1}^{N_{IC}}$, and $\{(\boldsymbol{x}^{BC}_i,t^{BC}_i)\}_{i=1}^{N_{BC}}$ be 3 sets of collocation points corresponding to each equation in \eqref{eq:ModelPDE}. For a specified architecture for the neural network approximation $\widetilde{\boldsymbol{f}}(\boldsymbol{x},t)$ of $\boldsymbol{f}(\boldsymbol{x},t)$, we seek the weight matrices and the bias vectors that minimize the loss function
\begin{align} \label{eq:LossPINN}
\mathcal{L} = \frac{1}{N_{Op}} \sum_{i=1}^{N_{Op}} \|\mathcal{N}[\widetilde{\boldsymbol{f}}(\boldsymbol{x}_i^{Op},t_i^{Op})]\|_2^2 & + \frac{1}{N_{IC}} \sum_{i=1}^{N_{IC}} \|\widetilde{\boldsymbol{f}}(\boldsymbol{x}_i^{IC},0)-\boldsymbol{g}(\boldsymbol{x}_i^{IC})\|_2^2 \\
& + \frac{1}{N_{BC}} \sum_{i=1}^{N_{BC}} \|\widetilde{\boldsymbol{f}}(\boldsymbol{x}_i^{BC},t_{i}^{BC})-\boldsymbol{h}(\boldsymbol{x}_i^{BC},t_i^{BC})\|_2^2 \notag.
\end{align}
Computing $\mathcal{L}$ requires calculating gradients of $\widetilde{\boldsymbol{f}}$ that are present in the operator $\mathcal{N}$ which is efficiently carried out through automatic differentiation in TensorFlow \cite{paper:TensorFlow}.

If $\boldsymbol{x}$ is high-dimensional, a large number of collocation points in $D$ would be required to ensure that $\widetilde{\boldsymbol{f}}$ satisfies the constraints in \eqref{eq:ModelPDE}. In this case, \cite{paper:SirignanoS2018} proposes a meshfree method in which randomly chosen batches of collocation points in $D \times [0,T]$ are selected to enforce \eqref{eq:ModelPDE} in the process of training the neural network.

We emphasize that our objective is to investigate the feasibility of neural networks in representing the probability density function $f(\boldsymbol{x},t)$ of the state $\boldsymbol{X}(t)$ that arises as the solution of a PDE. Consequently, we adopt the physics-informed neural network approach in \cite{paper:RaissiPK2019} together with most of the architecture specifications they have used in their simulations. Our focus is not on finding the most effective choice of activation functions, number of hidden layers and hidden neurons, optimization algorithm, etc.


\section{Differential equations for the state pdf} \label{sec:CharFunFPPDE}

We detail how the pdf of the state can be represented as the solution of some differential equation. Stochastic analysis is performed in Section~\ref{subsec:CharFunPDE} to derive a differential equation for the chf of the state subject to commonly used random forcings. The Fokker-Planck equation is then reviewed in Section~\ref{subsec:FPeqn} which is a consequence of the differential equation for the chf. 

\subsection{Differential equation for the characteristic function of the state} \label{subsec:CharFunPDE}

We derive the differential equation for the characteristic function of dynamical systems subject to Gaussian and Poisson white noise defined as formal derivatives of the Brownian motion and compound Poisson processes. Examples are then presented to illustrate the application of the derived equation.

Consider the $\mathbb{R}^d$-valued diffusion process defined by
the stochastic differential equation
\begin{equation}
d\boldsymbol{X}(t)=\boldsymbol{a}\big(\boldsymbol{X}(t-)\big)\,dt+\boldsymbol{b}\big(\boldsymbol{X}(t-)\big)\,d\boldsymbol{B}(t)+\boldsymbol{c}\big(\boldsymbol{X}(t-)\big)\,d\boldsymbol{C}(t),
\quad t\geq 0,
\label{eq:diffusionEq}
\end{equation}
where the drift $\boldsymbol{a}$ is a $(d,1)$-matrix, the diffusions $\boldsymbol{b}$ and
$\boldsymbol{c}$ are $(d,m_B)$ and $(d,m_C)$-matrices, the Brownian motion $\boldsymbol{B}$
is a vector of $m_B$ independent standard Brownian motions, $\boldsymbol{C}$ is
a vector of $m_C$ independent compound Poison processes
$C_r(t)=\sum_{\nu=1}^{N_r(t)}Y_{r,\nu}$, $r=1,\ldots,m_C$, which
depend on the homogeneous Poisson processes $\{N_r\}$ of
intensities $\{\lambda_r\}$ and jump sizes $\{Y_{r,1}, Y_{r,2},\ldots\}$
that are independent copies of the random variables $\{Y_r\}$
with $E[Y_r]=0$, and $\boldsymbol{X}(t-)=\lim_{s\uparrow t}\boldsymbol{X}(s)$.  It is
assumed that the drift and diffusion coefficients are such that
Eq.~\eqref{eq:diffusionEq} admits a unique strong solution \cite[Sect.~4.7.1.1 and 4.7.2]{book:Grigoriu2002}.

\begin{theorem}
The characteristic function of $\boldsymbol{X}(t)$,
$\varphi(\boldsymbol{u},t)=E\big[\exp\big(i\,\boldsymbol{u}'\,\boldsymbol{X}(t)\big)\big]$ for
$\boldsymbol{u}\in\mathbb{R}^d$,  satisfies
\begin{align}
\frac{\partial \varphi(\boldsymbol{u},t)}{\partial t} &= i\,\sum_{k=1}^d
u_k\,E\bigg[ e^{i\,\boldsymbol{u}'\,\boldsymbol{X}(t-)}\,a_k\big(\boldsymbol{X}(t-)\big)\bigg]\nonumber\\
&-\frac{1}{2}\sum_{k,l=1}^d u_k\,u_l\,E\bigg[
\exp\big(i\,\boldsymbol{u}'\, \boldsymbol{X}(t-)\big)\,
\sum_{w=1}^{m_B}b_{kw}\big(\boldsymbol{X}(t-)\big)\,b_{lw}\big(\boldsymbol{X}(t-)\big)\bigg]\nonumber\\
&+\sum_{r=1}^{m_c}\lambda_r\,E\bigg[\int_{\mathbb{R}}e^{i\,\boldsymbol{u}'\,\big(\boldsymbol{X}(t-)
+c^{(r)}(\boldsymbol{X}(t-))\,y\big)}\,dF_r(y)-e^{i\,\boldsymbol{u}'\,\boldsymbol{X}(t-)}\bigg].
\label{eq:CharFunPDE}
\end{align}
\end{theorem}

\begin{proof}
We use the It\^o formula to develop a differential equation for
$\varphi(\boldsymbol{u},t)$.   The integral version of this
formula for a mapping $\boldsymbol{X}(t)\mapsto g\big(\boldsymbol{X}(t)\big)$ which has
continuous second order partial derivatives has the form
\begin{align}
g\big(\boldsymbol{X}(t)\big)-g\big(\boldsymbol{X}(0)\big)&=\underbrace{\sum_{k=1}^d \int_{0+}^t
\frac{\partial g\big(\boldsymbol{X}(s-)\big)}{\partial x_k}\,dX_k(s)}_{I}
+\underbrace{\frac{1}{2}\,\sum_{k,l=1}^d \int_{0+}^t \frac{\partial^2
g\big(\boldsymbol{X}(s-)\big)}{\partial x_k\,\partial
x_l}\,d\big[X_k,X_l\big]^c(s)}_{II}\nonumber\\
&+\underbrace{\sum_{0<s\leq
t}\bigg[g\big(\boldsymbol{X}(s)\big)-g\big(\boldsymbol{X}(s-)\big)-\sum_{k=1}^d
\frac{\partial g\big(\boldsymbol{X}(s-)\big)}{\partial x_k}\,\Delta
X_k(s)\bigg]}_{III},
\label{eq:ItoLemma}
\end{align}
where $[X_k,X_l]^c(s)$ denotes the continuous part of the
quadratic covariation of the components $X_k$ and $X_l$ of $\boldsymbol{X}$
and  $\Delta
X_k(s)=X_k(s)-X_k(s-)$ is the jump of component $X_k$ at time $s$
\cite[Sect.~4.6.2]{book:Grigoriu2002}. The It\^o formula holds for
semimartingales $\boldsymbol{X}(t)$ and shows that $g\big(\boldsymbol{X}(t)\big)$ is also a
semimartingale.

The above formula can be applied for the real and imaginary parts
of the mapping $\boldsymbol{X}(t)\mapsto \exp\big(i\,\boldsymbol{u}'\,\boldsymbol{X}(t)\big)$ since
$\boldsymbol{X}(t)$ in Eq.~\eqref{eq:diffusionEq} is a semimartingale and
$\exp\big(i\,\boldsymbol{u}'\,\boldsymbol{X}(t)\big)$ has continuous partial derivatives.
Since It\^o's formula is linear in $g$ and its derivative, it can
be applied directly to the complex-valued mapping $g:\boldsymbol{X}(t)\mapsto
\exp\big(i\,\boldsymbol{u}'\,\boldsymbol{X}(t)\big)$.

An overview of the derivation is as follows. We find the terms $I,II,III$ on the right side of Eq.~\eqref{eq:ItoLemma} for $g(\boldsymbol{X}(t)) = \exp\big(i\,\boldsymbol{u}'\,\boldsymbol{X}(t)\big)$, calculate their
expectations, and find the output of the It\^o formula. The differential equation for the characteristic function then results by differentiating the expectation of~\eqref{eq:ItoLemma} with respect to time.

The first term $I$ of the right side of Eq.~\eqref{eq:ItoLemma} is
\begin{align}
i\,\sum_{k=1}^d u_k\,\int_{0+}^t e^{i\,\boldsymbol{u}'\,\boldsymbol{X}(s-)}\,dX_k(s)&=
i\,\sum_{k=1}^d u_k\,\int_{0+}^t
e^{i\,\boldsymbol{u}'\,\boldsymbol{X}(s-)}\,\bigg(a_k\big(\boldsymbol{X}(s-)\big)\,ds
+\sum_{w=1}^{m_B}b_{k w}\big(\boldsymbol{X}(s-)\big)\,dB_{w}(s)
\nonumber\\
&+\sum_{v=1}^{m_C}c_{kv}\big(\boldsymbol{X}(s-)\big)\,dC_v(s)\bigg)
\nonumber
\end{align}
so that its expectation is
\begin{equation}
i\,\sum_{k=1}^d u_k\,E\bigg[\int_{0+}^t
e^{i\,\boldsymbol{u}'\,\boldsymbol{X}(s-)}\,a_k\big(\boldsymbol{X}(s-)\big)\,ds\bigg]
\label{eq:Term1RHS}
\end{equation}
since the stochastic integrals with respect to the Brownian motion
and compound Poisson processes are martingales starting at zero so
that they have zero expectations.

For the second term $II$ of the right side of Eq.~\eqref{eq:ItoLemma}, we note
that the Poisson white noise does not contribute to the processes
$[X_k,X_l]^c(s)$ so that
\begin{align*}
&d[X_k,X_l]^c(s)\nonumber\\
&=\bigg[a_k\big(\boldsymbol{X}(s-)\big)\,ds
+\sum_{w=1}^{m_B}b_{kw}\big(\boldsymbol{X}(s-)\big)\,dB_w(s),a_l\big(\boldsymbol{X}(s-)\big)\,ds
+\sum_{v=1}^{m_B}b_{lv}\big(\boldsymbol{X}(s-)\big)\,dB_v(s)\bigg]^c\nonumber\\
&=\bigg[\sum_{w=1}^{m_B}b_{kw}\big(\boldsymbol{X}(s-)\big)\,dB_w(s),
\sum_{v=1}^{m_B}b_{lv}\big(\boldsymbol{X}(s-)\big)\,dB_v(s)\bigg]^c\nonumber\\
&=\sum_{w,v=1}^{m_B}
\bigg[b_{kw}\big(\boldsymbol{X}(s-)\big)\,dB_w(s),b_{lv}\big(\boldsymbol{X}(s-)\big)\,dB_v(s)\bigg]^c
\nonumber\\
&=\sum_{w,v=1}^{m_B}b_{kw}\big(\boldsymbol{X}(s-)\big)\,
b_{lv}\big(\boldsymbol{X}(s-)\big)\,\delta_{wv}\,ds=
\sum_{w=1}^{m_B}b_{kw}\big(\boldsymbol{X}(s-)\big)\,b_{lw}\big(\boldsymbol{X}(s-)\big)\,ds
\end{align*}
by the definition and linearity of the quadratic covariation
process, the postulated independence of the components of $\boldsymbol{B}(t)$,
and properties of the Brownian motion, e.g.,
$[dB_w(t),dB_w(t)]=dt$.
The integrand of the second term of~\eqref{eq:ItoLemma} is
$i^2\,u_k\,u_l\,\exp\big(i\,\boldsymbol{u}'\,\boldsymbol{X}(s-)\big)$ so that the
expectation of the term is
\begin{equation}
-\frac{1}{2}\sum_{k,l=1}^d u_k\,u_l\,E\bigg[\int_{0+}^t
\exp\big(i\,\boldsymbol{u}'\,\boldsymbol{X}(s-)\big)\,
\sum_{w=1}^{m_B}b_{kw}\big(\boldsymbol{X}(s-)\big)\,b_{lw}\big(\boldsymbol{X}(s-)\big)\,ds\bigg].
\label{eq:Term2RHS}
\end{equation}

The continuous part of $\boldsymbol{X}(t)$ does not contribute to the last term $III$
of the right side of Eq.~\eqref{eq:ItoLemma}.  This term has three entries.
The last entry has zero mean since the jumps $\Delta
X_k(s)=X_k(s)-X_k(s-)$ are scaled versions of the jumps of $\boldsymbol{C}(s)$
which have zero expectations by assumption and so $E\big[\frac{\partial g(\boldsymbol{X}(s-))}{\partial x_k} \Delta X_k(s) \big] = E\big[\frac{\partial g(\boldsymbol{X}(s-))}{\partial x_k} \big] E[\Delta X_k(s)] = 0.$ It therefore remains to examine the first two entries, i.e.
 $h(t) = \sum_{0 < s \le t} [g(\boldsymbol{X}(s)) - g(\boldsymbol{X}(s-))] = \sum_{0 < s \le t} [\exp\big(i\,\boldsymbol{u}'\,\boldsymbol{X}(s)\big)-\exp\big(i\,\boldsymbol{u}'\,\boldsymbol{X}(s-)\big)] $. Instead of computing $E[h(t)]$ explicitly to find $\frac{d}{dt}E[h(t)]$, we consider $E[h(t+\Delta t)] - E[h(t)] = E \bigg[\sum_{t < s \le t + \Delta t} \big(e^{i\,\boldsymbol{u}'\,\boldsymbol{X}(s)} - e^{i\,\boldsymbol{u}'\,\boldsymbol{X}(s-)}\big) \bigg]$ and calculate
\begin{align}\label{eq:limit3rdTerm}
\lim_{\Delta t \rightarrow 0} \frac{E \bigg[ \displaystyle \sum_{t < s \le t + \Delta t} \big(e^{i\,\boldsymbol{u}'\,\boldsymbol{X}(s)} - e^{i\,\boldsymbol{u}'\,\boldsymbol{X}(s-)}\big) \bigg]}{\Delta t}
\end{align} 
to find the contribution of this term  to the differential equation for the characteristic function of $\boldsymbol{X}(t)$.

Consider a small time interval interval $(t,t+\Delta t]$,
$0<\Delta t\ll 1$.  We compute $E \bigg[ \displaystyle \sum_{t < s \le t + \Delta t} \big(e^{i\,\boldsymbol{u}'\,\boldsymbol{X}(s)} - e^{i\,\boldsymbol{u}'\,\boldsymbol{X}(s-)}\big) \bigg]$ by conditioning on the possible number of jumps of the Poisson processes $\{N_r\}_{r=1}^{m_C}$ in this interval. The probability of the event $\{N_r(\Delta
t)= 1,N_s(\Delta t)=0, s\not=r\}$ that component $r$ of $\boldsymbol{C}(s)$
has a jump in  $(t,t+\Delta t]$ and that the other components do not
jump in this time interval is
\begin{align} \label{eq:probOneJump}
P(N_r(\Delta
t)= 1,N_s(\Delta t)=0, s\not=r) = \lambda_r \Delta t e^{-\lambda_r \Delta t}\,\prod_{s\not=r}\exp\big(-\lambda_s\,\Delta
t\big)\simeq \lambda_r\,\Delta t
\end{align}
provided that $\lambda_s\,\Delta t\ll 1$, $s=1,\ldots,m_C$.  Note
also that the probabilities of two or more jumps of the same or of
different components of $\boldsymbol{C}(s)$ in $(t,t+\Delta t]$ are of order
$(\Delta t)^2$ so that conditioning on these events will not contribute to the differential
equation for the characteristic function of $\boldsymbol{X}(t)$ following \eqref{eq:limit3rdTerm}. This means that we only need to consider single component
jumps and add their contributions.
%

Suppose that the component $C_r(s)$ has a jump of size $Y_r$ at $T_r$ in
$(t,t+\Delta t]$. Then
\begin{align}
\sum_{t<s\le t+\Delta t}\bigg[e^{i\,\boldsymbol{u}'\,\boldsymbol{X}(s)}-e^{i\,\boldsymbol{u}'\,\boldsymbol{X}(s-)}\bigg]
&=
e^{i\,\boldsymbol{u}'\,\boldsymbol{X}(T_r)}-e^{i\,\boldsymbol{u}'\,\boldsymbol{X}(T_r-)}\nonumber\\
&=e^{i\,\boldsymbol{u}'\,\big(\boldsymbol{X}(T_r-)
+c^{(r)}(\boldsymbol{X}(T_r-))\,Y_r\big)}-e^{i\,\boldsymbol{u}'\,\boldsymbol{X}(T_r-)},
\nonumber
\end{align}
where $c^{(r)}$ is the $r$th column of the $(d,m_C)$ diffusion
matrix $\boldsymbol{c}$. It follows that 
\begin{align} \label{eq:condExpect}
& E \bigg[ \displaystyle \sum_{t < s \le t + \Delta t} \big(e^{i\, \boldsymbol{u}'\, \boldsymbol{X}(s)} - e^{i\,\boldsymbol{u}'\, \boldsymbol{X}(s-)}\big)  \, \bigg |\, N_r(\Delta t) = 1, N_s(\Delta t) = 0, s\neq r \bigg] \notag \\
&=  E \bigg[ e^{i\,\boldsymbol{u}'\,\big(\boldsymbol{X}(T_r-)
+c^{(r)}(\boldsymbol{X}(T_r-))\,Y_r\big)}-e^{i\,\boldsymbol{u}'\,\boldsymbol{X}(T_r-)} \bigg] \notag \\
& =
E\bigg[\int_{\mathbb{R}}e^{i\,\boldsymbol{u}'\,\big(\boldsymbol{X}(T_r-)
+c^{(r)}(\boldsymbol{X}(T_r-))\,y\big)}\,dF_r(y)-e^{i\,\boldsymbol{u}'\,\boldsymbol{X}(T_r-)}\bigg] 
\end{align}
where the last equality holds
since the jump $Y_r$ of $C_r$ at time $T_r$ is independent of $\boldsymbol{X}(T_r-)$. In the third line of \eqref{eq:condExpect}, $F_r$ denotes the distribution of $Y_r$ and the expectation
refers to $\boldsymbol{X}(T_r-)$.  
Since $P(T_r \le s \, | \, N_r(\Delta t) = 1) = s/\Delta t$, i.e. the jump times of $C_r(t)$ in $(t,t+\Delta t]$
are uniformly distributed,
we can
replace $T_r$ with a random number in $(t,t+\Delta t]$ so that \eqref{eq:condExpect} becomes
%
\begin{align} \label{eq:ApproxToCondExpect}
&E\bigg[\int_t^{t+\Delta t}\,\frac{ds}{\Delta
t}\,\int_{\mathbb{R}}e^{i\,\boldsymbol{u}'\,\big( \boldsymbol{X}(s-)
+c^{(r)}(\boldsymbol{X}(s-))\,y\big)}\,dF_r(y)-e^{i\,\boldsymbol{u}'\,\boldsymbol{X}(s-)}\bigg]\nonumber\\
&\simeq E\bigg[\int_{\mathbb{R}}e^{i\,\boldsymbol{u}'\,\big( \boldsymbol{X}(t-)
+c^{(r)}(\boldsymbol{X}(t-))\,y\big)}\,dF_r(y)-e^{i\,\boldsymbol{u}'\,\boldsymbol{X}(t-)}\bigg].
\end{align}
%
%
We therefore have that
\begin{align}
& E \bigg[ \displaystyle \sum_{t < s \le t + \Delta t} \big(e^{i\,\boldsymbol{u}'\,\boldsymbol{X}(s)} - e^{i\,\boldsymbol{u}'\,\boldsymbol{X}(s-)}\big) \bigg] \notag \\
& \simeq \sum_{r=1}^{m_c} E \bigg[ \displaystyle \sum_{t < s \le t + \Delta t} \big(e^{i\,\boldsymbol{u}'\,\boldsymbol{X}(s)} - e^{i\,\boldsymbol{u}'\,\boldsymbol{X}(s-)}\big)  \, \bigg |\, N_r(\Delta t) = 1, N_s(\Delta t) = 0, s\neq r \bigg] \Delta t \lambda_r + O((\Delta t)^2) \notag \\ 
& =
\Delta
t\,\sum_{r=1}^{m_c}\lambda_r\,E\bigg[\int_{\mathbb{R}}e^{i\,\boldsymbol{u}'\,\big(\boldsymbol{X}(t-)
+c^{(r)}(\boldsymbol{X}(t-))\,y\big)}\,dF_r(y)-e^{i\,\boldsymbol{u}'\,\boldsymbol{X}(t-)}\bigg] + O((\Delta t)^2)
\label{eq:Term3RHS}
\end{align}
following \eqref{eq:probOneJump} and~\eqref{eq:ApproxToCondExpect}.

To conclude the derivation, we now apply the expectation operator to~\eqref{eq:ItoLemma}. We subsequently differentiate with respect
to time the left side of~\eqref{eq:ItoLemma} and the first two terms $I,II$ on
the right side of this equation given by~\eqref{eq:Term1RHS} and \eqref{eq:Term2RHS}. The time derivative of the third term $III$ is then computed by simplifying \eqref{eq:limit3rdTerm} using \eqref{eq:Term3RHS}.
  This yields~\eqref{eq:CharFunPDE} as a result.
\end{proof}

We remark that \eqref{eq:CharFunPDE} is not a differential equation for the characteristic
function of $\boldsymbol{X}(t)$ since the drift and diffusion coefficients are
arbitrary functions.  It becomes a differential equation for $\boldsymbol{X}(t)$
if the drift and diffusion coefficients are polynomials of $\boldsymbol{X}(t)$
and the diffusion matrix $\boldsymbol{c}$ has a particular structure. For example, if $\boldsymbol{c}$ in \eqref{eq:diffusionEq} does not depend on $\boldsymbol{X}(t)$ so that the Poisson white noise is additive, \eqref{eq:CharFunPDE} results in a PDE. If $\boldsymbol{c}$ is linear in $\boldsymbol{X}(t)$ so that the Poisson white noise is multiplicative, \eqref{eq:CharFunPDE} becomes an integro-differential equation as the subsquent examples demonstrate.

As $\varphi(\boldsymbol{u},t) = \int_{\mathbb{R}^d} e^{i\boldsymbol{u}' \boldsymbol{x}} f(\boldsymbol{x},t) \,d \boldsymbol{x}$, the existence of an integrable function $g(\boldsymbol{x})$ such that $|\frac{\partial}{\partial t} f(\boldsymbol{x},t)| \le g(\boldsymbol{x})$ for $(\boldsymbol{x},t) \in D \times [0,T]$ guarantees the existence of $\frac{\partial}{\partial t} \varphi(\boldsymbol{u},t)$ by the dominated convergence theorem. In addition, a solution to~\eqref{eq:CharFunPDE} exists as the characteristic function of a random vector can always be computed. Using facts from probability theory, and assuming that $\boldsymbol{X}(t)$ has a density and finite moments of order $q$, $\varphi(\boldsymbol{u},t)$ satisfies the following conditions $\forall t$, cf. \cite[p. 480]{book:Grigoriu2002}:
\begin{itemize}
\item $\varphi(\boldsymbol{0},t) = 1$ where $\boldsymbol{0} \in \mathbb{R}^d$,
\item $|\varphi(\boldsymbol{u},t)| \le 1$,
\item $\varphi(\boldsymbol{u},t) \rightarrow 0$ and $\frac{\partial^q \varphi (\boldsymbol{u},t)}{\partial u_1^{q_1} \cdots \partial u_d^{q_d}} \rightarrow 0$ as $\|\boldsymbol{u}\| \rightarrow \infty$, where $\boldsymbol{u} = (u_1,\dots,u_d)$ and $q_i \in \{0\} \cup \mathbb{Z}^+$ such that $\sum_{i=1}^d q_i = q$.
\end{itemize}
Since the pdf and the chf of $\boldsymbol{X}(t)$ are Fourier pairs, the pdf $f(\boldsymbol{x},t)$ can be obtained via
\begin{align*}
f(\boldsymbol{x},t) = \frac{1}{(2\pi)^d} \int_{\mathbb{R}^d} e^{-i \boldsymbol{u}' \boldsymbol{x}} \, \varphi(\boldsymbol{u},t) \,d \boldsymbol{u}.
\end{align*}

We now demonstrate the application of \eqref{eq:CharFunPDE} to commonly studied diffusion processes. In the following calculations, we use the fact that $E\big[\exp\big(i\,\boldsymbol{u}'\, \boldsymbol{X}(t-)\big)\big] = E\big[\exp\big(i\, \boldsymbol{u}'\,\boldsymbol{X}(t)\big)\big]$. To see this, observe that $\boldsymbol{X}(t)$ and $\boldsymbol{X}(t-)$ differ on the event $\{N_r(\Delta t) \ge 1\}$ with $P(N_r(\Delta t) \ge 1) \rightarrow 0$ as $\Delta t \rightarrow 0$.


\begin{example}[Verhulst model] Suppose that $X(t)$ is a real-valued
diffusion process with $d=1$, $m_B=m_C=1$, $a(x)=\rho\,x-x^2$,
$b(x)=x$, and $c(x)=x$ and let the jumps of $C(t)$ be distributed according to $F$.
\end{example}
%
The three terms on the right side of \eqref{eq:CharFunPDE} are
\begin{align}
&i\,u\,E\big[e^{i\,u\,X(t-)}\,\big(\rho\,X(t-)-X(t-)^2)\big)\big]=
\rho\,u\frac{\partial \varphi(u,t)}{\partial
u}+i\,u\,\frac{\partial^2 \varphi(u,t)}{\partial u^2},
\nonumber\\
&-\frac{u^2}{2}\,E\big[e^{i\,u\,X(t-)}\,X(t-)^2\big]=\frac{u^2}{2}\,\frac{\partial^2
\varphi(u,t)}{\partial u^2}, \quad
{\rm and}\nonumber\\
&\lambda\,\bigg(E\bigg[ \int_{\mathbb{R}} e^{i\,u\,X(t-)\,(1+y)} \,dF(y) \bigg]-E\big[e^{i\,u\,X(t-)}\big]\bigg)
=\lambda\bigg(\int\varphi\big(u\,(1+y),t\big)\,dF(y)-\varphi(u,t)\bigg)
\nonumber
\end{align}
since $\varphi(u,t)=E\big[\exp\big(i\,u\,X(t)\big)\big]$ and
$\partial^r \varphi(u,t)/\partial
u^r=E\big[\big(i\,u\,X(t)\big)^r\,\exp\big(i\,u\,X(t)\big)\big]$.
%

The expressions of these terms and~\eqref{eq:CharFunPDE} give the
following integro-differential equation for the characteristic
function of the state $X(t)$ of the Verhulst model
\begin{equation}
\frac{\partial \varphi(u,t)}{\partial t}= \rho\,u\frac{\partial
\varphi(u,t)}{\partial u}+i\,u\,\frac{\partial^2
\varphi(u,t)}{\partial u^2}+\frac{u^2}{2}\,\frac{\partial^2
\varphi(u,t)}{\partial
u^2}+\lambda\bigg(\int\varphi\big(u\,(1+y)\big)\,dF(y)-\varphi(u,t)\bigg).
\label{eq:VerhulstGeneral}
\end{equation}
%

%

\begin{example}[Duffing model]
Let $Y(t)$ be the displacement of
a Duffing oscillator subjected to Gaussian and Poisson white noise
processes whose jumps are distributed according to $F$ so that the bivariate process $\boldsymbol{X}(t)$ with components
$X_1(t)=Y(t)$ and $X_2(t)=\dot{Y}(t)$ satisfies the It\^o
differential equation
\begin{equation}
\left\{
\begin{array}{ll}
dX_1(t)  &=X_2(t-)\,dt,\\
dX_2(t)  &=-\nu^2\,\big(X_1(t-)+\alpha\,X_1(t-)^3\big)\,dt
   -2\,\zeta\,\nu\,X_2(t-)\,dt+b\,dB(t)+c\,dC(t),
\end{array}
\right.
\label{eq:DuffingGeneralSDE}
\end{equation}
where $\zeta \in (0,1)$ is the damping ratio, $\nu$ denotes
the initial frequency, and $\alpha, b, c$ are real constants.
\end{example}
 The
differential equation of~\eqref{eq:CharFunPDE} for $d=2$, $u = (u_1,u_2)$, $m_B=m_C=1$, $\boldsymbol{a}(x)=\begin{bmatrix}
a_1(x) \\ a_2(x)
\end{bmatrix} = \begin{bmatrix}
x_2 \\ -\nu^2\,(x_1+\alpha\,x_1^3)-2\,\zeta\,\nu\,x_2
\end{bmatrix},
$
$\boldsymbol{b}=\begin{bmatrix}
0 \\ b 
\end{bmatrix}
$,
$\boldsymbol{c}=\begin{bmatrix}
0 \\ c 
\end{bmatrix}
$
gives
\begin{align}
\frac{\partial \varphi(\boldsymbol{u},t)}{\partial
t}&=i\,\bigg(u_1\,E\big[e^{i\,\boldsymbol{u}'\,\boldsymbol{X}(t-)}\, X_2(t-)\big]\nonumber\\
&+u_2\,E\big[e^{i\,\boldsymbol{u}'\,\boldsymbol{X}(t-)}\,\big(-\nu^2\,(X_1(t-)+\alpha\,X_1(t-)^3)
-2\,\zeta\,\nu\,X_2(t-)\big)\big]\nonumber\\
&-\frac{1}{2}\,u_2^2\,b^2\,E\big[e^{i\,\boldsymbol{u}'\,\boldsymbol{X}(t-)}\big]
+\lambda\,\bigg(\int_{\mathbb{R}}E \big[ e^{i\,\big(u_1\,X_1(t-)+u_2\,X_2(t-) + u_2 \,c\,y \big)} \big]
\,dF(y) -E\big[e^{i\,\boldsymbol{u}'\,\boldsymbol{X}(t-)}\big]\bigg).
\nonumber
\end{align}
Since 
\begin{align*}
\int_{\mathbb{R}}E \big[ e^{i\,\big(u_1\,X_1(t-)+u_2\,X_2(t-) + u_2 \,c\,y \big)}  \big]
\,dF(y) & = E \bigg[ e^{i\,\big(u_1\,X_1(t-)+u_2 \,X_2(t-)\big)} \bigg] \int_{\mathbb{R}} e^{i u_2 cy } \,dF(y) \\
& = \varphi(\boldsymbol{u},t) \phi(cu_2)
\end{align*}
where $\phi$ is the characteristic function of the jumps, the above simplifies to 
\begin{align}
\frac{\partial \varphi(\boldsymbol{u},t)}{\partial t}&=u_1\,\frac{\partial
\varphi(\boldsymbol{u},t)}{\partial u_2}-\nu^2\,u_2\,\frac{\partial
\varphi(\boldsymbol{u},t)}{\partial u_1}+\alpha\,\nu^2\,u_2\,\frac{\partial^3
\varphi(\boldsymbol{u},t)}{\partial u_1^3}-2\,\zeta\,\nu\,u_2\,\frac{\partial
\varphi(\boldsymbol{u},t)}{\partial u_2}\nonumber\\
&-\frac{b^2\,u_2^2}{2}\,\varphi(\boldsymbol{u},t)
+\lambda \varphi(\boldsymbol{u},t)  \,[\phi (cu_2)-1].
%
\label{eq:DuffingCHFGeneralPDE}
\end{align}

\subsection{Fokker-Planck equation} \label{subsec:FPeqn}

If $\boldsymbol{c}= \boldsymbol{0}_{d \times m_C}$ (a $d \times m_C$ matrix of zeros) in \eqref{eq:diffusionEq}, i.e. the Poisson white noise is absent, the Fokker-Planck equation can be recovered \cite[p. 482]{book:Grigoriu2002} by applying the Fourier transform to \eqref{eq:CharFunPDE}. The pdf of $\boldsymbol{X}(t)$ satisfies the PDE
\begin{align} \label{eq:FPeqn}
\frac{\partial f(\boldsymbol{x},t)}{\partial t} = -\sum_{i=1}^d \frac{\partial}{\partial x_i} [a_i(\boldsymbol{x}) \, f(\boldsymbol{x},t)] + \frac{1}{2}\sum_{i,j=1}^d \frac{\partial^2}{\partial x_i \partial x_j} \bigl [ (\boldsymbol{b}(\boldsymbol{x}) \boldsymbol{b}(\boldsymbol{x})')_{ij} \, f(\boldsymbol{x},t) \bigr]
\end{align}
for $(\boldsymbol{x},t) \in \mathbb{R}^d \times [0,T]$  subject to the constraint
\begin{align} \label{eq:FPconstraint}
\int_{\mathbb{R}^d} f(\boldsymbol{x},t) \, d \boldsymbol{x} = 1, \,\,\, t \in [0,T]
\end{align}
and the boundary conditions
\begin{align*}
& \lim_{|x_i| \rightarrow \infty} a_i(\boldsymbol{x}) f(\boldsymbol{x},t) = 0, \\
& \lim_{|x_i| \rightarrow \infty} (\boldsymbol{b}(\boldsymbol{x})\boldsymbol{b}(\boldsymbol{x})')_{ij} \, f(\boldsymbol{x},t) = 0,\\
& \lim_{|x_i| \rightarrow \infty} \frac{\partial}{\partial x_i} [(\boldsymbol{b}(\boldsymbol{x})\boldsymbol{b}(\boldsymbol{x})')_{ij} \, f(\boldsymbol{x},t)] = 0,  
\end{align*}
for $i,j=1,\dots,d$ where $a_i(x)$ is the $i$-th row of $\boldsymbol{a}(x)$ while $(\boldsymbol{b}(\boldsymbol{x}) \boldsymbol{b}(\boldsymbol{x})')_{ij}$ is the $(i,j)$-th component of this $d \times d$ matrix. The Fokker-Planck equation is a parabolic PDE whose maximum order of partial derivatives is at most 2 unlike the PDE for the chf, e.g.~\eqref{eq:VerhulstGeneral},~\eqref{eq:DuffingCHFGeneralPDE}.

If the SDE is driven by Poisson white noise, an extended Fokker-Planck equation can be derived under special cases~\cite{paper:Grigoriu2004,book:Grigoriu2002}. To illustrate, if $X(t) \in \mathbb{R}$ satisfies the SDE
\begin{align*}
d X(t) = -\rho X(t-) \,dt + dC(t), \,\,\, t \ge 0
\end{align*}
where $\rho > 0$, $C(t) = \sum_{w=1}^{N(t)} Y_w$ is a compound Poisson process that depends on the Poisson process $N(t)$ with intensity $\lambda$ and $\{Y_w\}$ being independent copies of $Y$, applying \eqref{eq:CharFunPDE}, the PDE for the characteristic function of $X(t)$ is
\begin{align*}
\frac{\partial \varphi (u,t)}{\partial t} = -\rho u \frac{\partial \varphi (u,t)}{\partial u} + \lambda \left (E[e^{iuY}]-1 \right) \varphi(u,t).
\end{align*}
Following \cite[p. 484]{book:Grigoriu2002}, the Fourier transform of this PDE yields 
\begin{align} \label{eq:FPPWN}
\frac{\partial f(x,t)}{\partial t} = \rho \frac{\partial}{\partial x} (xf(x,t)) + \lambda \sum_{k=1}^{\infty} \frac{(-1)^k E[Y^k]}{k!} \frac{\partial^k f(x,t)}{\partial x^k}
\end{align}
under some conditions which include that $Y$ has finite moments of any order.

We finally remark that for some systems driven by L\'evy white noise with $\alpha$-stable random variable increments, a PDE for the state chf may be formulated while it may not be possible to derive a Fokker-Planck type PDE for the pdf if $\alpha \in (0,2)$ \cite[Ex. 7.34]{book:Grigoriu2002}. 

\section{Neural network-based representation for the state pdf} \label{sec:NNapproxPDE}

We illustrate how a neural network can be trained to approximate solutions to the differential equations introduced in Section~\ref{sec:CharFunFPPDE}. Section~\ref{subsec:NNtrainingforPDFCHF} discusses the formulation of the loss function of the neural network to incorporate the constraints of these 2 types of differential equations. The physics-informed neural network approach is  then applied  to compute the pdf and the chf of the Brownian motion for which analytical solutions are available. Finally, a comparison is made in Section~\ref{subsec:CompCHFvsPDF} which elaborates on the advantages and disadvantages of solving each differential equation using neural networks from an analytical and numerical perspective. 

\subsection{Training neural networks to approximate the pdf and the chf} \label{subsec:NNtrainingforPDFCHF}


We discuss how neural networks can be trained to solve the differential equations introduced in Sections~\ref{subsec:CharFunPDE} and~\ref{subsec:FPeqn} and comment on their approximation quality.

An important constraint for the Fokker-Planck equation \eqref{eq:FPeqn} is that the pdf must integrate to 1 at all times. It was observed in \cite{paper:AlAradiCNJS2018,paper:AlAradiCNJS2020} that failure to enforce this constraint in the neural network training resulted in a pdf that did not match the analytical solution. It was therefore proposed to represent $f(x,t)$ by
\begin{align} \label{eq:FPtransformation}
f(\boldsymbol{x},t) = \frac{e^{-v(\boldsymbol{x},t)}}{\int_{\mathbb{R}^d} e^{-v(\boldsymbol{x},t)} \, d \boldsymbol{x}}
\end{align} 
for some function $v(\boldsymbol{x},t)$ so that $f(x,t)$ is positive and integrates to 1. Hence, if $\mathcal{N}$ represents the operator of the Fokker-Planck PDE, instead of solving for $f(\boldsymbol{x},t)$ such that $\mathcal{N}[f(\boldsymbol{x},t)] = 0$ and \eqref{eq:FPconstraint} holds, we solve for $v(\boldsymbol{x},t)$ such that $\mathcal{M}[v(\boldsymbol{x},t)] = 0$ for some operator $\mathcal{M}$. Our objective is therefore to find a neural network approximation $\widetilde{v}(\boldsymbol{x},t)$ for which $\mathcal{M}[\widetilde{v}(\boldsymbol{x},t)] = 0$ is satisfied so that an approximation $\widetilde{f}(\boldsymbol{x},t)$ to the state pdf $f(\boldsymbol{x},t)$ can be obtained via $\widetilde{f}(\boldsymbol{x},t) = \frac{e^{-\widetilde{v}(\boldsymbol{x},t)}}{\int_{\mathbb{R}^d} e^{-\widetilde{v}(\boldsymbol{x},t)} \, d \boldsymbol{x}}$.

Suppose that the pdf $f(\boldsymbol{x},0) = \frac{e^{-v(\boldsymbol{x},0)}}{\int_{\mathbb{R}^d} e^{-v(\boldsymbol{x},0)} \, d \boldsymbol{x}}$ of $\boldsymbol{X}(0)$ is available. For a specified neural network architecture (number of hidden layers, number of neurons per hidden layer, and type of activation function), Algorithm~\ref{alg:NN-FP} summarizes how a neural network approximation $\widetilde{v}(\boldsymbol{x},t)$ for $v(\boldsymbol{x},t)$ on $(\boldsymbol{x},t) \in \mathbb{R}^{d+1} \times [0,T]$ can be obtained. The input layer in this case consists of $d+1$ neurons while there is only 1 neuron in the output layer.
\begin{algorithm}[H]
\caption{Training neural networks to solve the Fokker-Planck equation}
\begin{algorithmic}[1]
  \STATE Truncate the spatial domain by choosing a compact $D \subset \mathbb{R}^d$ sufficiently large so that most of the probability mass of $\boldsymbol{X}(t)$ is contained in $D$
  \STATE Select $N_{Op}$ collocation points $\{(\boldsymbol{x}^{Op}_i,t^{Op}_i)\}_{i=1}^{N_{Op}} \subset D \times [0,T]$ to enforce the governing equations 
  \STATE Select $N_{IC}$ collocation points $\{(\boldsymbol{x}^{IC}_i,0)\}_{i=1}^{N_{IC}} \subset D \times 0$ to enforce the initial condition
  \STATE Solve for the neural network parameters to minimize the loss
  \begin{align} \label{eq:FP-loss}
	\mathcal{L} =   \frac{1}{N_{Op}} \sum_{i=1}^{N_{Op}} (\mathcal{M}[\widetilde{v}(\boldsymbol{x}_i^{Op},t_i^{Op})])^2 & + \frac{1}{N_{IC}} \sum_{i=1}^{N_{IC}} (\widetilde{v}(\boldsymbol{x}_i^{IC},0)-v(\boldsymbol{x}_i^{IC},0))^2 
  \end{align}
\end{algorithmic}
\label{alg:NN-FP}
\end{algorithm}
Since the Fokker-Planck equation is generally defined on an unbounded spatial domain, enforcing the boundary conditions is numerically challenging regardless of the numerical scheme employed. However, once the optimal neural network parameters are found, $\widetilde{v}(\boldsymbol{x},t)$ and its derivatives can be queried for any point $(\boldsymbol{x},t) \in D \times [0,T]$ via automatic differentiation in TensorFlow. This can be used to verify that the boundary conditions are met.

We remark that \cite{paper:XuZLZLK2020} also pursued a neural network solution to the steady-state Fokker Planck equation ($\frac{\partial f(\boldsymbol{x},t)}{\partial t} = 0$) but did not utilize the transformation \eqref{eq:FPtransformation} to incorporate the normalization constraint. Instead, if $\widetilde{f}(\boldsymbol{x})$ is the neural network approximation to the steady-state pdf, their loss function included a discretized version of $(\int_{\mathbb{R}^d} f(\boldsymbol{x}) \,d \boldsymbol{x}-1)^2$ in addition to enforcing the governing PDE and the boundary condition. This approach does not guarantee that $\widetilde{f}(x) \ge 0$ and moreover, it is not clear how many time points are needed to incorporate the constraint~\eqref{eq:FPconstraint} to the loss function for solving the time-varying Fokker Planck equation. In the remainder of this work, we adopt the approach in \cite{paper:AlAradiCNJS2018,paper:AlAradiCNJS2020}.

Unlike the probability density function, the characteristic function of a random vector is generally complex-valued. This means that the output layer of a neural network representation $\widetilde{\varphi}(\boldsymbol{u},t)$ of $\varphi(\boldsymbol{u},t)$ has 2 neurons, each corresponding to the real and imaginary parts of $\varphi(\boldsymbol{u},t)$. In some cases, it will be demonstrated in Section~\ref{sec:NumExp} that probabilistic arguments on $\boldsymbol{X}(t)$ can be made to show that $f(\boldsymbol{x},t)$ is symmetric about $\boldsymbol{x}=\boldsymbol{0}$ thereby showing that $\varphi(\boldsymbol{u},t)$ is real-valued. Assuming that a differential equation for the chf can be derived, denote its governing equation by $\mathcal{Q}[\varphi(\boldsymbol{u},t)] = 0$ for some operator $\mathcal{Q}$. For a specified architecture, Algorithm~\ref{alg:NN-CHF} elaborates how a neural network approximation can be obtained for $\varphi(\boldsymbol{u},t)$ on $(\boldsymbol{u},t) \in \mathbb{R}^{d+1} \times [0,T]$ assuming that the chf $\varphi(\boldsymbol{u},0)$ of $\boldsymbol{X}(0)$ is available. As with the Fokker-Planck equation, the input layer has $d+1$ neurons.


\begin{algorithm}[H]
\caption{Training neural networks to solve the diff. eq. for the characteristic function}
\begin{algorithmic}[1]
  \STATE Truncate the frequency domain by choosing a compact $D \subset \mathbb{R}^d$ sufficiently large so that $D$ contains most of the frequencies of the chf of $\boldsymbol{X}(t)$
  \STATE Select $N_{Op}$ collocation points $\{(\boldsymbol{u}^{Op}_i,t^{Op}_i)\}_{i=1}^{N_{Op}} \subset D \times [0,T]$ to enforce the governing equations 
  \STATE Select $N_{IC}$ collocation points $\{(\boldsymbol{u}^{IC}_i,0)\}_{i=1}^{N_{IC}} \subset D \times 0$ to enforce the initial condition
  \STATE Select $N_{0}$ collocation points $\{(\boldsymbol{0},t_i^0)\}_{i=1}^{N_{0}} \subset \boldsymbol{0} \times [0,T]$ to enforce the condition at the origin
  \STATE Solve for the neural network parameters to minimize the loss
  \begin{align} \label{eq:CHF-loss}
	\mathcal{L} =   \frac{1}{N_{Op}} \sum_{i=1}^{N_{Op}} \bigg\vert \mathcal{Q}[\widetilde{\varphi}(\boldsymbol{u}_i^{Op},t_i^{Op})]\bigg \vert^2 & + \frac{1}{N_{IC}} \sum_{i=1}^{N_{IC}} \bigg \vert\widetilde{\varphi}(\boldsymbol{u}_i^{IC},0)-\varphi(\boldsymbol{u}_i^{IC},0)\bigg \vert^2  + \frac{1}{N_{0}} \sum_{i=1}^{N_{0}} \bigg \vert \widetilde{\varphi}(\boldsymbol{0},t_i^0)-1 \bigg \vert^2
  \end{align}
\end{algorithmic}
\label{alg:NN-CHF}
\end{algorithm}
In the loss function \eqref{eq:CHF-loss}, $\vert \cdot \vert$ refers to the magnitude of a complex number since $\widetilde{\varphi}(\boldsymbol{u},t)$ is complex-valued. As in the Fokker-Planck equation, it is numerically challenging to incorporate the boundary conditions of the differential equation for the chf however, once the neural network approximation is attained, it can be queried to ensure that such conditions are met. As for the constraint that $|\varphi(\boldsymbol{u},t)| \le 1$, it may be possible to apply a transformation to $\varphi(\boldsymbol{u},t)$ to impose this condition. Our numerical experiments revealed that the resulting neural network solution does not violate this constraint.


In order to assess the quality of the neural network approximation to the PDE solution, it was demonstrated in \cite[Theorem 7.1]{paper:SirignanoS2018} that if a neural network with a single hidden layer is used to approximate the solution to a class of quasilinear parabolic PDEs,  the loss function approaches zero as the number of neurons increases. This holds if the terms in the PDE satisfy a Lipschitz condition. The proof of this result can be adapted and extended to other types of PDEs provided that similar Lipschitz-type conditions are met. Furthermore, \cite[Theorem 7.2]{paper:SirignanoS2018} shows that the neural network approximation converges to the unique PDE solution under additional assumptions on the parabolic PDE.  While this result can be applied to the Fokker-Planck equation, it is not directly applicable to the chf PDE since the order of partial derivatives of the latter can exceed 2.

Following the above discussion on training neural networks to solve the Fokker-Planck equation or the differential equation for the chf, we apply this methodology to a simple example for which the analytical solution of both the pdf (Example~\ref{ex:BMFP}) and the chf (Example~\ref{ex:BMCHF}) are available for all time. 

\begin{example} \label{ex:BMFP}
Let $X(t)$ satisfy $dX(t) = \sigma dB(t), t \in [0,1]$, with $X(0) \sim N(0,\nu)$ where $\sigma, \nu = 1$ and $B(t)$ is the Brownian motion. We use physics-informed neural networks to solve the Fokker-Planck equation and reconcile the approximation with the analytical solution.
\end{example}

The Fokker-Planck equation for $X(t)$ is given by
\begin{align*}
\frac{\partial f(x,t)}{\partial t} - \frac{1}{2} \sigma^2 \frac{\partial^2 f(x,t)}{\partial x^2} = 0
\end{align*}
whose analytical solution can be readily verified as $f(x,t) = \frac{1}{\sqrt{2\pi (\nu + \sigma^2 t)}} e^{-\frac{x^2}{2(\nu+\sigma^2 t)}}$.
Applying the transformation $f(x,t) = \frac{e^{-v(x,t)}}{\int_{-\infty}^{\infty} e^{-v(x,t)} \,dx}$ yields the following PDE for $v(x,t)$:
\begin{align}\label{eq:BMFPtransformed}
\mathcal{M}[v(x,t)] = \frac{\partial v(x,t)}{\partial t} - \frac{\sigma^2}{2} \left( \frac{\partial^2 v(x,t)}{\partial x^2} -\left(\frac{\partial v(x,t)}{\partial x} \right)^2 \right) - \frac{\int_{-\infty}^{\infty} e^{-v(x,t)} \frac{\partial v(x,t)}{\partial t} \,dx}{\int_{-\infty}^{\infty} e^{-v(x,t)} \,dx} = 0,
\end{align}
for which we seek a neural network approximation. Observe that \eqref{eq:BMFPtransformed} does not have a unique solution because if $v(x,t)$ satisfies \eqref{eq:BMFPtransformed} then so does $v(x,t) + g(t)$ for some differentiable function $g$ with $g(0) = 0$. Analytically, this does not pose an issue since we still recover $f(x,t)$ but as will be noted in the applications in Section~\ref{sec:NumExp}, it may be a source of numerical issues which was not explored in \cite{paper:AlAradiCNJS2018,paper:AlAradiCNJS2020}.

The neural network architecture we used to solve~\eqref{eq:BMFPtransformed} on the truncated spatial domain $D = [-7,7]$  is comprised of an input layer with 2 neurons, 4 hidden layers with 100 neurons each, and an output layer with a single neuron. Hyperbolic tangent activation functions were utilized in all simulations in this work. To construct the loss function \eqref{eq:FP-loss} for Example~\ref{ex:BMFP}, we generated a regular grid of $N_{Op} = 151 \times 101$ points in the input domain $[-7,7]\times [0,1]$ to enforce the governing equation while $N_{IC}=100$ points were randomly chosen among the 151 equally spaced points in $D$ to impose the initial condition. The same mesh on $D$ was used to numerically approximate the integral terms that appear in the operator \eqref{eq:BMFPtransformed}. The value of the loss function \eqref{eq:FP-loss} of the trained neural network is $7.994686 \times 10^{-7}$. Figure~\ref{fig:BMFPAnalyVSNN} examines the performance of the neural network approximation compared to the analytical solution. The left panel displays plots of $v(x,t) = \frac{x^2}{2(1+t)}$ (solid) and $\widetilde{v}(x,t) $ (dashed) for $t=1$. As can be seen, $\widetilde{v}(x,t)$ is a shifted version of $v(x,t)$ because \eqref{eq:BMFPtransformed} does not have a unique solution. This behavior was also observed for other values of $t$. However, the middle panel indicates that $f(x,t)$ (solid) and $\widetilde{f}(x,t)$ (dashed) at $t=1$ coincide once the solution is normalized because~\eqref{eq:FPtransformation} is unaffected by the vertical shift in the left panel. The right panel plots the error $\max_x |f(x,t)-\widetilde{f}(x,t)|$ thereby confirming that the neural network is able to recover the analytical solution.

\begin{figure}[h!]
		\centering
		\includegraphics[width = 0.32\textwidth]		
						{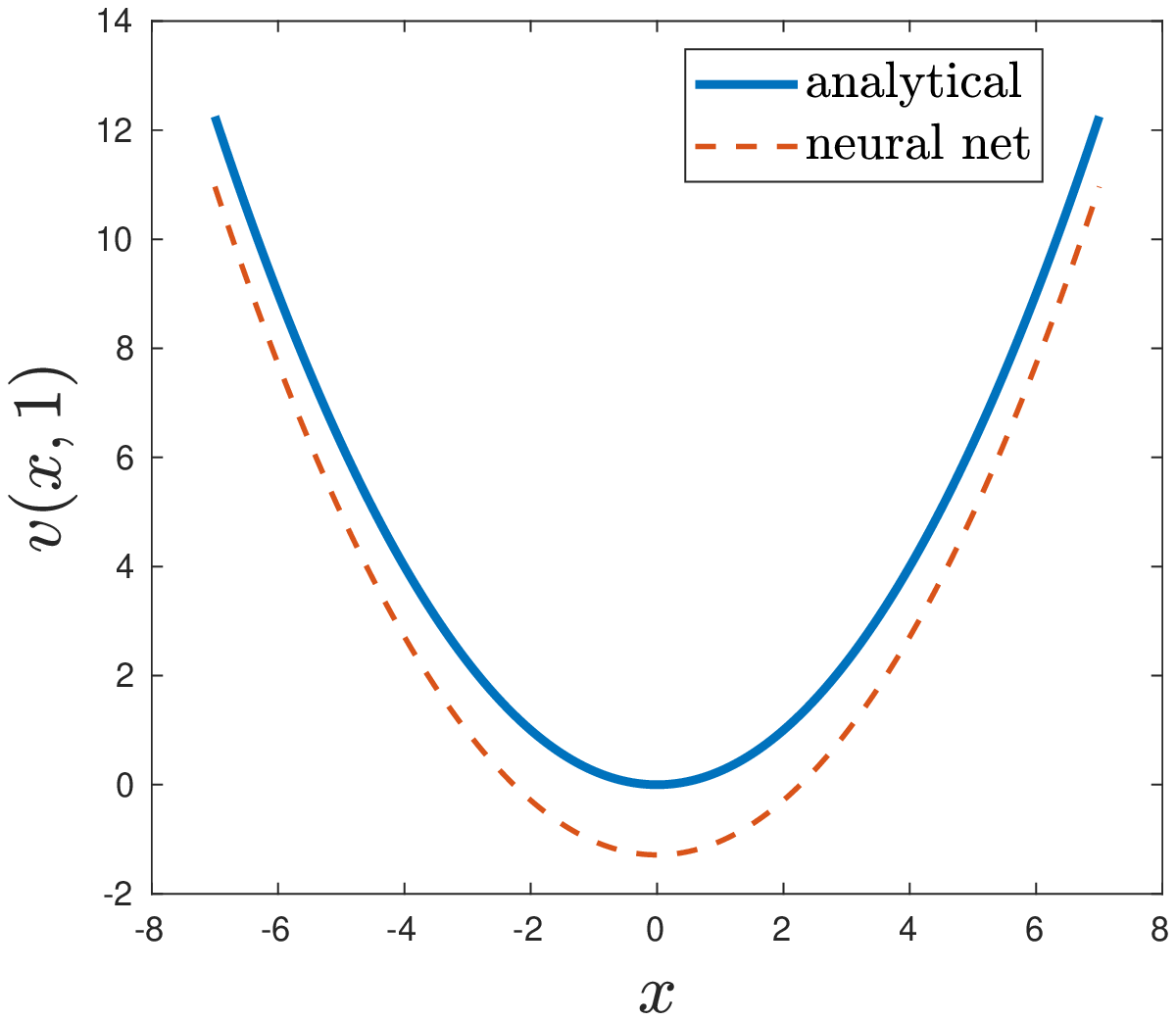}
		\includegraphics[width = 0.34\textwidth]		
						{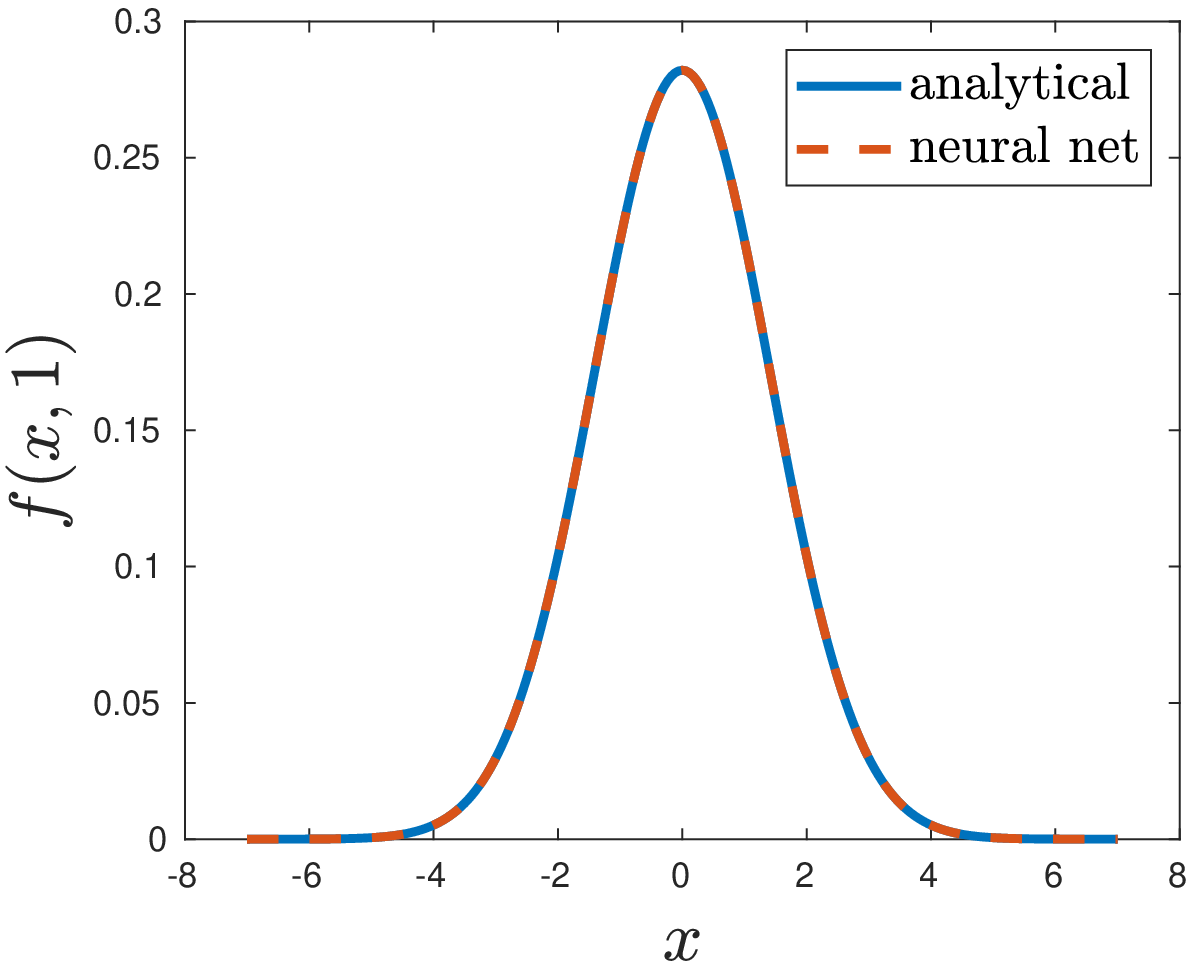}
		\includegraphics[width = 0.32\textwidth]		
						{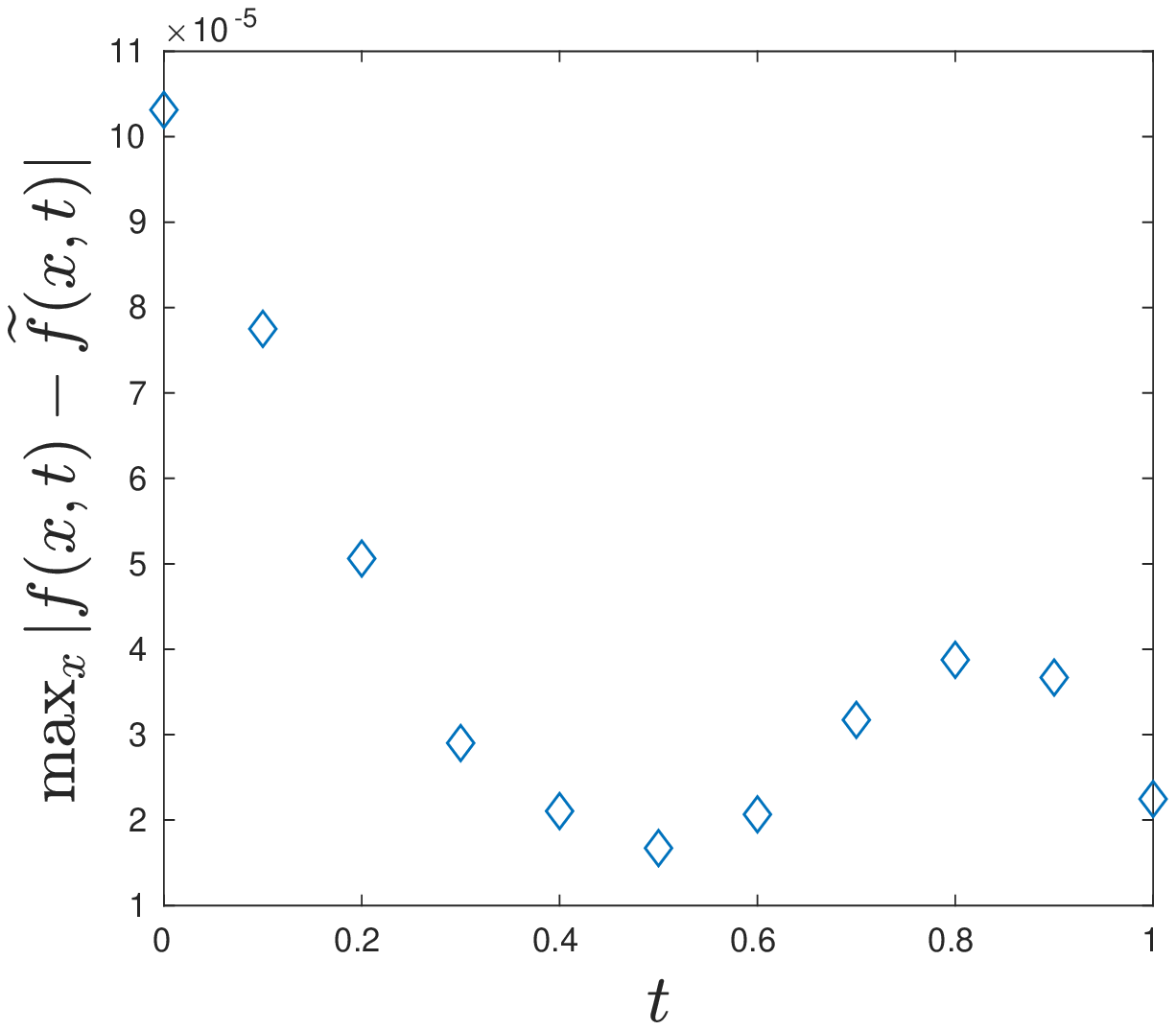}
		\caption{Comparison between the neural network approximation and the analytical solution of the Fokker-Planck equation for Example~\ref{ex:BMFP}. Left: $v(x,1)$ (solid) vs $\widetilde{v}(x,1)$ (dashed). Middle: $f(x,1)$ (solid) vs $\widetilde{f}(x,1)$ (dashed). Right: $\max_x |f(x,t) - \widetilde{f}(x,t)|$ for $t = [0,1]$ with increments of 0.1.}\label{fig:BMFPAnalyVSNN}
\end{figure}

\begin{example} \label{ex:BMCHF}
Let $X(t)$ satisfy the SDE and initial conditions in Example~\ref{ex:BMFP}. A physics-informed neural network is trained to solve the PDE for the chf of $X(t)$ which is then reconciled with the analytical solution.
\end{example}
The PDE for the chf of $X(t)$ has the form (see \eqref{eq:CharFunPDE})
\begin{align} \label{eq:BMCHFpde}
\mathcal{Q}[\varphi(u,t)] = \frac{\partial \varphi(u,t)}{\partial t} + \frac{1}{2}u^2 \sigma^2 \varphi(u,t) = 0
\end{align}
that admits the analytical solution $\varphi(u,t) = e^{-\frac{1}{2}(\nu +\sigma^2 t)u^2}$ which is the characteristic function of a Gaussian random variable with mean $0$ and variance $\nu + \sigma^2t$. Without knowledge of this analytical solution, we anticipate that the chf is real-valued since $X(t)$ is a scaled Brownian motion which is a Gaussian process with mean 0. The pdf of $X(t)$ is therefore symmetric with respect to the spatial origin.


The same neural network architecture in Example~\ref{ex:BMFP} was implemented for this example with $D = [-7,7]$ as the truncated frequency domain for $\varphi(x,t)$. To formulate the loss function \eqref{eq:CHF-loss} for Example~\ref{ex:BMCHF}, we used $N_{Op} = 15000$ latin hypercube samples \cite{paper:Stein1987} in $D \times [0,1]$, $N_{IC} = 100$ random points in $D$, and $N_0 = 100$ equally spaced points in $[0,1]$ so that the governing equation, initial condition, and condition at the origin hold, respectively.  The trained neural network attained a loss function value of $3.4449415\times 10^{-6}$. Figure~\ref{fig:BMCHFAnalyVSNN} exhibits the comparison between the neural network approximation and the known analytical soltion. The left panel highlights that $\widetilde{\varphi}(u,t)$ (dashed) matches $\varphi(u,t)$ (solid) for $t=1$ while the right panel affirms that the neural network solution is comparable to the analytical solution based on the plotted values of $\max_u |\varphi(u,t)-\widetilde{\varphi}(u,t)|$ for various time points.

\begin{figure}[h!]
		\centering
		\includegraphics[width = 0.42\textwidth]		
						{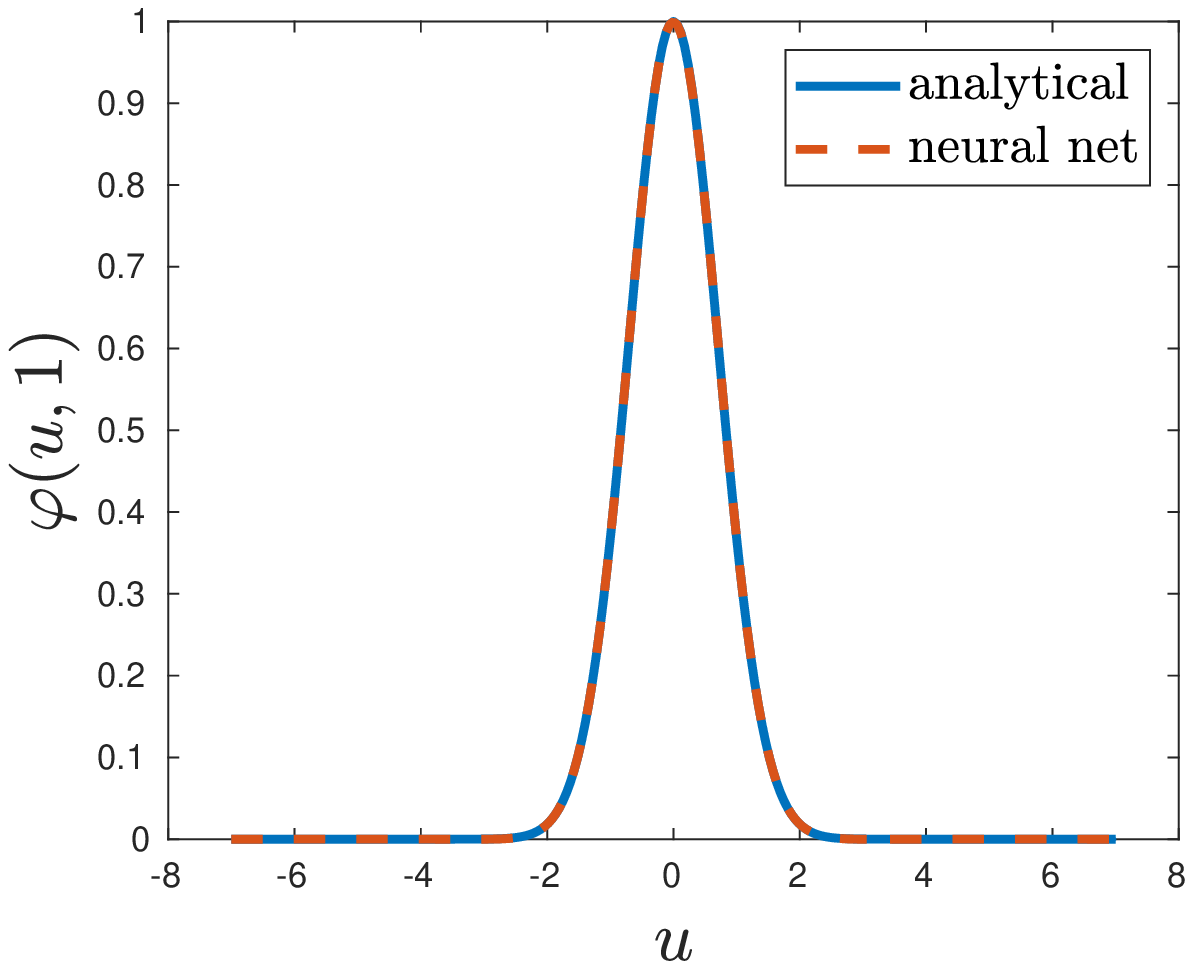} 
						\hspace{3em}
		\includegraphics[width = 0.4\textwidth]		
						{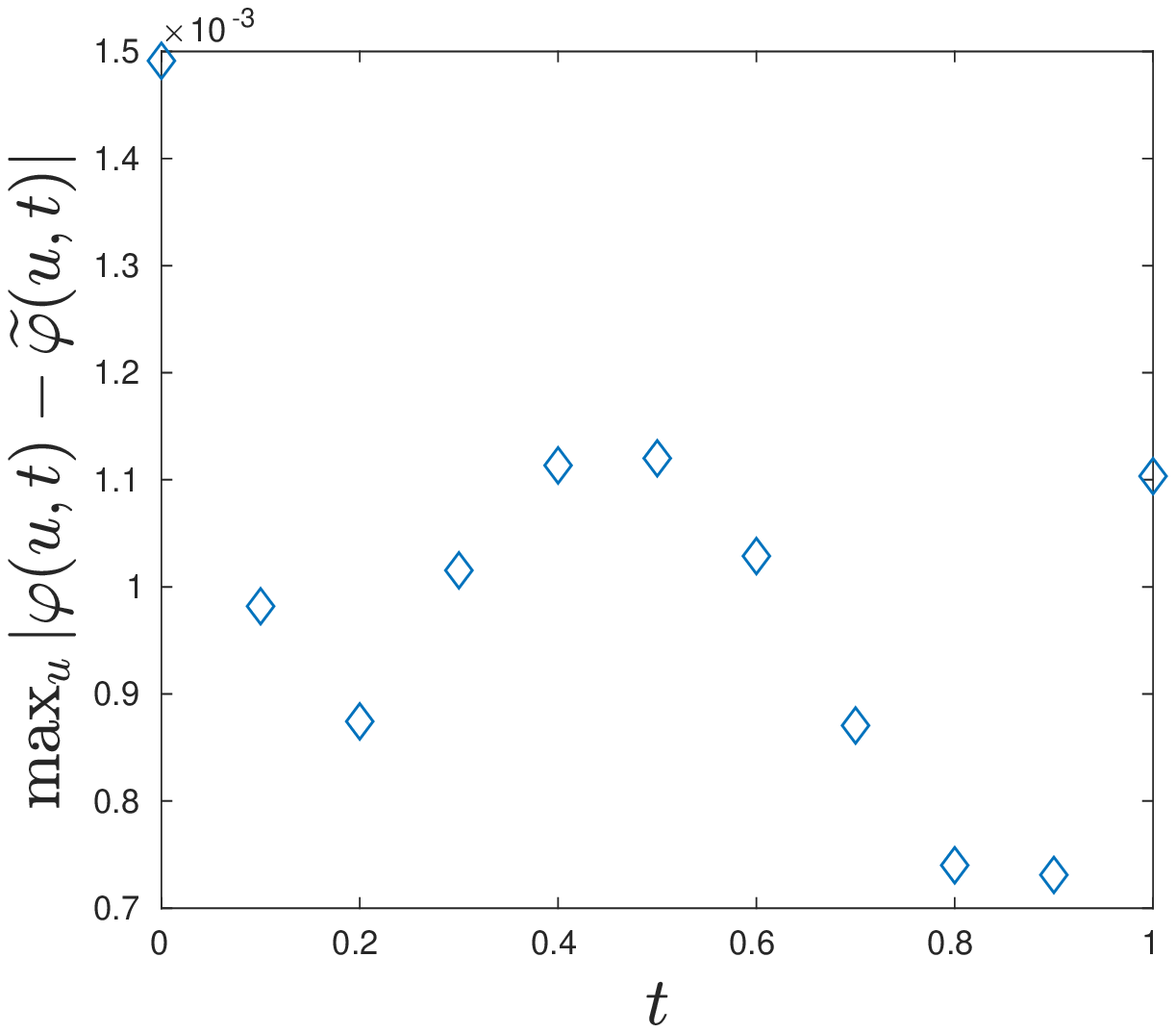}
		\caption{Comparison between the neural network approximation and the analytical solution for the chf PDE in Example~\ref{ex:BMCHF}. Left: $\varphi(u,1)$ (solid) vs $\widetilde{\varphi}(u,1)$ (dashed). Right: $\max_u |\varphi(u,t) - \widetilde{\varphi}(u,t)|$ for $t = [0,1]$ with increments of 0.1.}
		\label{fig:BMCHFAnalyVSNN}
\end{figure}

\subsection{Comparison between the differential equations for the chf and the pdf} \label{subsec:CompCHFvsPDF}

\begin{table}[h!]
\tabulinesep=6pt
\begin{tabu}{>{\cellcolor{white}\color{black}}r |X[cm] |X[cm]}
\hline
\hline
\rowcolor{white}\strut  & \color{black} Fokker-Planck PDE & \color{black} Chf diff. eq. \\
\hline
Can be derived? & available for GWN forcing but may be unavailable otherwise & polynomial drift and diffusion and special structure for jump coefficient \\
Dimension & 1 (real-valued) & 2 (complex-valued) \\
Transformation & needed to incorporate normalization constraint & none \\
Diff. eq. type & integro-differential equation if transformation applied & generally an integro-differential equation \\
Derivative order & at most 2 for GWN forcing & max degree of polynomial coefficient \\
Post-processing & none & Fourier transform \\
\hline
\end{tabu}
\caption{Comparison between the Fokker-Planck equation and the differential equation for the chf.} \label{table:FPvsCHFpde}
\end{table}

Our objective in this work is to obtain the pdf of the state $\boldsymbol{X}(t)$ that satisfies a stochastic differential equation.  
From the subsections above, two approaches have been presented to accomplish this, namely, solving the Fokker-Planck equation and solving the differential equation for the characteristic function of $\boldsymbol{X}(t)$ and computing its Fourier transform. In the following, the advantages and disadvantages of each approach are outlined, see Table~\ref{table:FPvsCHFpde} for a summary.

The main drawback of the Fokker-Planck equation is that it may not be possible to derive a PDE for the pdf of $\boldsymbol{X}(t)$ if the forcing term is not Gaussian white noise (GWN). Even if an extended Fokker-Planck equation as in \eqref{eq:FPPWN} can be deduced, it may not be favorable to solve this PDE using any numerical scheme because truncating the series expression may introduce numerical errors if the jumps have non-zero moments of any order. From a numerical perspective, enforcing the constraint \eqref{eq:FPconstraint} can be challenging. While a transformation such as \eqref{eq:FPtransformation} resolves this issue, our numerical experiments in Section~\ref{sec:NumExp} reveal that this may not be suitable in high dimensions. In particular, the non-uniqueness of the solution to the PDE for $v(\boldsymbol{x},t)$ and the absence of constraints on $v(\boldsymbol{x},t)$ imply that it is likely for $v(\boldsymbol{x},t)$ to be very negative rendering  $\int_{\mathbb{R}^d} e^{-v(\boldsymbol{x},t)} \,d \boldsymbol{x}$ impossible to be approximated numerically. The loss function $\mathcal{L}$ becomes nan as a result and the optimization algorithm is unable to continue searching for a local minimum.

However, if the Fokker-Planck equation can be derived for $\boldsymbol{X}(t)$ and if the constraint can be seamlessly incorporated, solving for the pdf in this manner is convenient since the solution to the Fokker-Planck PDE is already the quantity of interest, $f(\boldsymbol{x},t)$, which does not need to be post-processed. For systems with Gaussian white noise forcing, the maximum number of derivatives in the PDE is at most 2.

In contrast, the differential equation for the characteristic function can only be derived if the drift and diffusion coefficients are polynomials and that the diffusion matrix for the Poisson white noise forcing has a special structure. Despite this, an advantage of this approach is that there are scenarios \cite{book:Grigoriu2002} for which a differential equation for the chf can be derived while a PDE for the pdf is not available, especially when the forcing term has jumps such as the Poisson and L\'evy white noise. In addition, the constraints of this equation are more convenient to implement as they do not involve any normalization.

The disadvantages of solving the differential equation for the chf include the fact that the chf is generally complex-valued. This implies that a system of differential equations needs to be solved. From \eqref{eq:CharFunPDE}, it can be observed that the chf differential equation is an integro-differential equation in which the maximum order of the derivatives is equivalent to the highest polynomial degree of the drift and diffusion coefficient. Finally, the solution to the differential equation has to be post-processed via the Fourier transform to derive the pdf of $\boldsymbol{X}(t)$.



\section{Applications} \label{sec:NumExp}

We investigate the capabilities of physics-informed neural networks to solve the Fokker-Planck equation or the differential equation for the characteristic function that arise from various diffusion processes. 
Since the analytical solution to these equations is unavailable, the target solution is approximated via Monte Carlo simulation. The applications below serve to highlight the advantages and disadvantages outlined in Section~\ref{subsec:CompCHFvsPDF} and also demonstrate how the neural network solution can be utilized to study probabilistic phenomenon. They also result in differential equations which are different from the ones tackled in \cite{paper:RaissiPK2019,paper:XuZLZLK2020,paper:SirignanoS2018,paper:AlAradiCNJS2018,paper:AlAradiCNJS2020}.

Section~\ref{subsec:VerhulstGWN} and~\ref{subsec:VerhulstPWN} consider the 1-dimensional Verhulst model subject to Gaussian and Poisson white noise, respectively. In the former, the Fokker-Planck equation is solved to show that the neural network solution can recover the known analytical stationary density. In the latter, the differential equation for the chf is solved which represents an example of a system of integro-differential equations. Section~\ref{subsec:DuffingFPvsCHF} is concerned with reconciling the pdf obtained from the Fokker-Planck equation and the one obtained from solving the chf PDE for the Duffing oscillator subject to Gaussian white noise. Section~\ref{subsec:DuffingPWNvsBM} revisits the Duffing oscillator and illustrates how a Poisson white noise forcing with sufficiently large jump intensity yields a chf that is similar to what would be obtained under Gaussian white noise. Finally, Section~\ref{subsec:3Dexample} deals with an example of solving the chf PDE in a 3-dimensional frequency domain.

The Python scripts written for all simulations presented is readily available upon request. We also reiterate that our objective is not to seek the optimal neural network architecture nor identify the ideal number of collocation points. Rather, we examine the feasibility of a neural network approximation for the state pdf. In practice, constructing a neural network involves a training, validation, and testing phase \cite{book:GoodfellowBC2016} with each phase relying on distinct sets of collocation points. In the simulations below, the number of collocation points for the training phase can be increased to see if the loss from the training phase decreases. The testing phase can serve to prevent the occurrence of overfitting. We do not undertake testing and validation in this work since we compare the neural network approximation with the estimate produced by Monte Carlo simulation.

\subsection{Verhulst model with Gaussian white noise} \label{subsec:VerhulstGWN}

Suppose that $X(t)$ satisfies the SDE given by the Verhulst model
\begin{align*}
dX(t) = (\rho X(t) - X(t)^2)\,dt + \sigma X(t) \,dB(t), \hspace{1em}t \ge 0, 
\end{align*}
where $B(t)$ is the Brownian motion. It will be shown that the neural network representation of the pdf of $X(t)$ coincides with the analytical stationary pdf. 

By applying \eqref{eq:FPeqn}, the Fokker-Planck equation for this SDE is
\begin{align} \label{eq:FPVerhulstGWN}
\frac{\partial f(x,t)}{\partial t} = -\frac{\partial }{\partial x} ((\rho x-x^2)f(x,t)) + \frac{\sigma^2}{2} \frac{\partial^2}{\partial x^2}(x^2f(x,t)).
\end{align}
Following the calculations in \cite[p. 72]{book:Grigoriu2002}, the analytical stationary density is $f_s(x) = k x^{2(\rho/\sigma^2 - 1)} e^{-2x/\sigma^2}$ for  $x > 0$ where $k$ is the normalizing constant. For the neural network implementation, we apply the transformation \eqref{eq:FPtransformation} so that we solve for $v(x,t)$ satisfying
\begin{align} \label{eq:FPVerhulstGWNTransf}
\mathcal{M}[v(x,t)] = v_t(x,t) + (-\rho + 2x + \sigma^2) -  &(x^2-\rho x + 2\sigma^2 x)v_x(x,t) +  \\ 
& \frac{\sigma^2 x^2}{2}(v_x(x,t)^2 - v_{xx}(x,t)) + \frac{c'(t)}{c(t)} = 0 \notag
\end{align}
in which $c(t)= \int_{\mathbb{R}} e^{-v(x,t)} \,dx$. In our simulations, we chose $\rho=2, \sigma = 1$ and $X(0)\sim \Gamma(k=1,\theta=1.5)$, i.e. a gamma distribution with shape $k=1$ and scale $\theta=1.5$, which implies that $f(x,0) = \frac{1}{\Gamma(k) \theta^k} e^{-\left( \frac{x}{\theta} - \ln x^{k-1}\right)}$ and $v(x,0) = \frac{x}{\theta} - (k-1)\ln x$. With these parameters, the stationary distribution can then be represented as $f_s(x) = 4x^2 e^{-2x} = 4e^{-v_s(x)}$, $v_s(x) = 2x - 2\ln x$. 

We pursued a neural network solution to \eqref{eq:FPVerhulstGWNTransf} on the truncated spatial domain $x \in [0,9]$ for $t \in [0,4]$. The network architecture is composed of 2 neurons for the input layer, 1 neuron for the output layer, and 4 hidden layers with 100 neurons each. A mesh of 151 equally spaced points in the spatial domain and 301 equally spaced points in the time domain was constructed to obtain $N_{Op} = 151 \times 301$ collocation points to enforce the governing equation. This mesh was also used to numerically compute the integral terms in \eqref{eq:FPVerhulstGWNTransf}. Out of the 151 points in $x \in [0,9]$, $N_{IC} = 101$ points were randomly chosen to impose the initial condition constraints.

Figures~\ref{fig:VerhulstGWNStationary} and ~\ref{fig:VerhulstGWNsomeTimes} exhibit the results of the trained neural network which attained a loss value of 0.000163. Figure~\ref{fig:VerhulstGWNStationary} compares the analytical solution with the neural network representation; in particular, the left panel displays $v_s(x)$ (solid) and $\widetilde{v}(x,4)$ (dashed) which differ by a constant while the right panel displays $f_s(x)$ (solid) and $\widetilde{f}(x,4)$ (dashed) which coincide. These plots show that the neural network solution is able to recover the analytical stationary density at $t=4$ and further confirms that \eqref{eq:FPVerhulstGWNTransf} has no unique solution as remarked above. Figure~\ref{fig:VerhulstGWNsomeTimes} plots the neural network solution $\widetilde{f}(x,t)$ at $t=0$ (left panel) and $t=0.5$ (right panel) to illustrate that it is consistent with histograms of $X(0)$ and $X(0.5)$ obtained from 100000 Monte Carlo samples.

\begin{figure}[h!]
		\centering
		\includegraphics[width = 0.46\textwidth]		
						{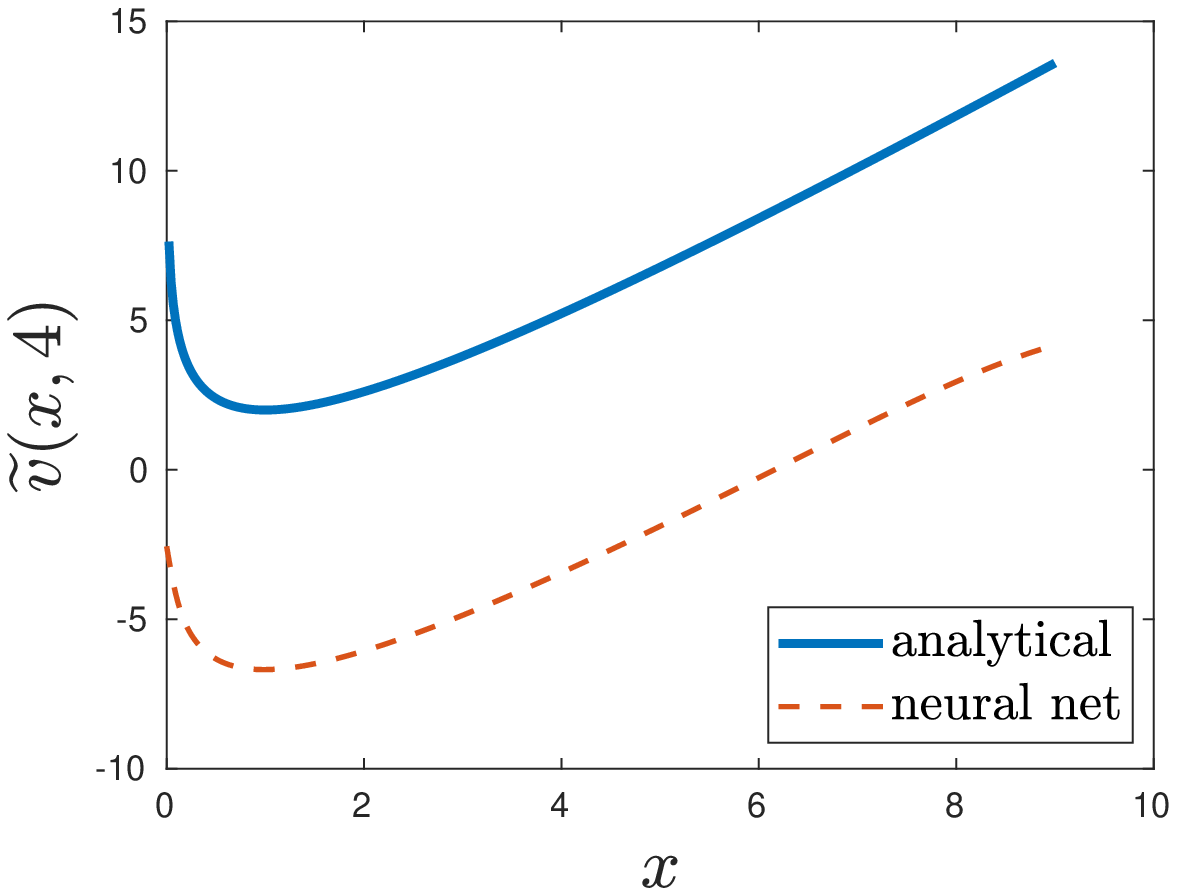} 
						\hspace{3em}
		\includegraphics[width = 0.43\textwidth]		
						{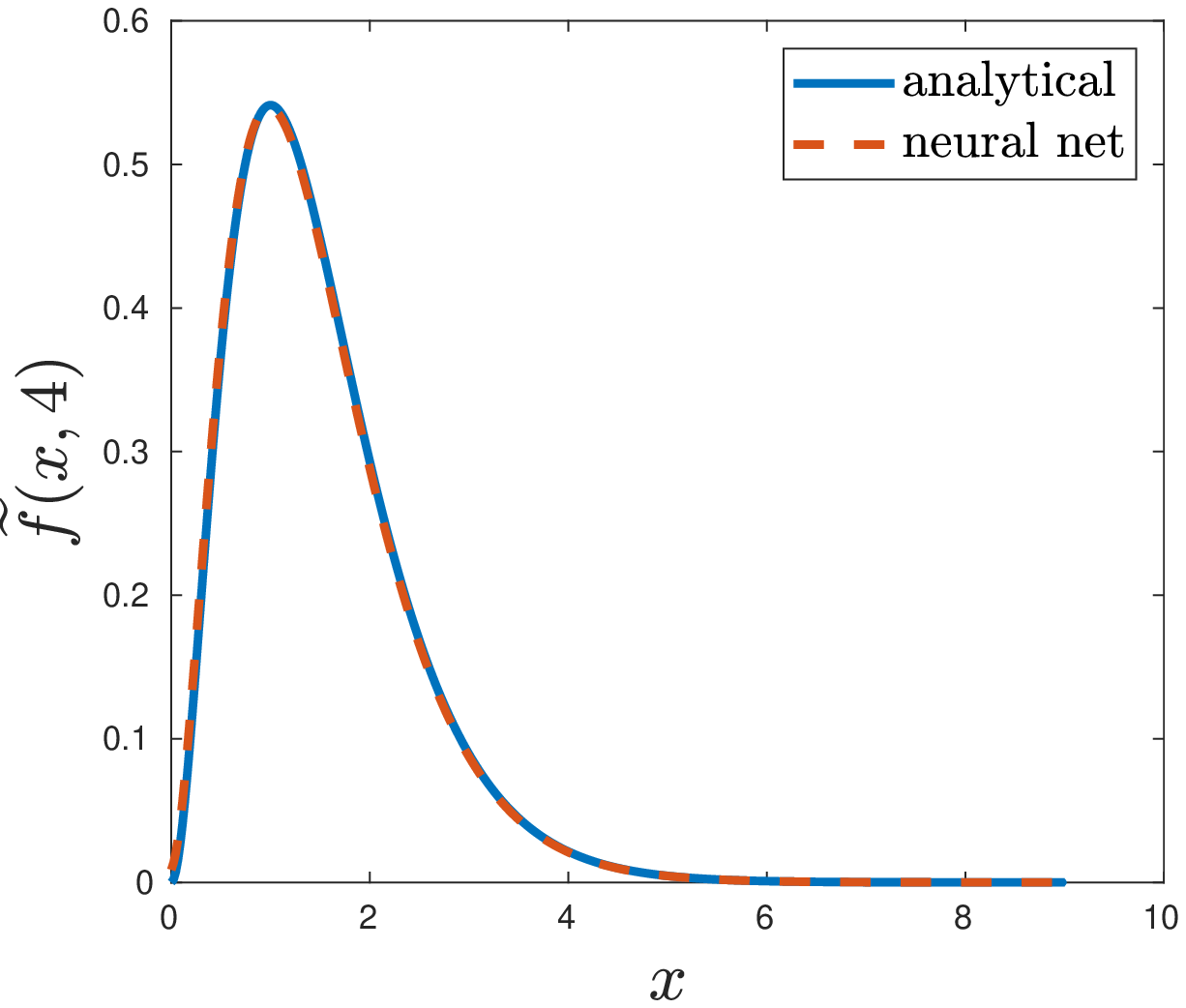}
		\caption{Comparison between the neural network approximation and the analytical stationary pdf for Section~\ref{subsec:VerhulstGWN}. Left: $v_s(x)$ (solid) vs $\widetilde{v}(x,4)$ (dashed). Right: $f_s(x)$ (solid) vs $\widetilde{f}(x,4)$ (dashed). }
		\label{fig:VerhulstGWNStationary}
\end{figure}

\begin{figure}[h!]
		\centering
		\includegraphics[width = 0.45\textwidth]		
						{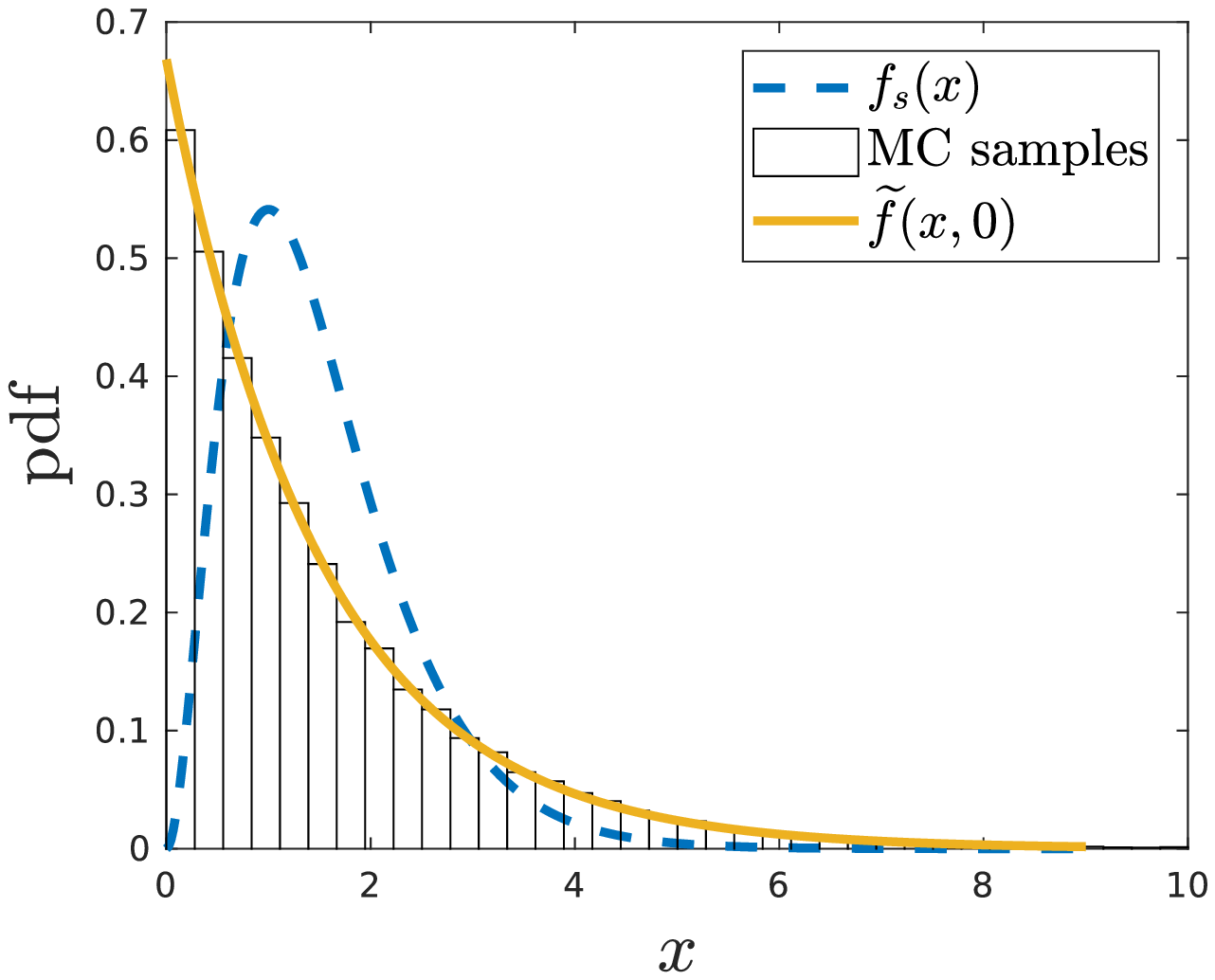} 
						\hspace{3em}
		\includegraphics[width = 0.46\textwidth]		
						{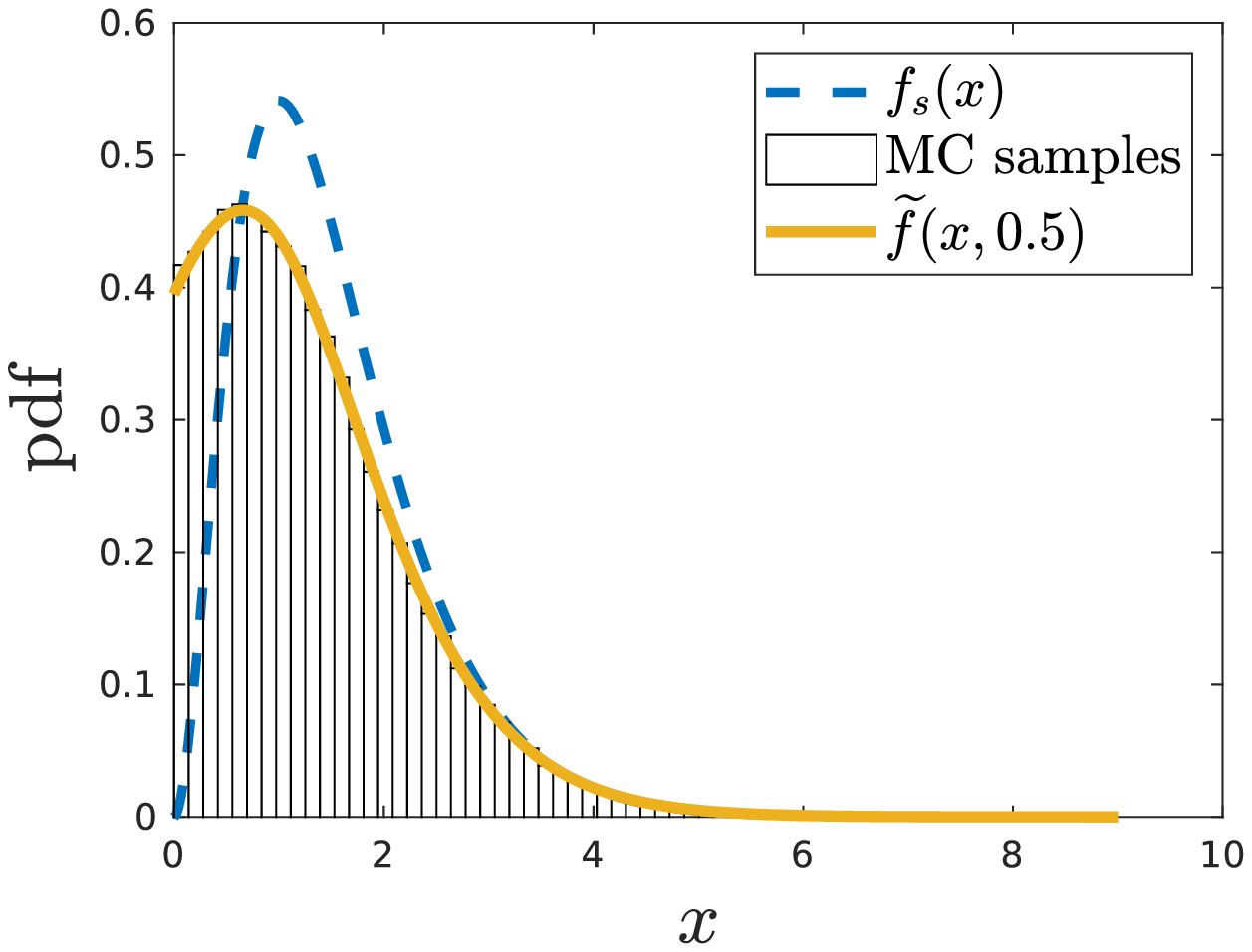}
		\caption{Comparison between the neural network approximation and Monte Carlo samples of $X(t)$  for Section~\ref{subsec:VerhulstGWN}. Left: histogram of $X(0)$ and $\widetilde{f}(x,0)$ (solid). Right: histogram of $X(0.5)$ and $\widetilde{f}(x,0.5)$ (solid). The analytical stationary density $f_s(x)$ is present in both plots (dashed). }
		\label{fig:VerhulstGWNsomeTimes}
\end{figure}

\subsection{Verhulst model with Poisson white noise} \label{subsec:VerhulstPWN}


Let $X(t)$ satisfy the SDE given by the Verhulst model
\begin{align*}
dX(t) = (\rho X(t-) - X(t-)^2)\,dt + X(t-) \,dC(t), \hspace{1em}t \ge 0, 
\end{align*}
where $C(t) = \sum_{k=1}^{N(t)}Y_k$ is a compound Poisson process that depends on a Poisson process $N(t)$ with intensity $\lambda$ and  iid random variables $Y_k$ with distribution $F$. It will be demonstrated that the neural network representation of the real and imaginary parts of the chf of $X(t)$ closely approximates the estimate provided by Monte Carlo simulation.

According to \eqref{eq:CharFunPDE} and~\eqref{eq:VerhulstGeneral}, the chf of $X(t)$ satisfies
\begin{equation} \label{eq:VerhulstPWNChfEx}
\mathcal{Q}[\varphi(u,t)] = \frac{\partial \varphi(u,t)}{\partial t} - \rho u \frac{\partial \varphi(u,t)}{\partial u} - i u\frac{\partial^2 \varphi(u,t)}{\partial u^2} - \lambda \left[\int_{\mathbb{R}} \varphi(u(1+y),t) \,dF(y) - \varphi(u,t) \right] = 0
\end{equation}
which is complex-valued and is consequently a system of partial integro-differential equations. For demonstration, the following parameters were utilized: $\rho=2, \lambda = 12$, $Y_k$ is a discrete random variable taking on values $\{z_k\}_{k=1}^7=\left \{-\frac{1}{2}+\frac{k-1}{6}\right \}_{k=1}^7$ with equal probability so that $$\displaystyle \int_{\mathbb{R}} \varphi(u(1+y),t) \,dF(y) = \frac{1}{7}\sum_{k=1}^7 \varphi(u(1+z_k),t)$$ in \eqref{eq:VerhulstPWNChfEx}, and $X(0) \sim U(0.5,5.5)$ with $\varphi(u,0) = \left( \frac{\sin(5.5u) - \sin(0.5u)}{5u} \right) -i \left(\frac{\cos(5.5u) - \cos(0.5u)}{5u} \right)$. We seek a neural network approximation $\widetilde{\varphi}(u,t) \in \mathbb{C}$ such that  $\mathcal{Q}[\widetilde{\varphi}(u,t)] = 0$ on the truncated domain $(u,t) \in [0,10] \times [0,1]$. It is furthermore assumed that $\widetilde{\varphi}(u,t) = 0$ whenever $u > 10$ to avoid extrapolating the neural network solution when evaluating the expression $\varphi(u(1+y),t)$ in \eqref{eq:VerhulstPWNChfEx}. Note that this assumption is not restricted to the approach pursued here in solving differential equations; such assumption would have to be made for other numerical schemes.

The neural network utilized is composed of an input and output layer with 2 neurons each and 4 hidden layers with 100 neurons each. To construct the loss function~\eqref{eq:CHF-loss}, $N_{IC} = 400$ equally spaced points in $(0,10]$, $N_0 = 100$ equally spaced points in $[0,1]$, and a regular grid of $N_{Op} = 20000$ points consisting of 400  points in $[0,10]$ and 50  points in $[0,1]$ were generated which yielded a loss value of $4.9294514 \times 10^{-4}$. The neural network solution is then compared to the chf $\varphi^{MC}(u,t)$ resulting from 50000 Monte Carlo samples of $X(t)$ simulated through forward Euler.

Figures~\ref{fig:VerhulstPWNTime1} and~\ref{fig:VerhulstPWNTime2} display the comparison between $\widetilde{\varphi}(u,t)$ and $\varphi^{MC}(u,t)$, $u\ge 0$, for $t = 0.25$ and $t=1$, respectively. In each figure, the left subplot shows the real part of the characteristic function while the right subplot shows the imaginary part. As the plots indicate, both approximations to the actual chf $\varphi(u,t)$ are similar. The discrepancy between the two approximations for $u$ close to 10 is due to the assumption imposed that $\widetilde{\varphi}(u,t) = 0$ for values of $u$ exceeding the truncated domain.

\begin{figure}[h!]
		\centering
		\includegraphics[width = 0.45\textwidth]		
						{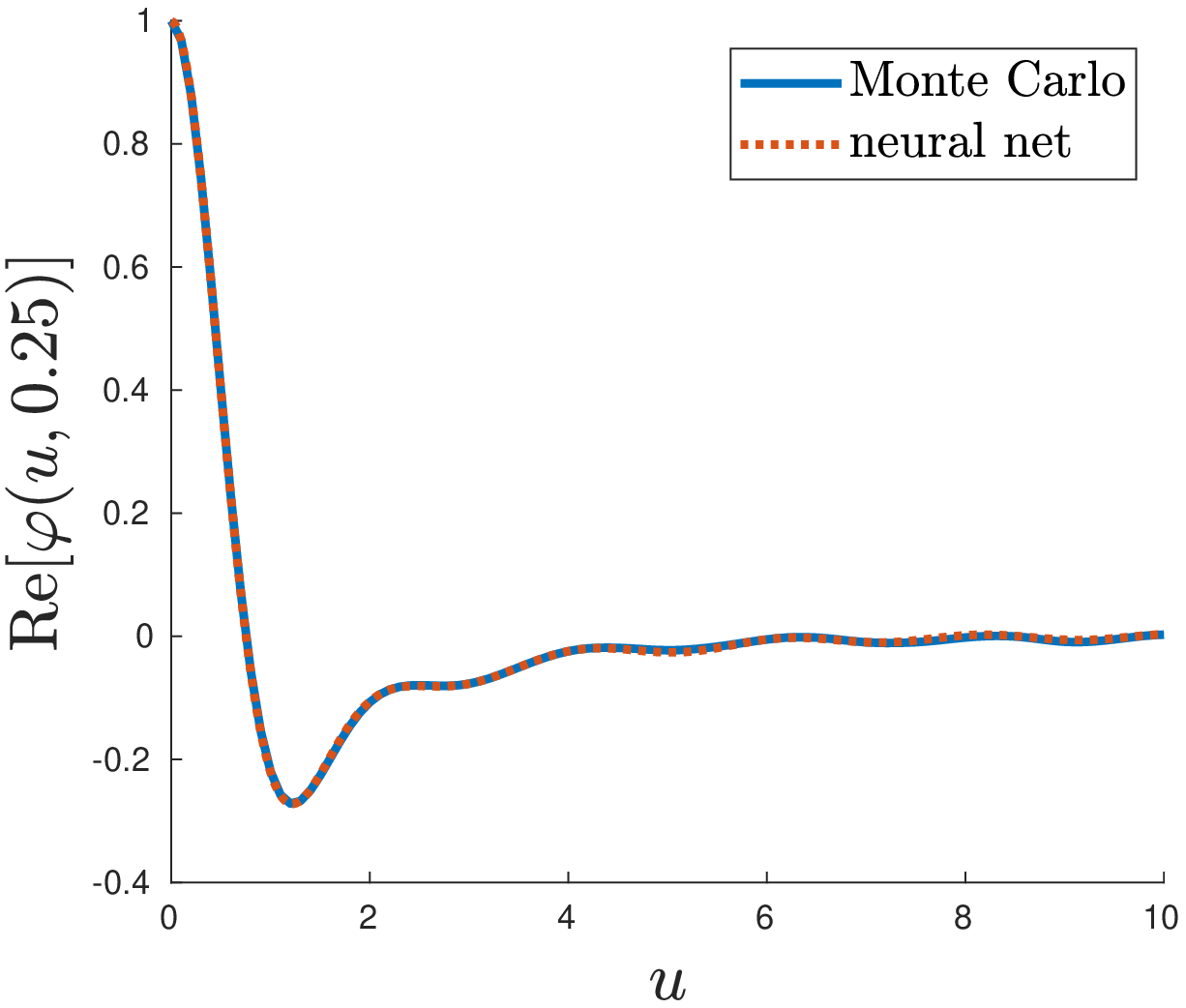} 
						\hspace{3em}
		\includegraphics[width = 0.45\textwidth]		
						{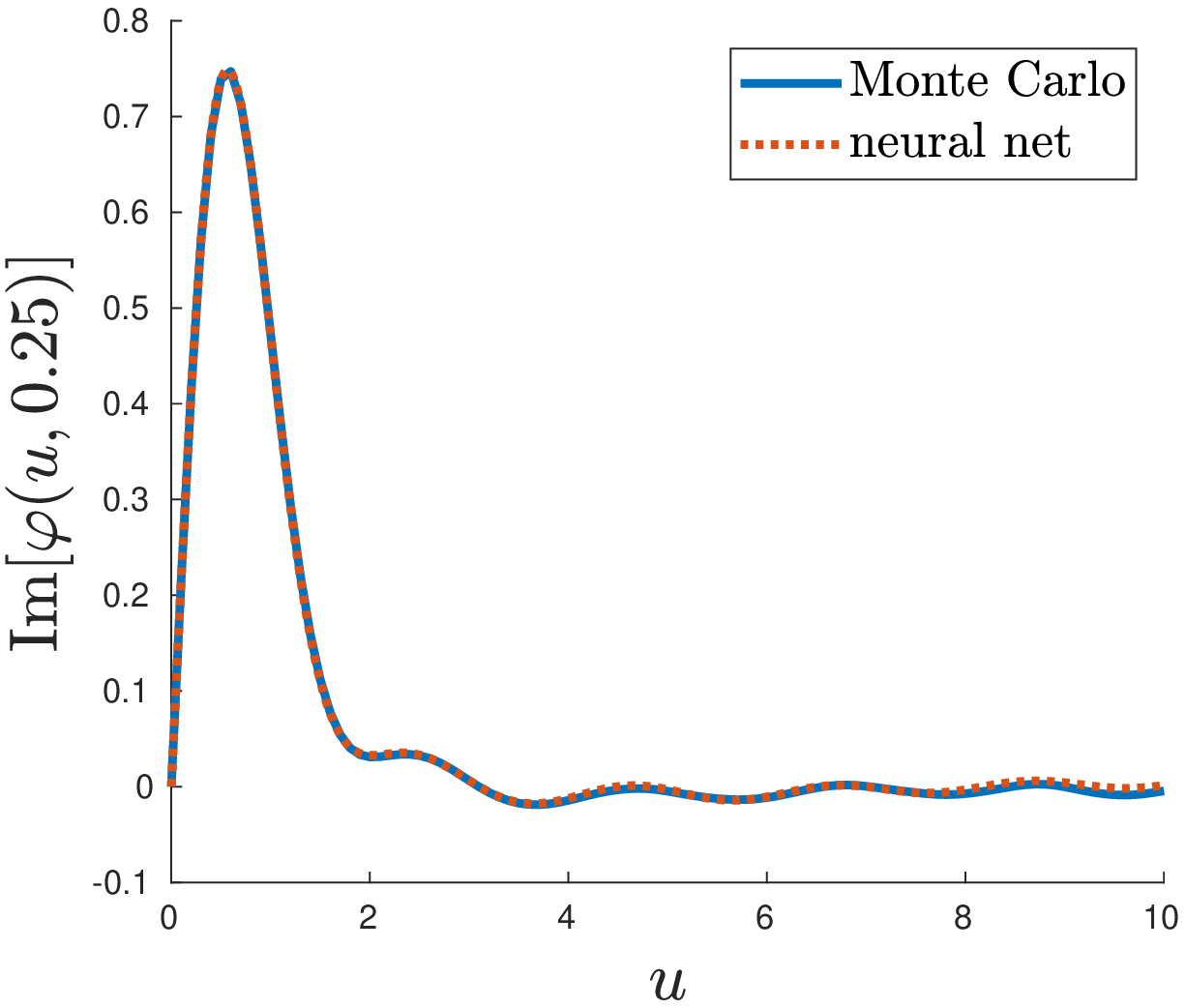}
		\caption{Comparison between the neural network approximation and the chf obtained via Monte Carlo for Section~\ref{subsec:VerhulstPWN} at $t=0.25$.  Left: $\text{Re}[\varphi^{MC}(u,0.25)]$ (solid) vs $\text{Re}[\widetilde{\varphi}(u,0.25)]$ (dashed). Right: $\text{Im}[\varphi^{MC}(u,0.25)]$ (solid) vs $\text{Im}[\widetilde{\varphi}(u,0.25)]$ (dashed). }
		\label{fig:VerhulstPWNTime1}
\end{figure}

\begin{figure}[h!]
		\centering
		\includegraphics[width = 0.45\textwidth]		
						{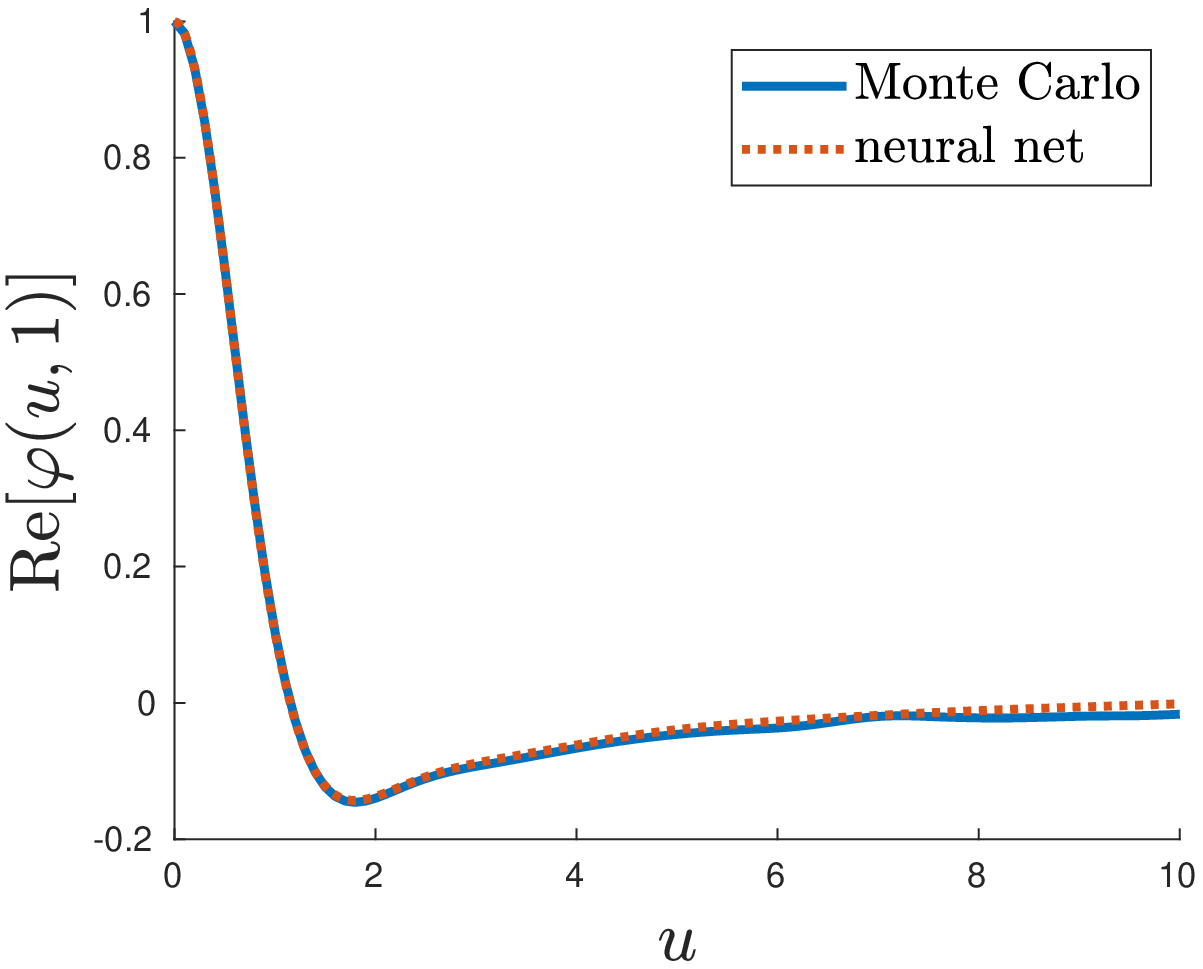} 
						\hspace{3em}
		\includegraphics[width = 0.45\textwidth]		
						{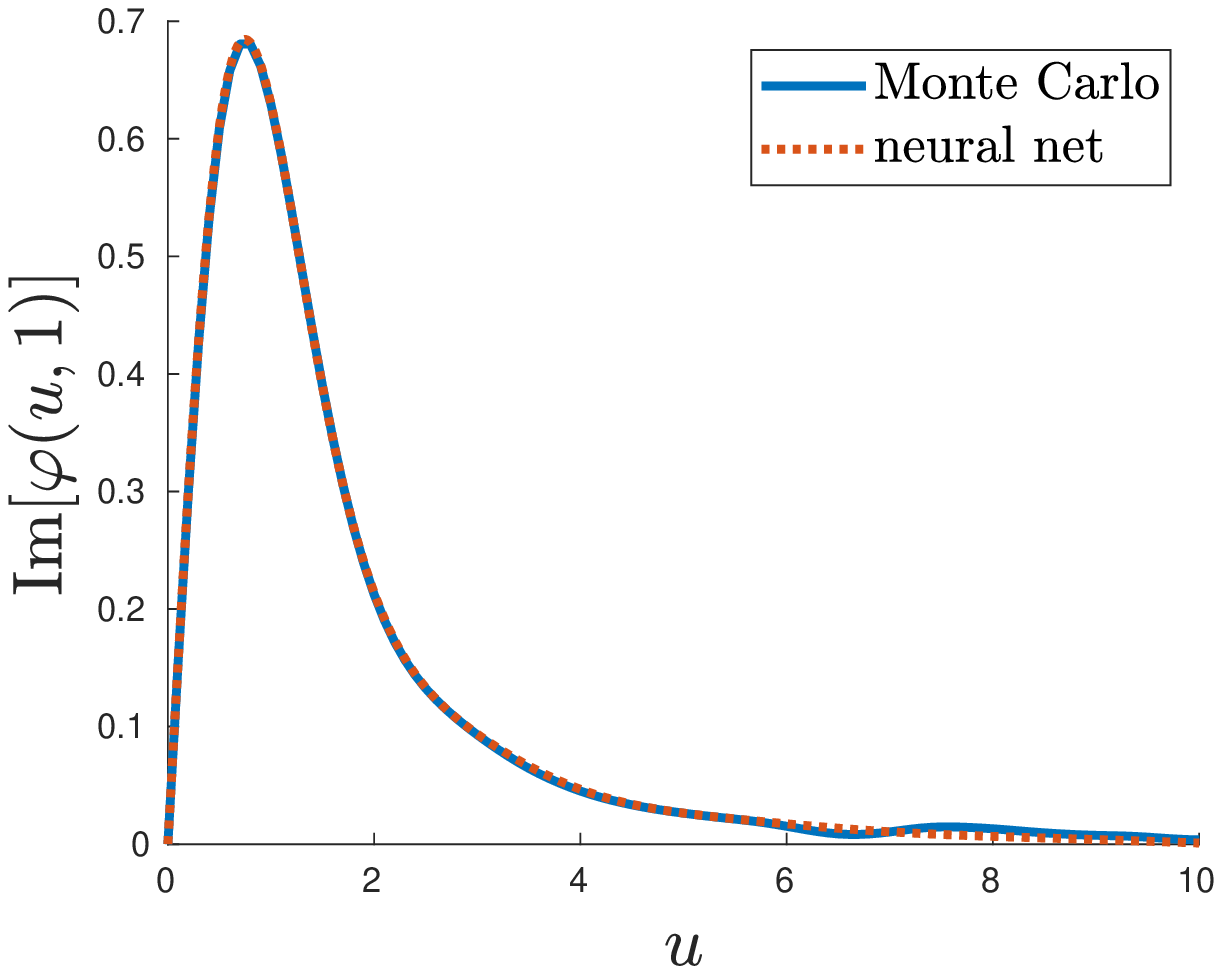}
		\caption{Comparison between the neural network approximation and the chf obtained via Monte Carlo for Section~\ref{subsec:VerhulstPWN} at $t=1$.  Left: $\text{Re}[\varphi^{MC}(u,1)]$ (solid) vs $\text{Re}[\widetilde{\varphi}(u,1)]$ (dashed). Right: $\text{Im}[\varphi^{MC}(u,1)]$ (solid) vs $\text{Im}[\widetilde{\varphi}(u,1)]$ (dashed). }
		\label{fig:VerhulstPWNTime2}
\end{figure}

\subsection{Duffing oscillator with Gaussian white noise} \label{subsec:DuffingFPvsCHF}

Let $X(t)$, $t \in [0,1]$, be the displacement of an oscillator with cubic stiffness under the Duffing model subject to Gaussian white noise external forcing. The displacement satisfies the SDE 
\begin{align} \label{eq:DuffOscillGWNSDE}
\ddot{X}(t) + 2\zeta \nu \dot{X}(t) + \nu^2 (X(t) + \alpha X(t)^3) = W(t)
\end{align}
where $\zeta \in (0,1)$ is the damping ratio, $\nu$ is the initial frequency, $\alpha$ is a constant, and $W(t)$ is Gaussian white noise with zero mean and one-sided spectral density of intensity $g_0 > 0$. In system form, this can be expressed as 
\begin{align}  \label{eq:DuffingGWNsystem}
d \begin{bmatrix}
X_1(t) \\ X_2(t)
\end{bmatrix}
 = 
\begin{bmatrix}
X_2(t) \\
-\nu^2 (X_1(t) + \alpha X_1(t)^3) - 2\zeta \nu X_2(t)
\end{bmatrix} \, dt +
\begin{bmatrix}
0\\
\sqrt{\pi g_0}
\end{bmatrix} \, dB(t)
\end{align}
where $X_1(t) = X(t), X_2(t) = \dot{X}(t)$ and $B(t)$ is Brownian motion \cite[p. 480]{book:Grigoriu2002}.  The objective is to numerically verify that the pdfs of $X(t)$ resulting from a neural network approximation of the Fokker-Planck equation and that of the PDE for the characteristic function coincide.


In the simulations that follow, the parameters values we selected were $\zeta = 0.25, \nu  = 1, \alpha = 1, g_0 = 1$ and $X_1(0) \sim N(0,1), X_2(0) \sim N(0,1)$ with $\rho=\text{Corr}(X_1(0),X_2(0)) = 0.8$. The sections below detail the neural network approximation we pursued under each approach. To examine the accuracy of each approximation, 250000 Monte Carlo samples of $X_1(t),X_2(t)$ are generated by simulating~\eqref{eq:DuffingGWNsystem} via Runge-Kutta scheme with time step size of $0.005$.

\subsubsection{Fokker-Planck equation} \label{subsubsec:DuffGWNFP}

According to~\eqref{eq:FPeqn}, the Fokker-Planck equation for~\eqref{eq:DuffingGWNsystem} is 
\begin{align*} 
\frac{\partial f(\boldsymbol{x},t)}{\partial t} = -x_2 \frac{\partial f(\boldsymbol{x},t)}{\partial x_1} + 2\zeta \nu f(\boldsymbol{x},t) + (\nu^2 (x_1 + \alpha x_1^3) + 2\zeta \nu x_2) \frac{\partial f(\boldsymbol{x},t)}{\partial x_2} + \frac{\pi g_0}{2} \frac{\partial^2 f(\boldsymbol{x},t)}{\partial x_2^2}
\end{align*}
which is transformed to
\begin{align} \label{eq:DuffGWNFPTransfPDE}
\mathcal{M}[v(\boldsymbol{x},t)] = v_t(\boldsymbol{x},t) + x_2 v_{x_1}(\boldsymbol{x},t) + 2\zeta \nu & - (\nu^2 (x_1 + \alpha x_1^3) + 2\zeta \nu x_2) v_{x_2}(\boldsymbol{x},t) \notag \\ 
&+ \frac{\pi g_0}{2} ((v_{x_2}(\boldsymbol{x},t))^2 - v_{x_2 x_2}(\boldsymbol{x},t)) + \frac{c'(t)}{c(t)} = 0
\end{align}
upon applying~\eqref{eq:FPtransformation} with $c(t) =  \int_{\mathbb{R}} \int_{\mathbb{R}} e^{-v(x_1,x_2,t)} \,dx_1\,dx_2$. From the initial conditions specified above, we have $f(\boldsymbol{x},0) = \frac{1}{2\pi \sqrt{1-\rho^2}}\exp \left( -\frac{1}{2(1-\rho^2)} \left[ x_1^2 + x_2^2 - 2\rho x_1 x_2 \right] \right)
$ or $v(\boldsymbol{x},0) =  \frac{1}{2(1-\rho^2)} \left[ x_1^2 + x_2^2 - 2\rho x_1 x_2 \right]$ for the transformed variable. 

We seek a neural network approximation $\widetilde{v}(\boldsymbol{x},t)$ such that $\mathcal{M}[\widetilde{v}(\boldsymbol{x},t)] = 0$ over the truncated domain $(x_1,x_2) \in [-4,4] \times [-8,8]$. The architecture of the network consists of an input layer with 3 neurons, an output layer with 1 neuron, and 6 hidden layers with 50 neurons each. The network was trained using a regular grid of $N_{IC} = 1025$ points (25 points in $x_1 \in [-4,4]$ and 41 points in $x_2 \in [-8,8]$) and $N_{Op} = 50000$ points formed by taking the tensor product of 50 latin hypercube samples in $t \in [0,1]$ and 1000 latin hypercube samples in $(x_1,x_2) \in [-4,4] \times [-8,8]$.  These $N_{Op}$ collocation points were also used to estimate the terms $c'(t)$ and $c(t)$ in the operator $\mathcal{M}[v(\boldsymbol{x},t)]$~\eqref{eq:DuffGWNFPTransfPDE} via Monte Carlo integration.

Figures~\ref{fig:DuffGWNFPtransfSoln},~\ref{fig:DuffGWNFPtime1}, and~\ref{fig:DuffGWNFPtime2} summarize the performance of the neural network approximation $\widetilde{f}(\boldsymbol{x},t)$ to the solution of the Fokker-Planck equation which is then compared to the pdf $f^{MC}(\boldsymbol{x},t)$ representing the kernel density estimate from the Monte Carlo samples of $X_1(t),X_2(t)$. The panels of Figure~\ref{fig:DuffGWNFPtransfSoln} display $\widetilde{v}(\boldsymbol{x},t)$ at $t = 0.25,0.75$. Figure~\ref{fig:DuffGWNFPtime1} compares $f^{MC}(\boldsymbol{x},t)$ (left) with $\widetilde{f}(\boldsymbol{x},t)$ (right) for $t=0.25$; the same comparison is made in Figure~\ref{fig:DuffGWNFPtime2} but for $t=0.75$.  As these plots reveal, the neural network and Monte Carlo approximations are similar in behavior. Furthermore, denote by $\widetilde{f}_1(x_1,t)$ the estimate of the pdf of $X(t)$ which results by marginalizing $\widetilde{f}(\boldsymbol{x},t)$ through $\widetilde{f}_1(x_1,t) =  \int_{[-8,8]} \widetilde{f}(x_1,x_2,t) \,dx_2$.  As Figure~\ref{fig:DuffGWNFPX1Hist} demonstrates, $\widetilde{f}_1(x_1,t)$ is able to match the histograms of Monte Carlo samples of $X_1(t)$ for $t=0.25$ (left) and $t = 0.75$ (right).

\begin{figure}[h!]
		\centering
		\includegraphics[width = 0.45\textwidth]		
						{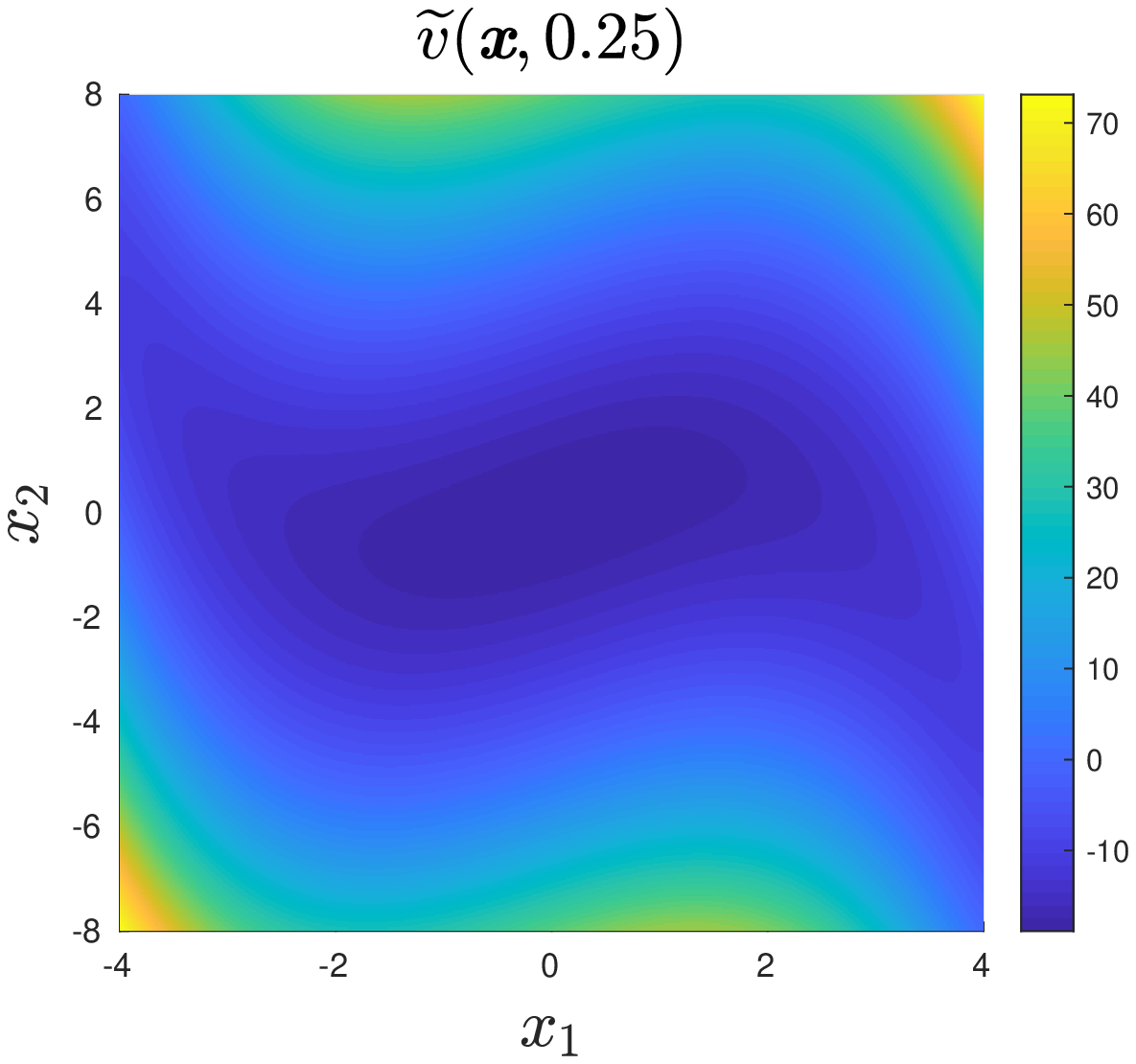} 
						\hspace{3em}
		\includegraphics[width = 0.45\textwidth]		
						{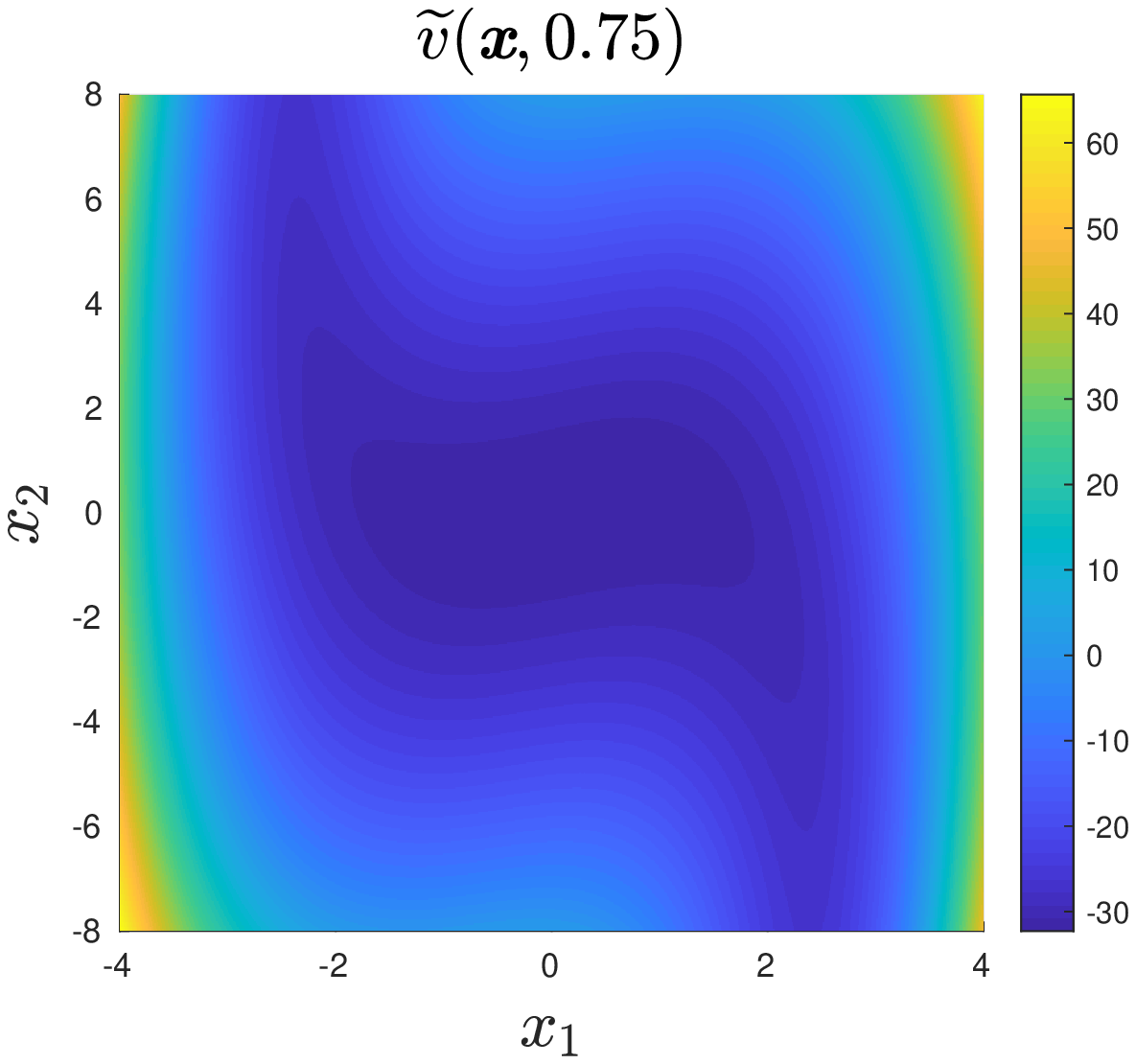}
		\caption{Neural network approximation $\widetilde{v}(\boldsymbol{x},t)$ for Section~\ref{subsubsec:DuffGWNFP} at $t=0.25$ (left) and $t=0.75$ (right).}
		\label{fig:DuffGWNFPtransfSoln}
\end{figure}

\begin{figure}[h!]
		\centering
		\includegraphics[width = 0.46\textwidth]		
						{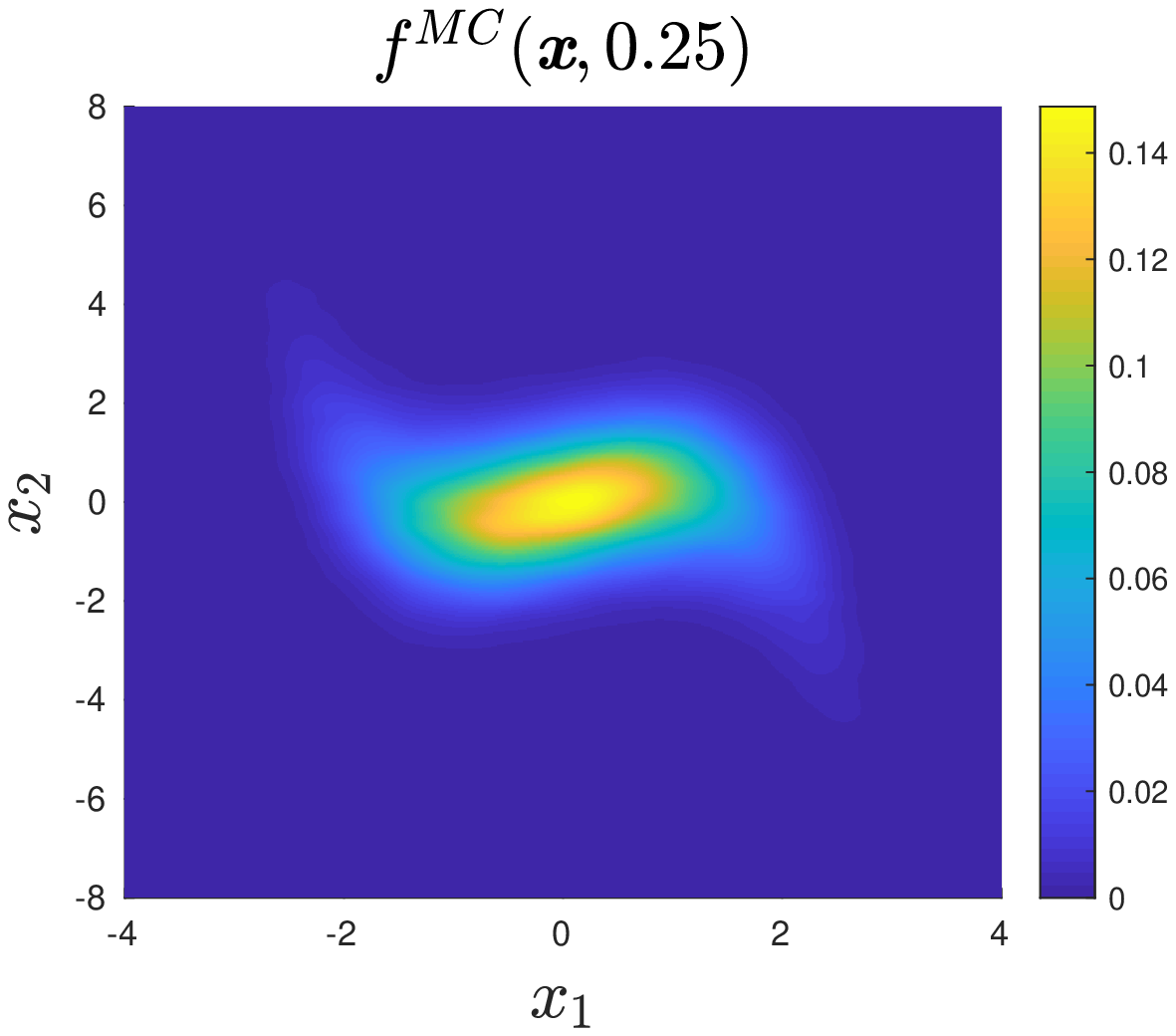} 
						\hspace{3em}
		\includegraphics[width = 0.44\textwidth]		
						{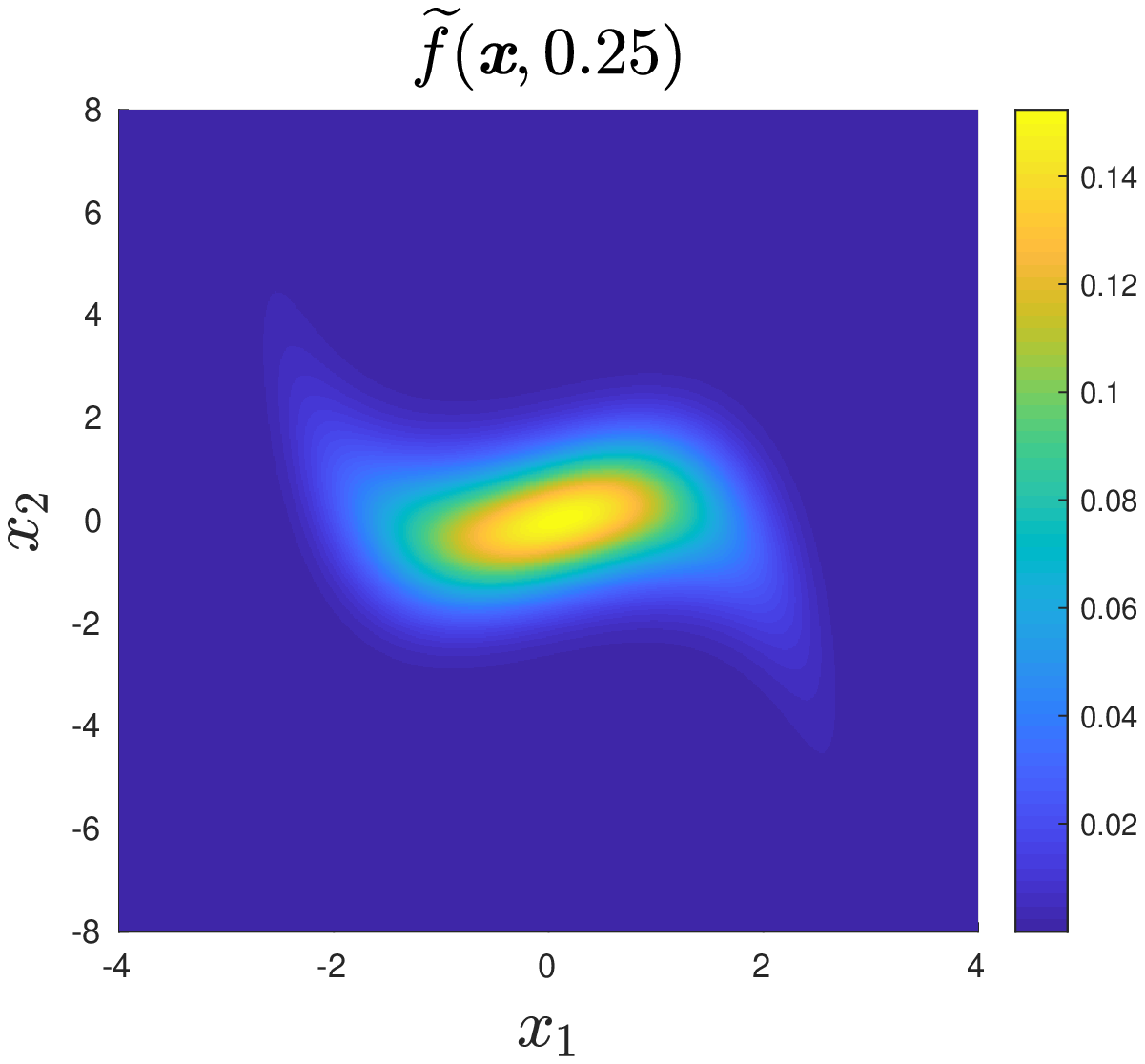}
		\caption{Comparison between the pdf $f^{MC}(\boldsymbol{x},t)$ obtained via kernel density estimation on Monte Carlo samples (left) and the neural network approximation $\widetilde{f}(\boldsymbol{x},t)$ (right) for Section~\ref{subsubsec:DuffGWNFP} at $t=0.25$.  }
		\label{fig:DuffGWNFPtime1}
\end{figure}

\begin{figure}[h!]
		\centering
		\includegraphics[width = 0.45\textwidth]		
						{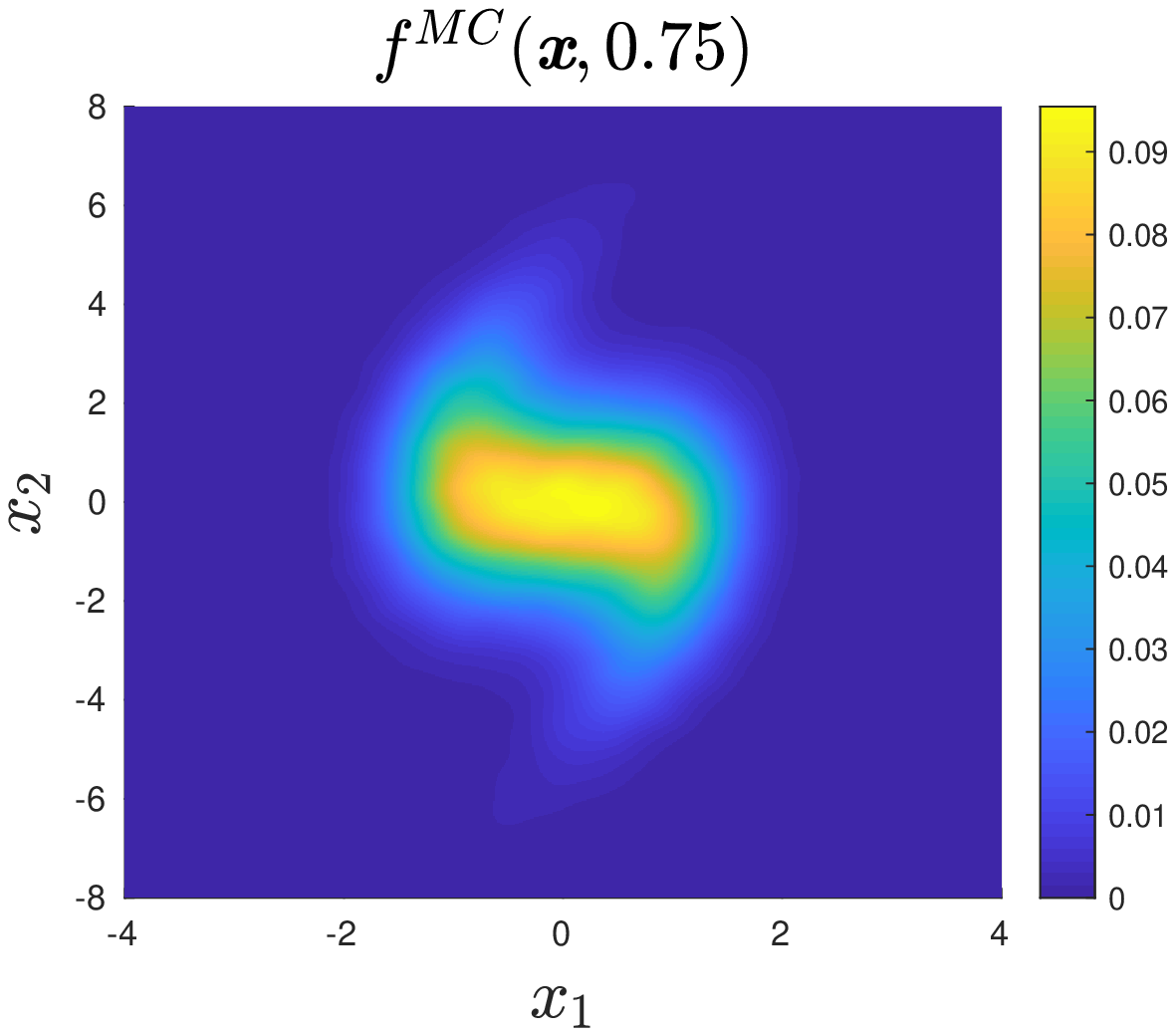} 
						\hspace{3em}
		\includegraphics[width = 0.45\textwidth]		
						{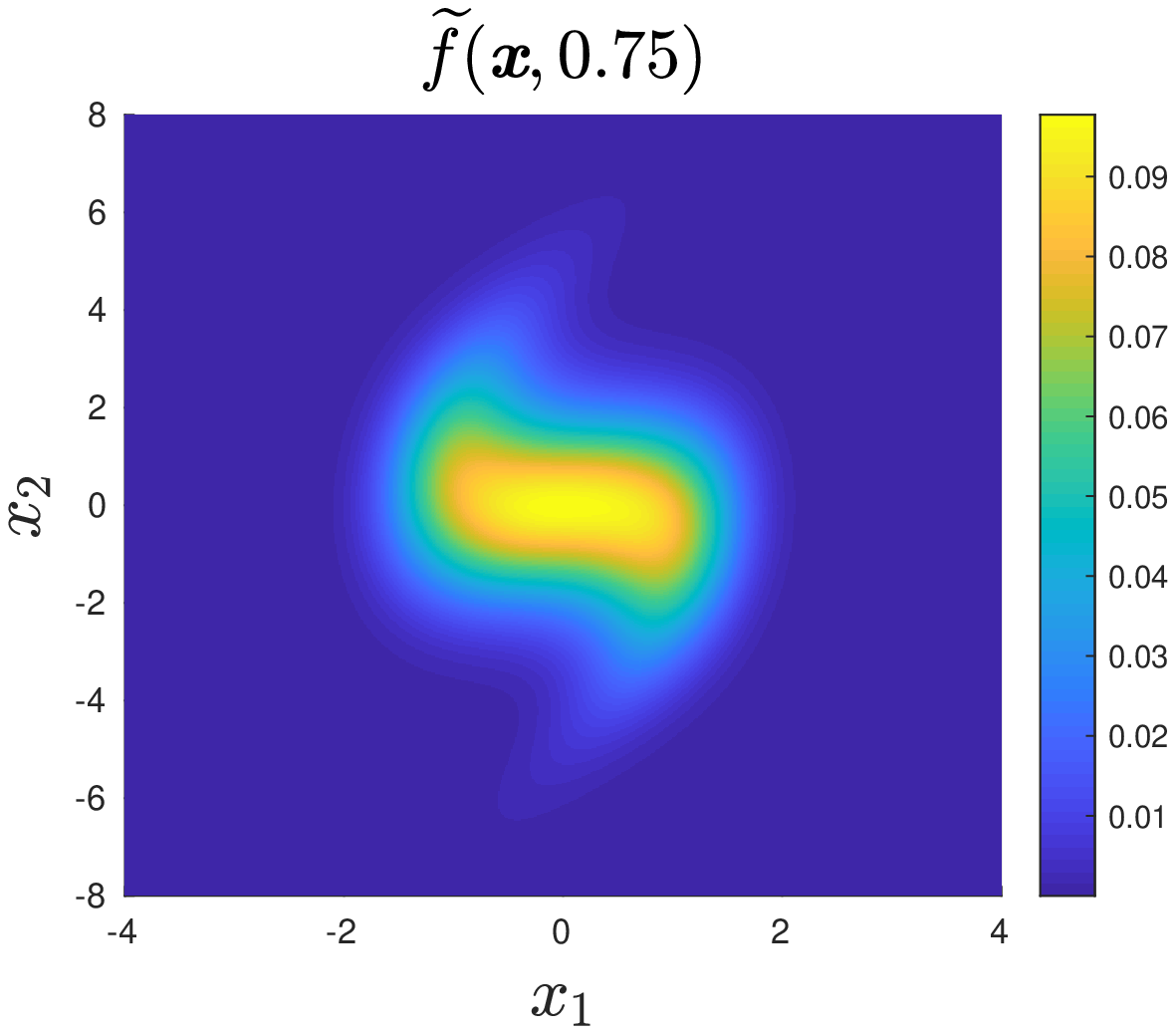}
		\caption{Comparison between the pdf $f^{MC}(\boldsymbol{x},t)$ obtained via kernel density estimation on Monte Carlo samples (left) and the neural network approximation $\widetilde{f}(\boldsymbol{x},t)$ (right) for Section~\ref{subsubsec:DuffGWNFP} at $t=0.75$. }
		\label{fig:DuffGWNFPtime2}
\end{figure}

While the above plots suggest that $\widetilde{f}(\boldsymbol{x},t)$ adequately solves the Fokker-Planck equation, the loss value corresponding to $\widetilde{v}(\boldsymbol{x},t)$ is 0.0131580755 which is relatively higher than those from the previous examples. To understand why $\widetilde{f}(\boldsymbol{x},t)$ offers a sufficient approximation despite having a large loss, we construct a binned scatterplot in Figure~\ref{fig:DuffGWNFPLargeLoss} between $\|\boldsymbol{x}\|$ and $|\mathcal{M}[\widetilde{v}(\boldsymbol{x},t)]|$ using the $N_{Op} = 50000$ collocation points we have generated. Figure~\ref{fig:DuffGWNFPLargeLoss} underscores that the loss value is large because the error in the governing equation is large for collocation points far from the origin, i.e. close to the boundary of $[-4,4]\times [-8,8]$. However, at these collocation points, the magnitude of the actual pdf $f(\boldsymbol{x},t)$ is considerably small due to the boundary conditions of the Fokker-Planck equation. Thus, when the transformation~\eqref{eq:FPtransformation} is applied to normalize the solution $\widetilde{v}(\boldsymbol{x},t)$, the large error at these collocation points is nullified for the approximation $\widetilde{f}(\boldsymbol{x},t)$. The same behavior persists for other neural network architectures we have investigated. This presents an example as to why it may be disadvantageous to solve the Fokker-Planck equation using neural networks -- it may not be always possible to diagnose why the loss value for $\widetilde{v}(\boldsymbol{x},t)$ is large. Finally, we noticed in our numerical experiments that for some architectures, $\widetilde{v}(\boldsymbol{x},t)$ has a tendency of being very negative which renders $\frac{c'(t)}{c(t)}$ and hence $\mathcal{M}[\widetilde{v}(\boldsymbol{x},t)]$~\eqref{eq:DuffGWNFPTransfPDE} nan. The optimization algorithm for minimizing the loss is unable to  proceed in such cases. This scenario is due to  lack of a unique solution to~\eqref{eq:DuffGWNFPTransfPDE} as discussed in Section~\ref{subsec:CompCHFvsPDF}.


\begin{figure}[h!]
		\centering
		\includegraphics[width = 0.43\textwidth]		
						{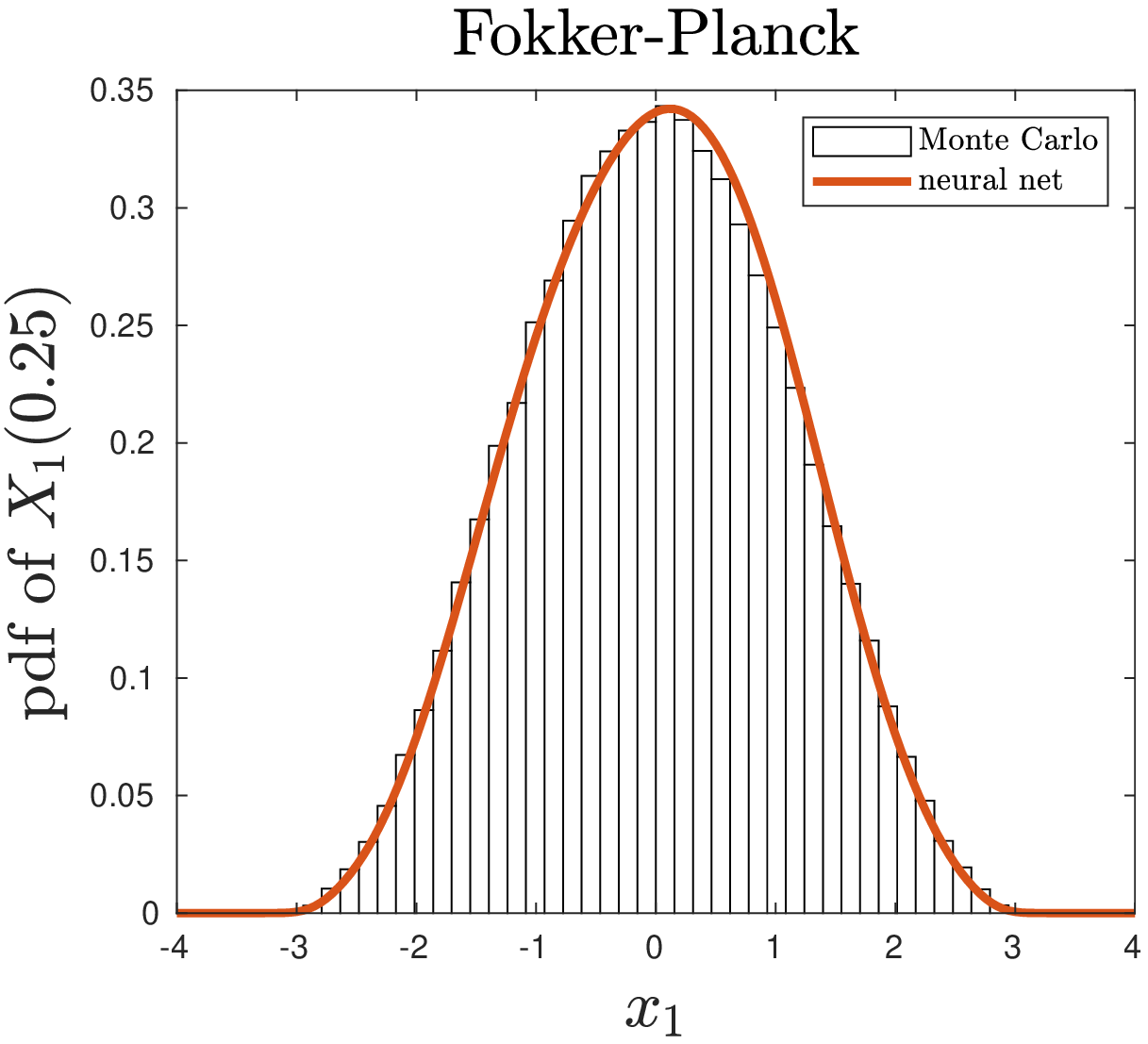} 
						\hspace{3em}
		\includegraphics[width = 0.46\textwidth]		
						{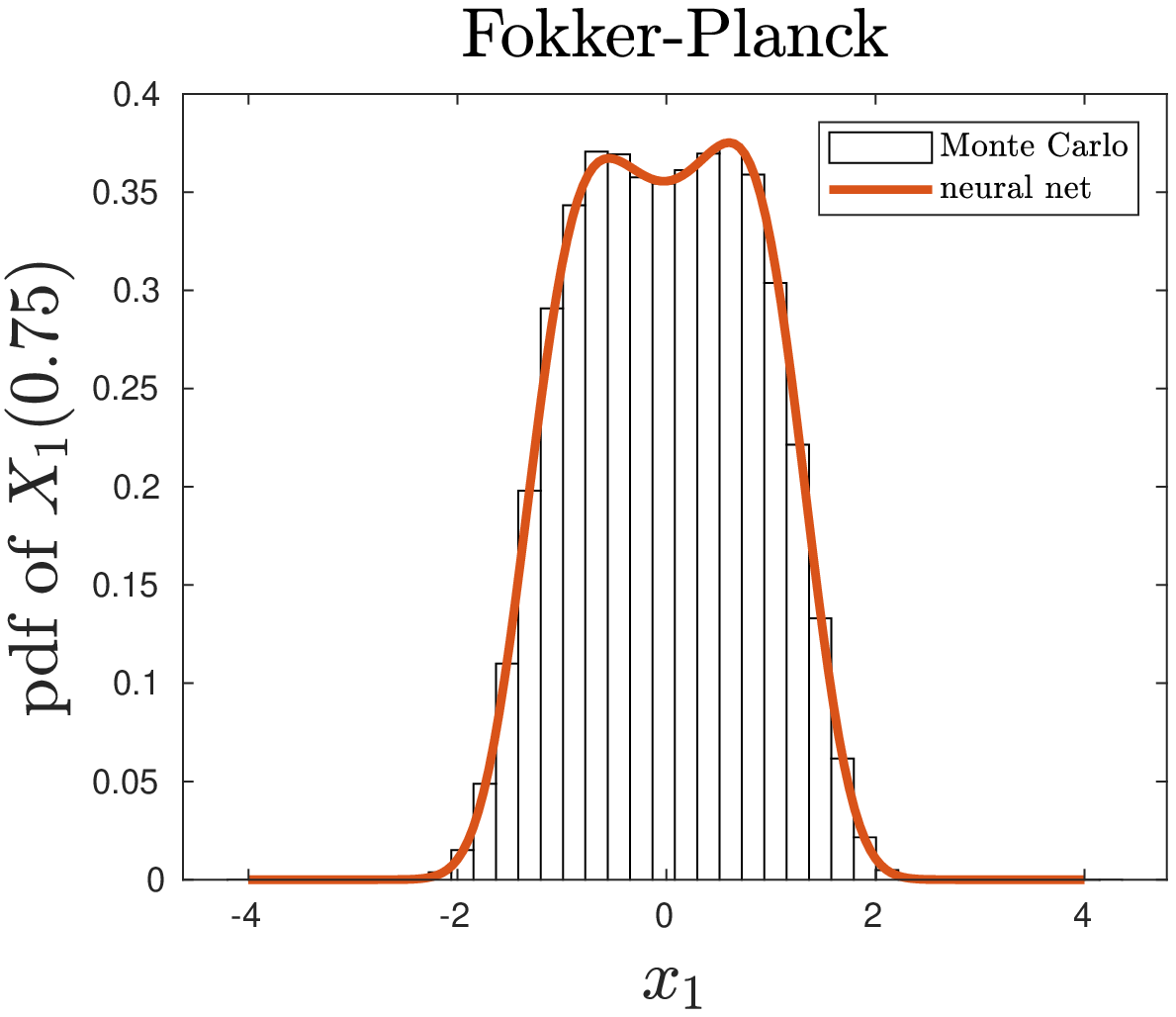}
		\caption{Comparison between the neural network approximation $\widetilde{f}_1(x_1,t)$ to the pdf of $X_1(t)$ and the histogram based on Monte Carlo samples of $X_1(t)$ for Section~\ref{subsubsec:DuffGWNFP} at $t=0.25$ (left) and $t = 0.75$ (right). }
		\label{fig:DuffGWNFPX1Hist}
\end{figure}

\begin{figure}[h!]
		\centering
		\includegraphics[width = 0.45\textwidth]		
						{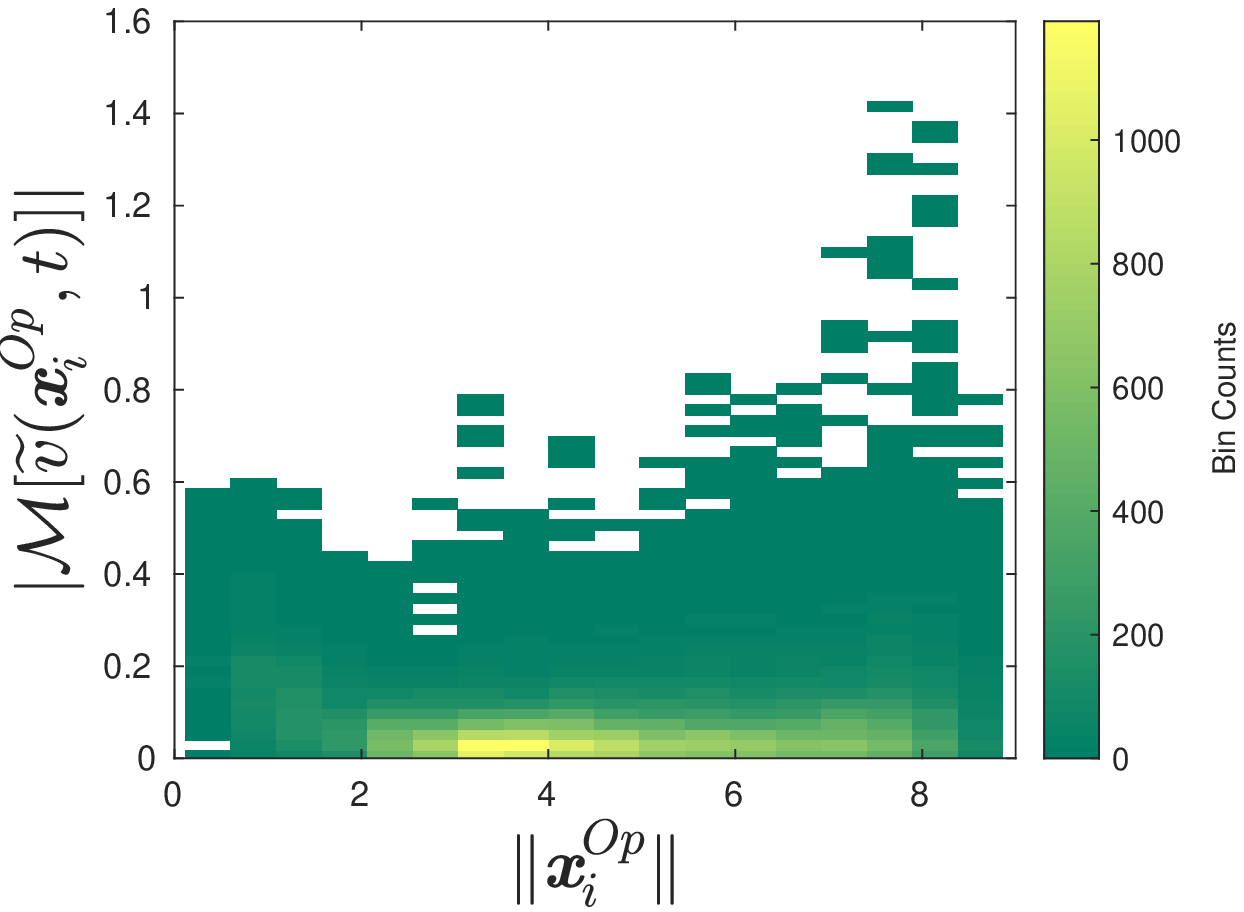} 
		\caption{Binned scatterplot of $\|\boldsymbol{x}\|$ vs $|\mathcal{M}[\widetilde{v}(\boldsymbol{x},t)]|$ using $N_{Op}$ collocation points to enforce~\eqref{eq:DuffGWNFPTransfPDE}   for Section~\ref{subsubsec:DuffGWNFP}.}
		\label{fig:DuffGWNFPLargeLoss}
\end{figure}

\subsubsection{Characteristic function PDE} \label{subsubsec:DuffGWNCharFun}

In contrast to the previous section, we estimate the pdf of $X_1(t) = X(t)$ by first solving the PDE of the characteristic function and subsequently applying the Fourier transform. From~\eqref{eq:CharFunPDE} and~\eqref{eq:DuffingCHFGeneralPDE}, the chf $\varphi(\boldsymbol{u},t)$ satisfies
\begin{align} \label{eq:DuffGWNCHFPDE}
\mathcal{Q}[\varphi(\boldsymbol{u},t)] = \frac{\partial \varphi(\boldsymbol{u},t)}{\partial t} - (u_1 + 2 \zeta \nu u_2 ) \frac{\partial \varphi(\boldsymbol{u},t)}{\partial u_2} - \nu^2 u_2 \frac{\partial \varphi(\boldsymbol{u},t)}{\partial u_1} + \nu^2 \alpha u_2 \frac{\partial^3 \varphi(\boldsymbol{u},t)}{\partial u_1^3} - \frac{\pi g_0}{2} u_2^2 \varphi(\boldsymbol{u},t) = 0
\end{align}
with initial condition
$\varphi(\boldsymbol{u},0) = \exp \left( -\frac{1}{2} 
\boldsymbol{u}'
\begin{bmatrix}
  1 & \rho \\
  \rho & 1
\end{bmatrix}
\boldsymbol{u}
\right).$ This PDE was solved on the truncated domain $\boldsymbol{u} \in [-6,6]^2.$

The neural network approximation $\widetilde{\varphi}(\boldsymbol{u},t)$ we seek is equipped with an architecture that constitutes an input layer with 3 neurons, an output layer with 1 neuron, and 5 hidden layers with 50 neurons each. The output layer only has 1 neuron because we can leverage on prior probabilistic information on~\eqref{eq:DuffingGWNsystem} to deduce that $\varphi(\boldsymbol{u},t)$ is real-valued. To see this, by using the fact that $B(t)$ and $-B(t)$ identically distributed, it follows that $(X_1(t),X_2(t))$ and $(-X_1(t),-X_2(t))$ are identically distributed since both sets of random vectors satisfy the SDE~\eqref{eq:DuffingGWNsystem}. This implies that $f(\boldsymbol{x},t)$ is symmetric with respect to the spatial origin and hence, $\varphi(\boldsymbol{u},t)$ has imaginary part 0. To compute the loss function, the collocation points we utilized were a regular grid of $N_{IC} = 33\times 33$ points in $\boldsymbol{u} \in [-6,6]^2$ as well as latin hypercube samples with $N_0 = 100$ points in $t \in [0,1]$ and $N_{Op} = 100000$ points in $(\boldsymbol{u},t) \in [-6,6]^2 \times [0,1]$ which resulted in a loss value of  $5.3324147\times 10^{-5}$.


\begin{figure}[h!]
		\centering
		\includegraphics[width = 0.45\textwidth]		
						{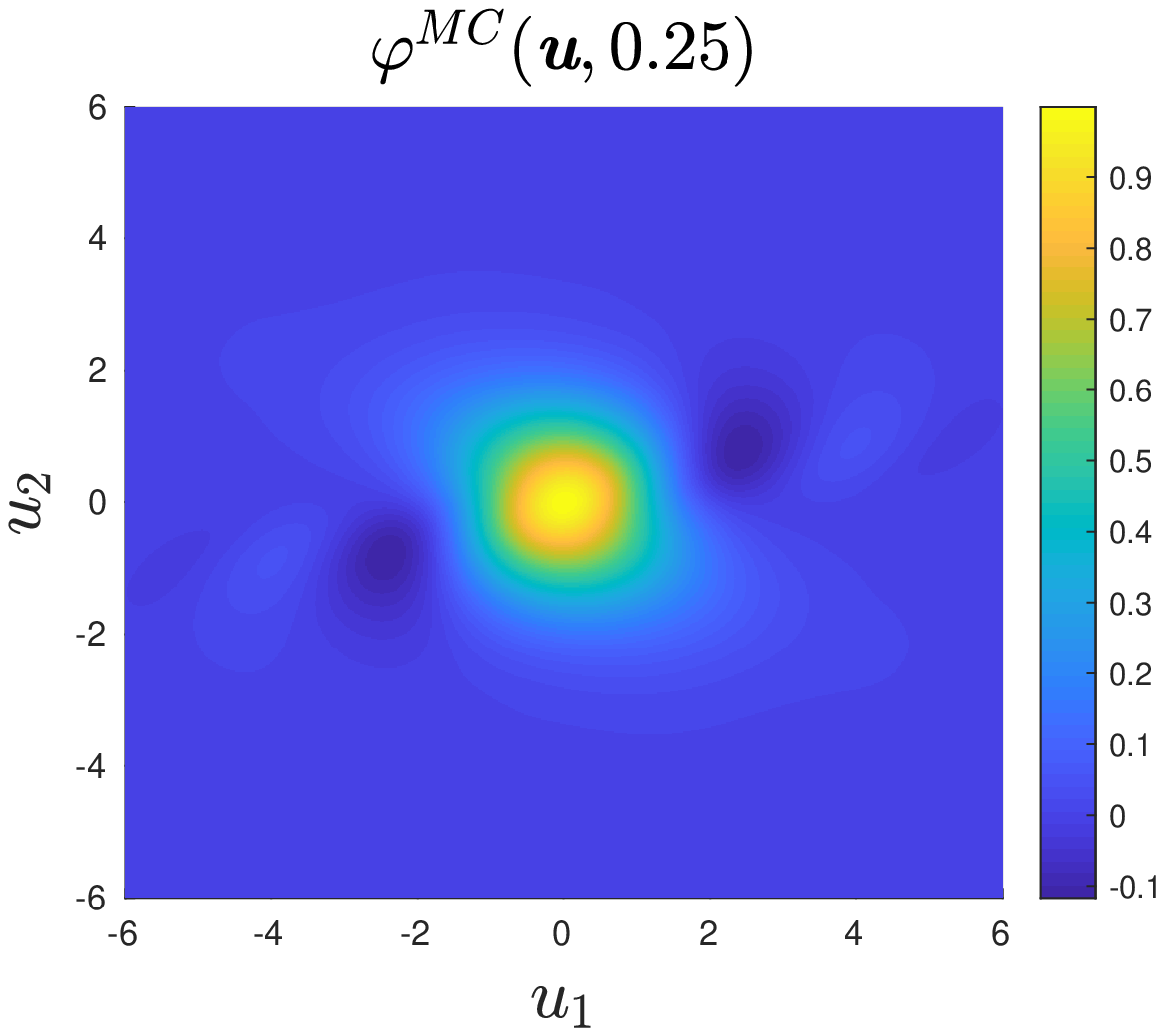} 
						\hspace{3em}
		\includegraphics[width = 0.45\textwidth]		
						{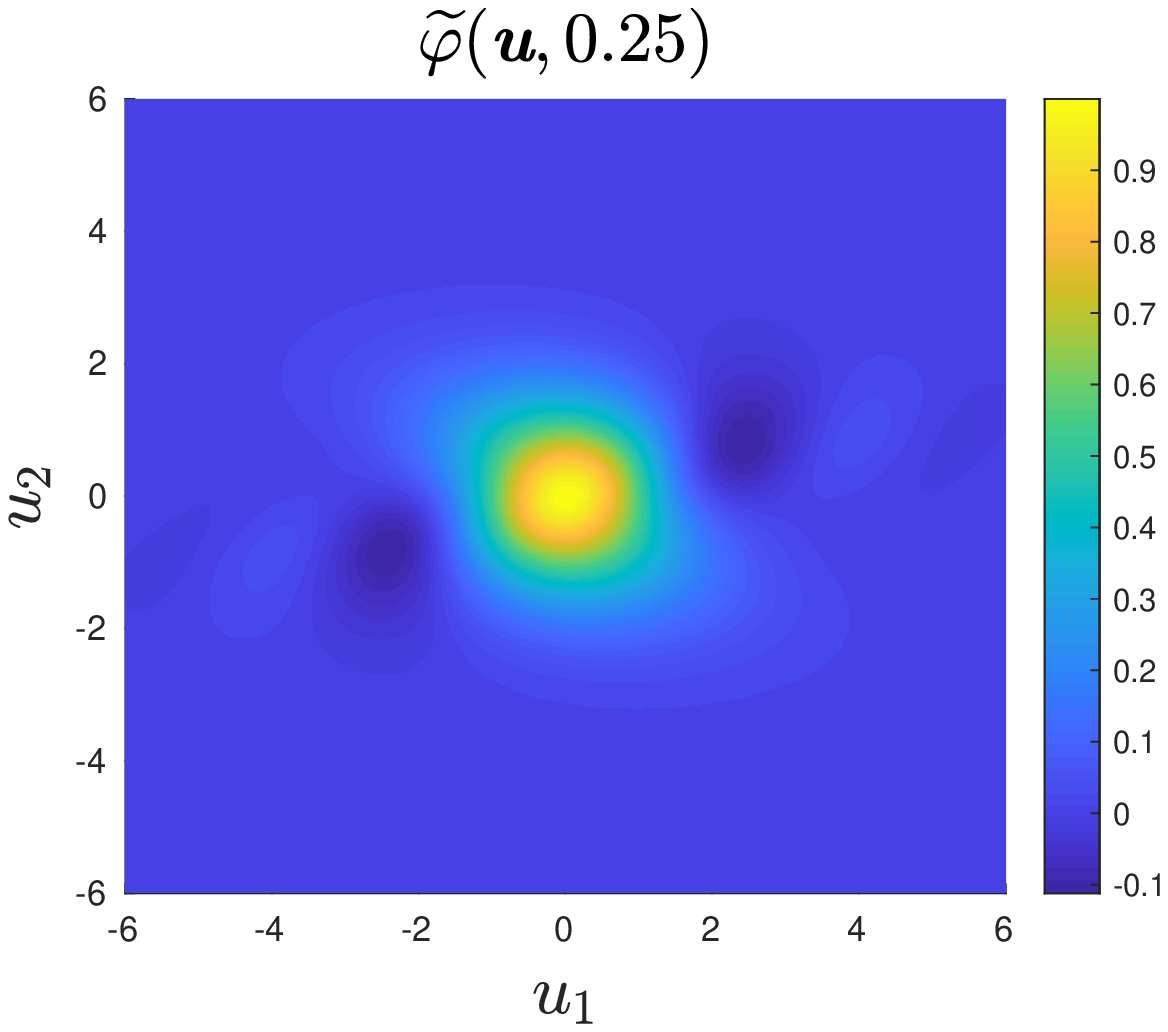}
		\caption{Comparison between the chf $\varphi^{MC}(\boldsymbol{u},t)$ obtained from Monte Carlo samples (left) and the neural network approximation $\widetilde{\varphi}(\boldsymbol{u},t)$ (right) for Section~\ref{subsubsec:DuffGWNCharFun} at $t=0.25$.}
		\label{fig:DuffGWNCHFTime1}
\end{figure}

\begin{figure}[h!]
		\centering
		\includegraphics[width = 0.45\textwidth]		
						{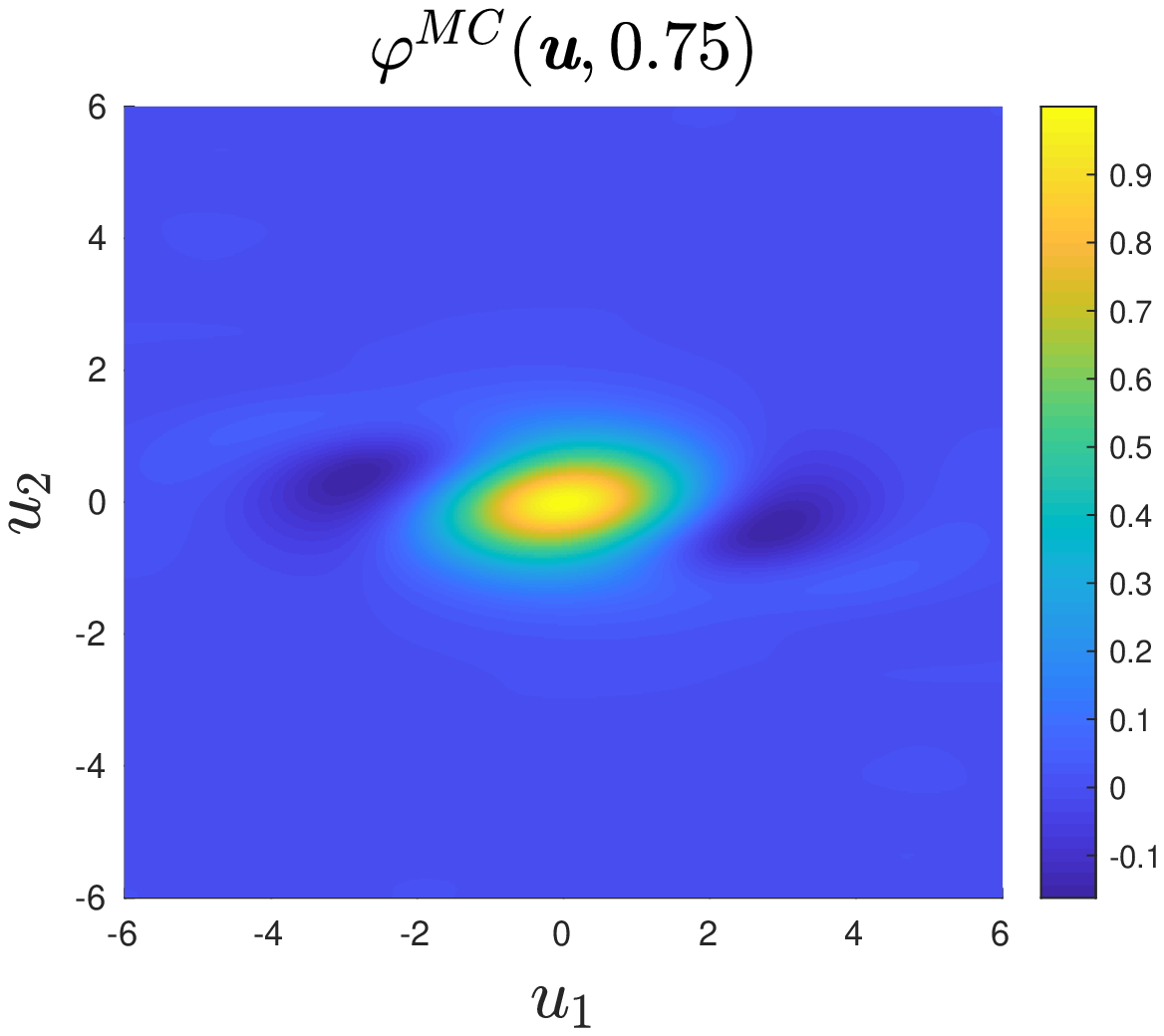} 
						\hspace{3em}
		\includegraphics[width = 0.45\textwidth]		
						{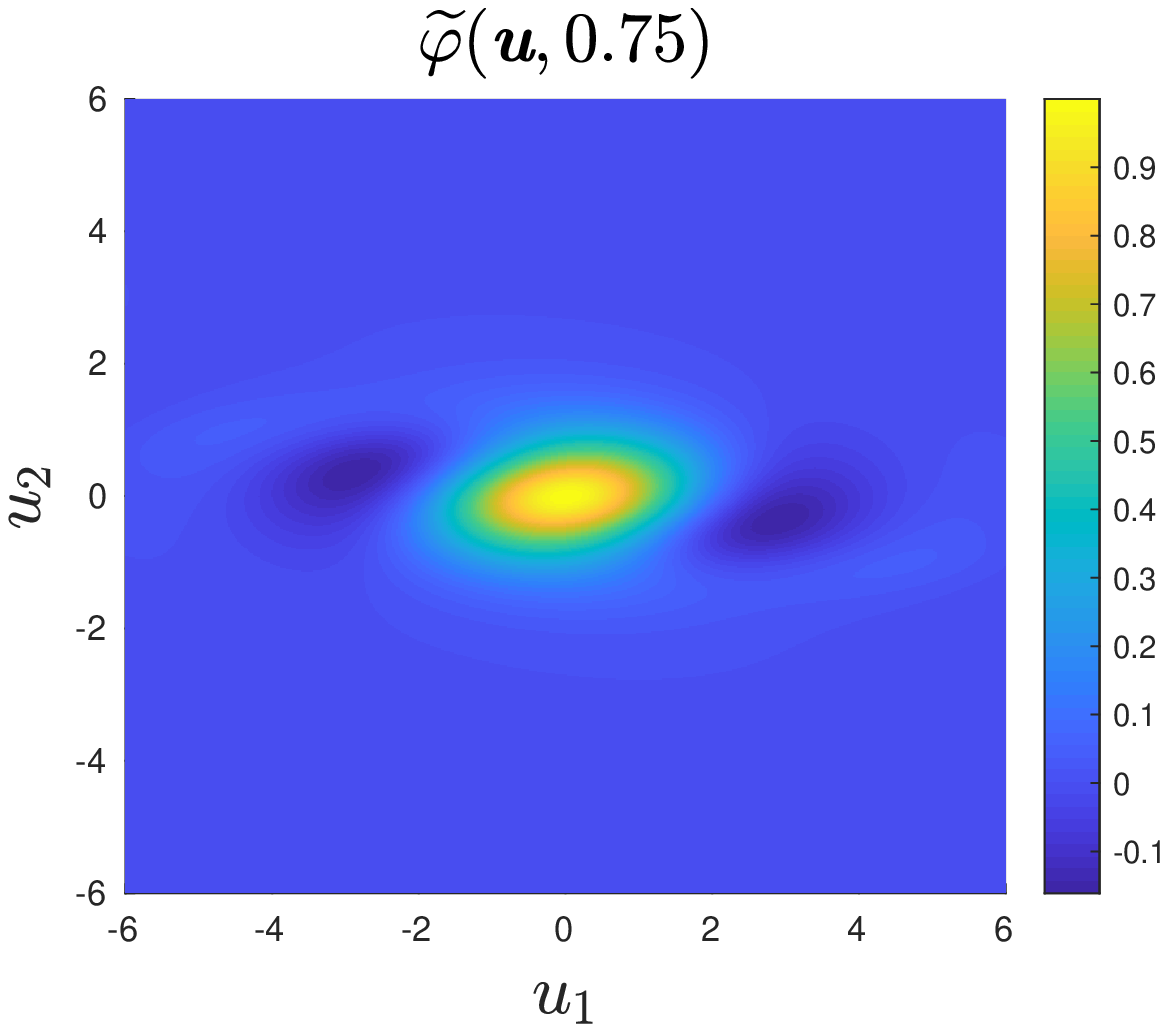}
		\caption{Comparison between the chf $\varphi^{MC}(\boldsymbol{u},t)$ obtained from Monte Carlo samples (left) and the neural network approximation $\widetilde{\varphi}(\boldsymbol{u},t)$ (right) for Section~\ref{subsubsec:DuffGWNCharFun} at $t=0.75$.}
		\label{fig:DuffGWNCHFTime2}
\end{figure}

Figures~\ref{fig:DuffGWNCHFTime1} and~\ref{fig:DuffGWNCHFTime2} illustrate the neural network approximation $\widetilde{\varphi}(\boldsymbol{u},t)$ at $t=0.25$ and $t=0.75$, respectively. In each figure, $\widetilde{\varphi}(\boldsymbol{u},t)$ (right subplot) is compared to $\varphi^{MC}(\boldsymbol{u},t)$ (left subplot) which is an estimate of the chf of $(X_1(t),X_2(t))$ based on Monte Carlo simulation.  The left and right panels of each figure support the observation that $\widetilde{\varphi}(\boldsymbol{u},t)$ adequately represents the target chf $\varphi(\boldsymbol{u},t)$.

To reconcile this approach with the approach based on the Fokker-Planck equation in Section~\ref{subsubsec:DuffGWNFP}, we apply the Fourier transform to $\widetilde{\varphi}(\boldsymbol{u},t)$ to determine an estimate of the pdf of $X_1(t)$, i.e. $\widetilde{f}_1(x_1,t) = \frac{1}{2\pi} \int_{-6}^6 e^{-iu_1 x_1} \widetilde{\varphi}(u_1,0,t) \,du_1$. Figure~\ref{fig:DuffGWNchfX1hist} presents plots of $\widetilde{f}_1(x_1,t)$ with a histogram of samples of $X_1(t)$ for $t=0.25$ (left) and $t = 0.75$ (right). The plots in Figures~\ref{fig:DuffGWNchfX1hist} and~\ref{fig:DuffGWNFPX1Hist} confirm that solving the PDE for the chf to obtain the pdf of $X(t)$ offers an alternative approach that is consistent with solving the Fokker-Planck equation as we expected.

In summary, Section~\ref{subsec:DuffingFPvsCHF} contrasted two approaches to construct a neural network representation for the pdf. The disadvantage of the approach elaborated in Section~\ref{subsubsec:DuffGWNFP} is that the loss value for $\widetilde{v}(\boldsymbol{x},t)$ may not be indicative of the accuracy of $\widetilde{f}(\boldsymbol{x},t)$. Although the approach in Section~\ref{subsubsec:DuffGWNCharFun} does not possess such challenges, automatic differentiation would have to be invoked for partial derivatives of order greater than 2 for the terms appearing in the chf PDE. In our experience, this meant a slower calculation of the loss function for the neural network during the training process.

\begin{figure}[h!]
		\centering
		\includegraphics[width = 0.47\textwidth]		
						{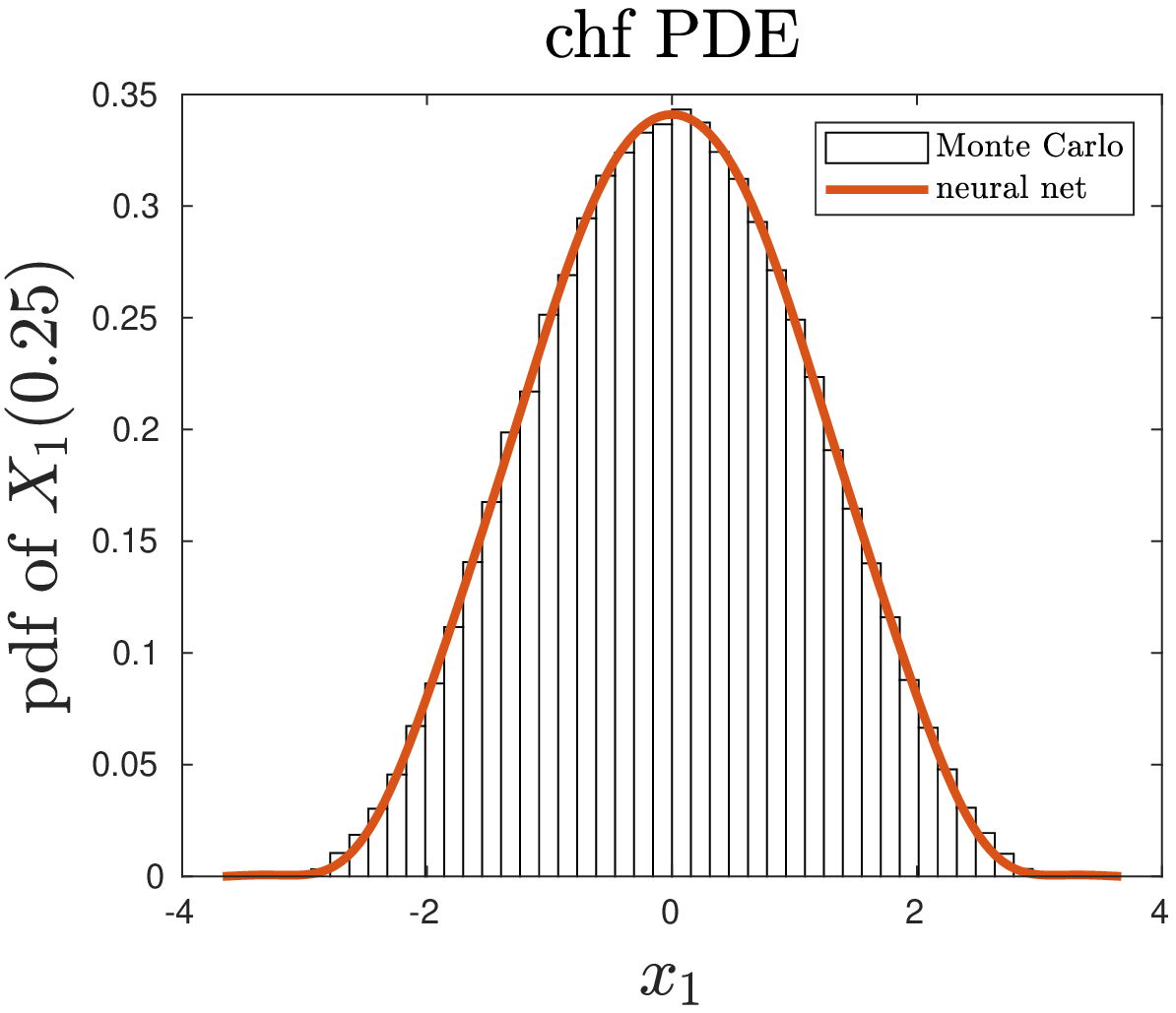} 
						\hspace{3em}
		\includegraphics[width = 0.44\textwidth]		
						{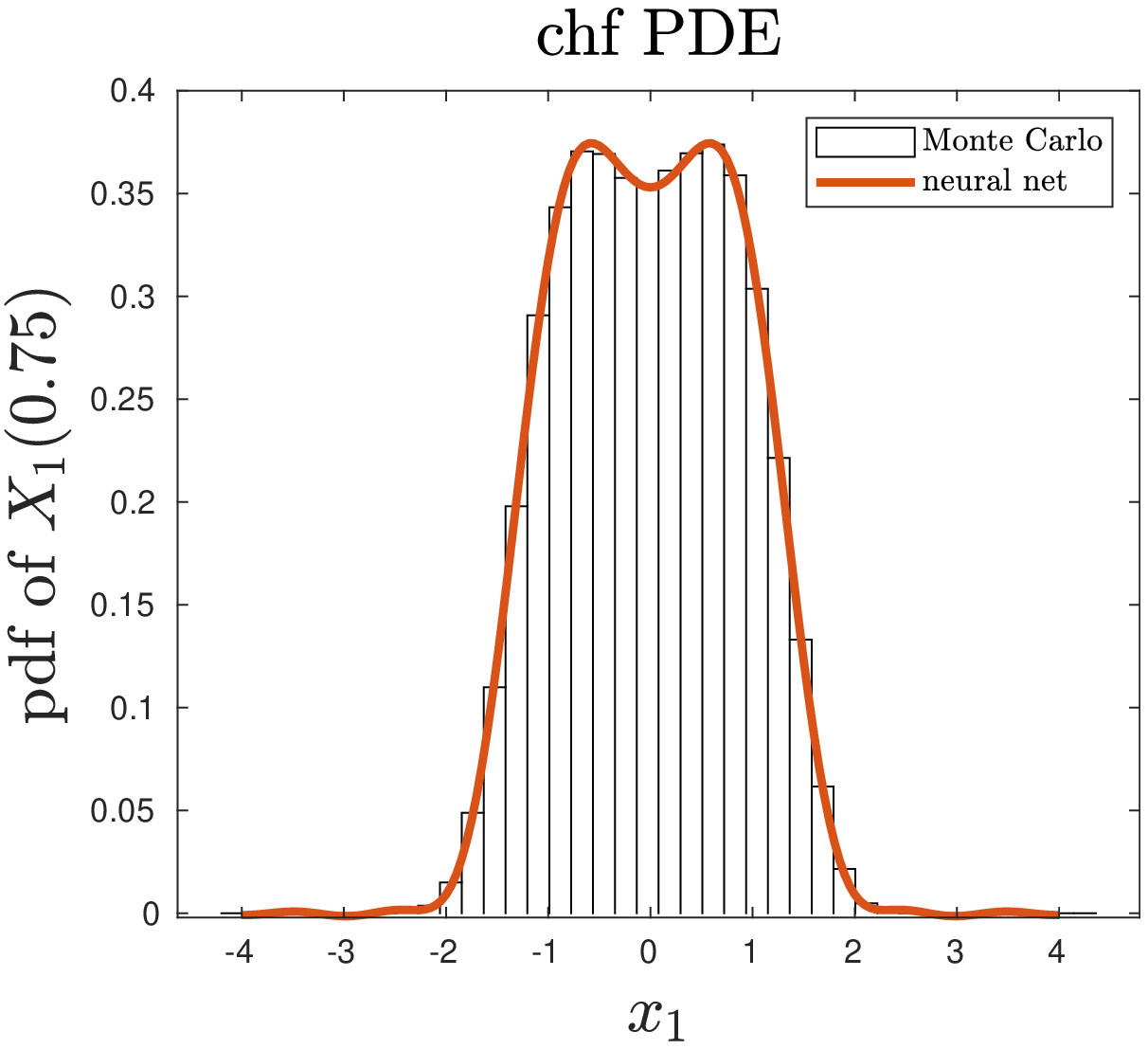}
		\caption{Comparison between the neural network approximation $\widetilde{f}_1(x_1,t)$ to the pdf of $X_1(t)$ and the histogram based on Monte Carlo samples of $X_1(t)$ for Section~\ref{subsubsec:DuffGWNCharFun} at $t=0.25$ (left) and $t = 0.75$ (right). }
		\label{fig:DuffGWNchfX1hist}
\end{figure}

\subsection{Duffing oscillator with Poisson white noise} \label{subsec:DuffingPWNvsBM}


We revisit the Duffing oscillator introduced in Section~\ref{subsec:DuffingFPvsCHF} where instead, $W(t)$ in~\eqref{eq:DuffOscillGWNSDE} is Poisson white noise~\cite{book:Grigoriu2002}, i.e. $W(t)$ is the formal derivative $\frac{dC^{\lambda}(t)}{dt}$ where $C^{\lambda}(t) = \sum_{k=1}^{N(t)}Y_k$ is a compound Poisson process described by a Poisson process $N(t)$ with intensity $\lambda$ and  random variables $Y_k$ which are independent copies of $Y$. For distinction, denote by $(X_1^{\lambda}(t),X_2^{\lambda}(t))$ the state variables of~\eqref{eq:DuffOscillGWNSDE} if $W(t)$ is Poisson white noise while $(X_1(t),X_2(t))$ refers to the state due to Gaussian white noise as in Section~\ref{subsec:DuffingFPvsCHF}. It will be numerically confirmed using neural network approximations for the chf of $(X^{\lambda}_1(t),X^{\lambda}_2(t))$ that as $\lambda \rightarrow \infty$, the distribution of $(X^{\lambda}_1(t),X^{\lambda}_2(t))$ approximates that of $(X_1(t),X_2(t))$.

We utilize the same parameters and initial condition as in Section~\ref{subsec:DuffingFPvsCHF} and set $Y \sim N(0,\sigma^2)$. The values for $\lambda, \sigma^2$ are chosen such that $\pi g_0 = \lambda E[Y^2]$ to ensure that the second moment properties of the random forcing in Section~\ref{subsec:DuffingFPvsCHF}, $\sqrt{\pi g_0} B(t)$, and in this section, $C^{\lambda}(t)$, are identical.
It was demonstrated analytically and numerically in~\cite{paper:Grigoriu2009} that under some conditions on $Y_k$, $(X_1^{\lambda}(t),X_2^{\lambda}(t)) \rightarrow (X_1(t),X_2(t))$ in probability as $\lambda \rightarrow \infty$ if the system~\eqref{eq:DuffingGWNsystem} is linear with additive noise. As~\eqref{eq:DuffingGWNsystem} is nonlinear, the following heuristic can be used to illustrate that  $(X_1^{\lambda}(t),X_2^{\lambda}(t))$ still converges in probability to $(X_1(t),X_2(t))$. Consider the system
\begin{align*}
d \boldsymbol{X}(t) & = \boldsymbol{a}(\boldsymbol{X}(t)) \,dt + \boldsymbol{b}  \sqrt{\pi g_0} \,dB(t), \quad \text{and} \\
d \boldsymbol{X}^{\lambda}(t) & = \boldsymbol{a}(\boldsymbol{X}^{\lambda}(t)) \,dt + \boldsymbol{b} \,dC^{\lambda}(t)
\end{align*}
where
$\boldsymbol{X}(t) = (X_1(t),X_2(t)), \boldsymbol{X}^{\lambda}(t) = (X_1^{\lambda}(t),X_2^{\lambda}(t))$, $\boldsymbol{a}(x_1,x_2) = 
\begin{bmatrix}
x_2 \\
-\nu^2 (x_1 + \alpha x_1^3) - 2\zeta \nu x_2
\end{bmatrix} 
$ and $\boldsymbol{b} = \begin{bmatrix}
0\\1
\end{bmatrix}
$. Since for every $t \ge 0$, $C^{\lambda}(t)$ converges in probability to $\sqrt{\pi g_0} B(t)$ as $\lambda \rightarrow \infty$ \cite{book:Skorohod1982}, so do the increments  $dC^{\lambda}(t)$ and $\sqrt{\pi g_0} \, dB(t)$. The convergence of $\boldsymbol{X}^{\lambda}(t)$ to $\boldsymbol{X}(t)$ follows because they depend continuously on $dC^{\lambda}(t)$ and $\sqrt{\pi g_0} \, dB(t)$, respectively. We  now verify this qualitatively by visually inspecting sample paths of the state variables in the following figures. Figure~\ref{fig:DuffGWNSamplePaths} corresponds to $(X_1(t),X_2(t))$ in Section~\ref{subsec:DuffingFPvsCHF} while Figures~\ref{fig:DuffPWNSamplePathAlmostBM} and~\ref{fig:DuffPWNSamplePathSmallInt} correspond to $(X_1^{\lambda}(t),X_2^{\lambda}(t))$ for $E[Y^2] = 0.01$ and $E[Y^2] = 3$, respectively, with $\lambda = \frac{\pi g_0}{E[Y^2]}$. The figures certify that for larger values of $\lambda$, the sample paths of $(X_1(t),X_2(t))$
and $(X_1^{\lambda}(t),X_2^{\lambda}(t))$ are almost indistinguishable, especially in the second state variable.


\begin{figure}[h!]
		\centering
		\includegraphics[width = 0.46\textwidth]		
						{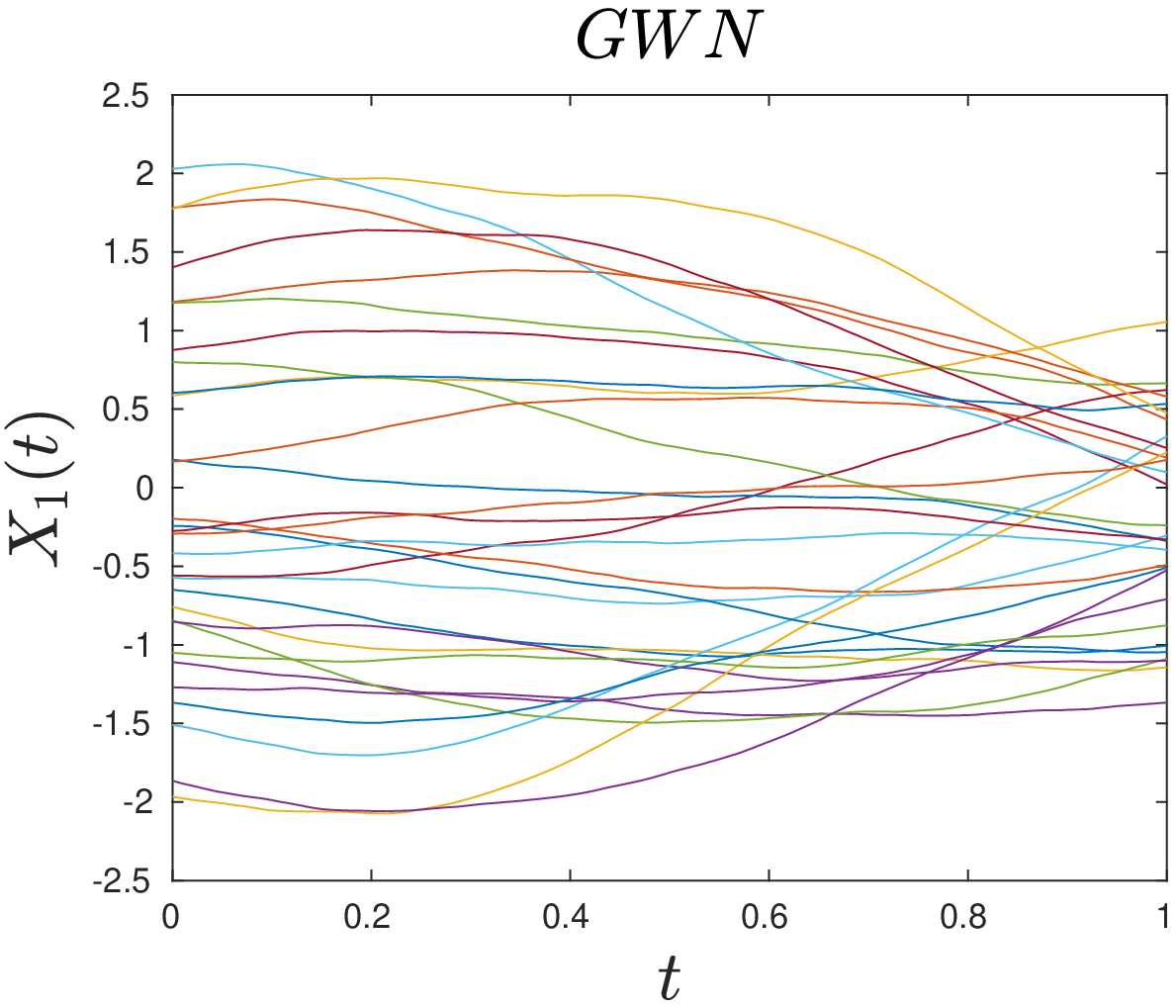} 
						\hspace{1.3em}
		\includegraphics[width = 0.45\textwidth]		
						{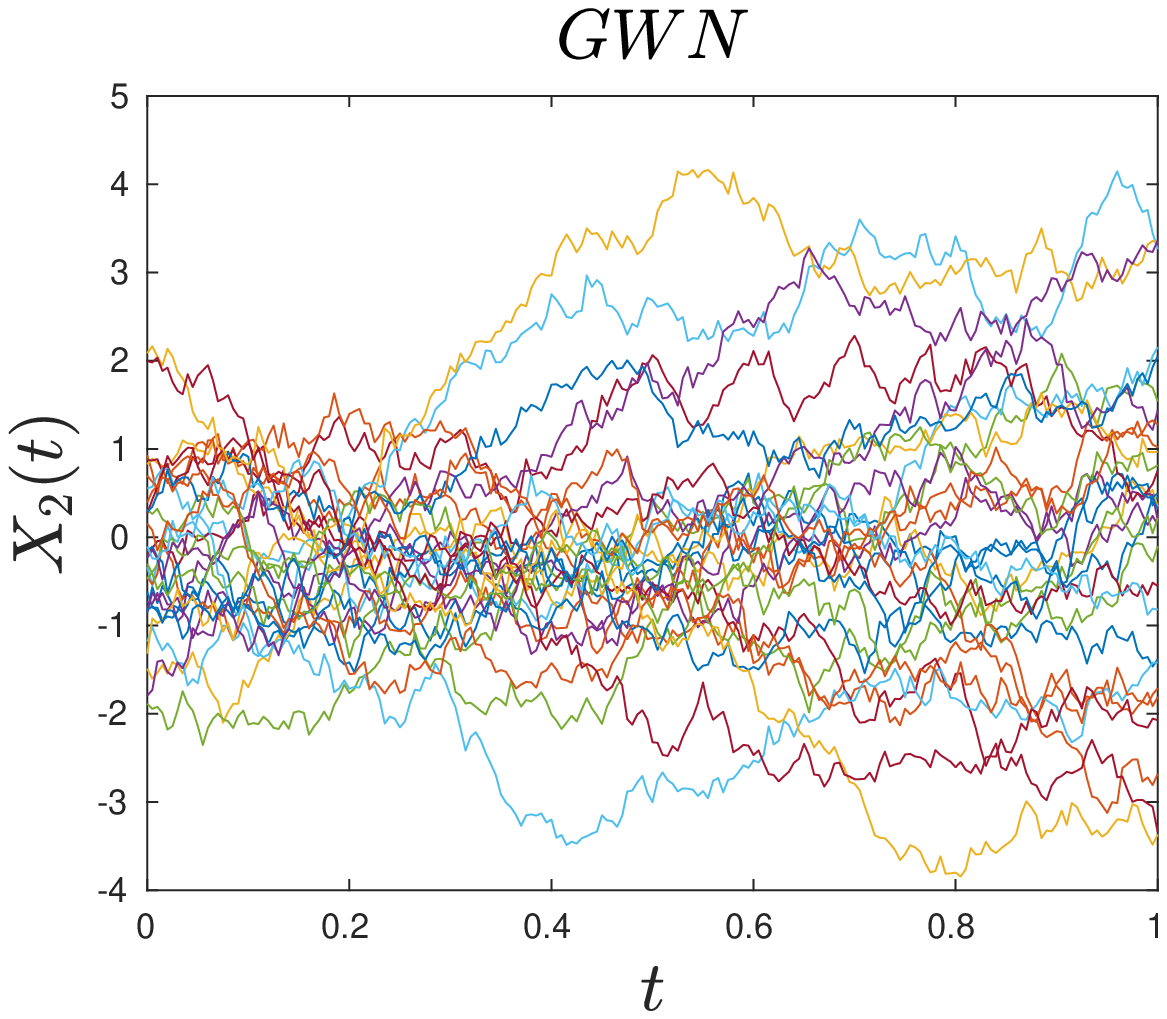}
		\caption{Sample paths of $X_1(t)$ (left) and $X_2(t)$ (right) for the Duffing oscillator subject to Gaussian white noise in Section~\ref{subsec:DuffingFPvsCHF}.}
		\label{fig:DuffGWNSamplePaths}
\end{figure}


\begin{figure}[h!]
		\centering
		\includegraphics[width = 0.5\textwidth]		
						{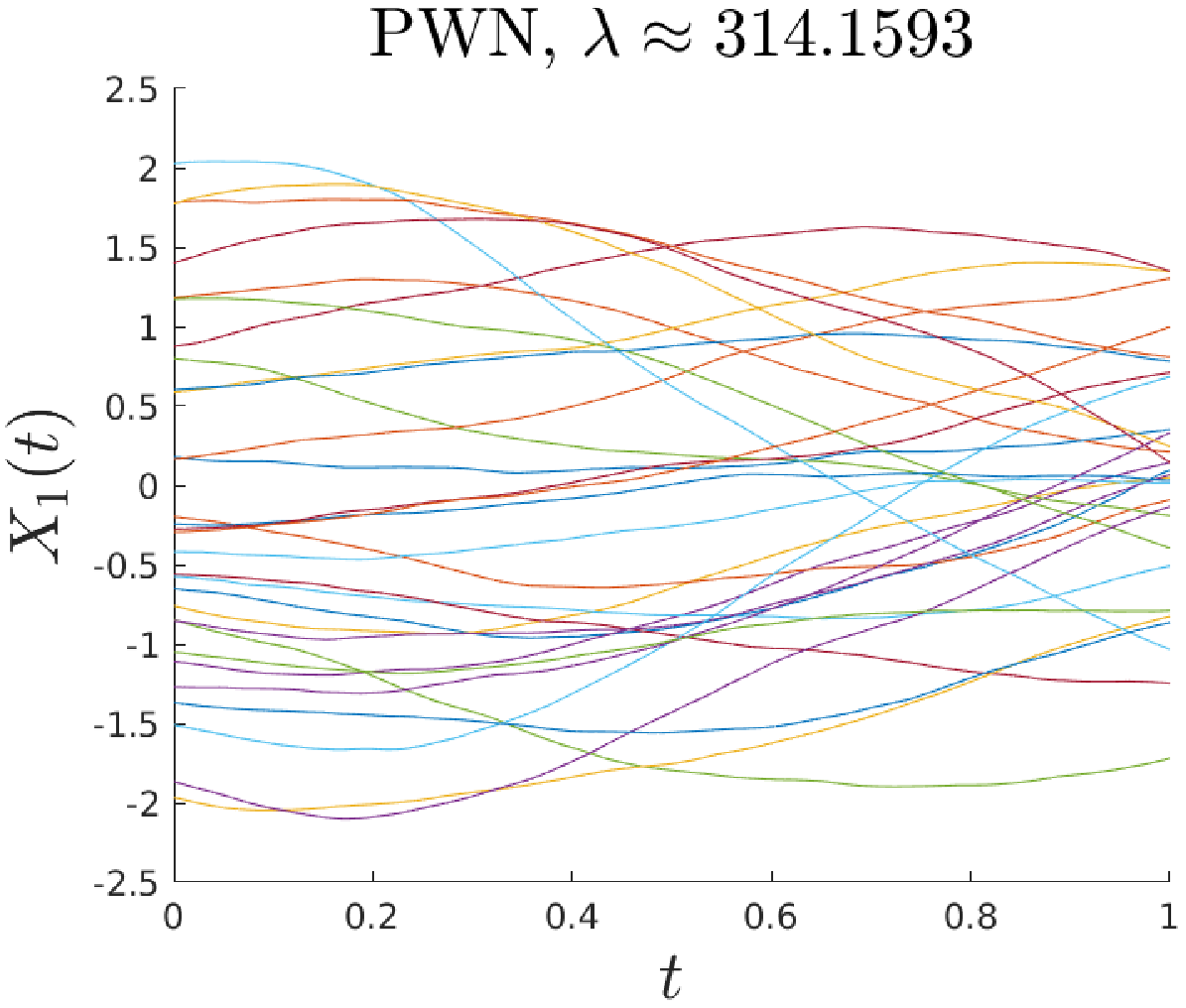} 
						\hspace{-0.75em}
		\includegraphics[width = 0.5\textwidth]		
						{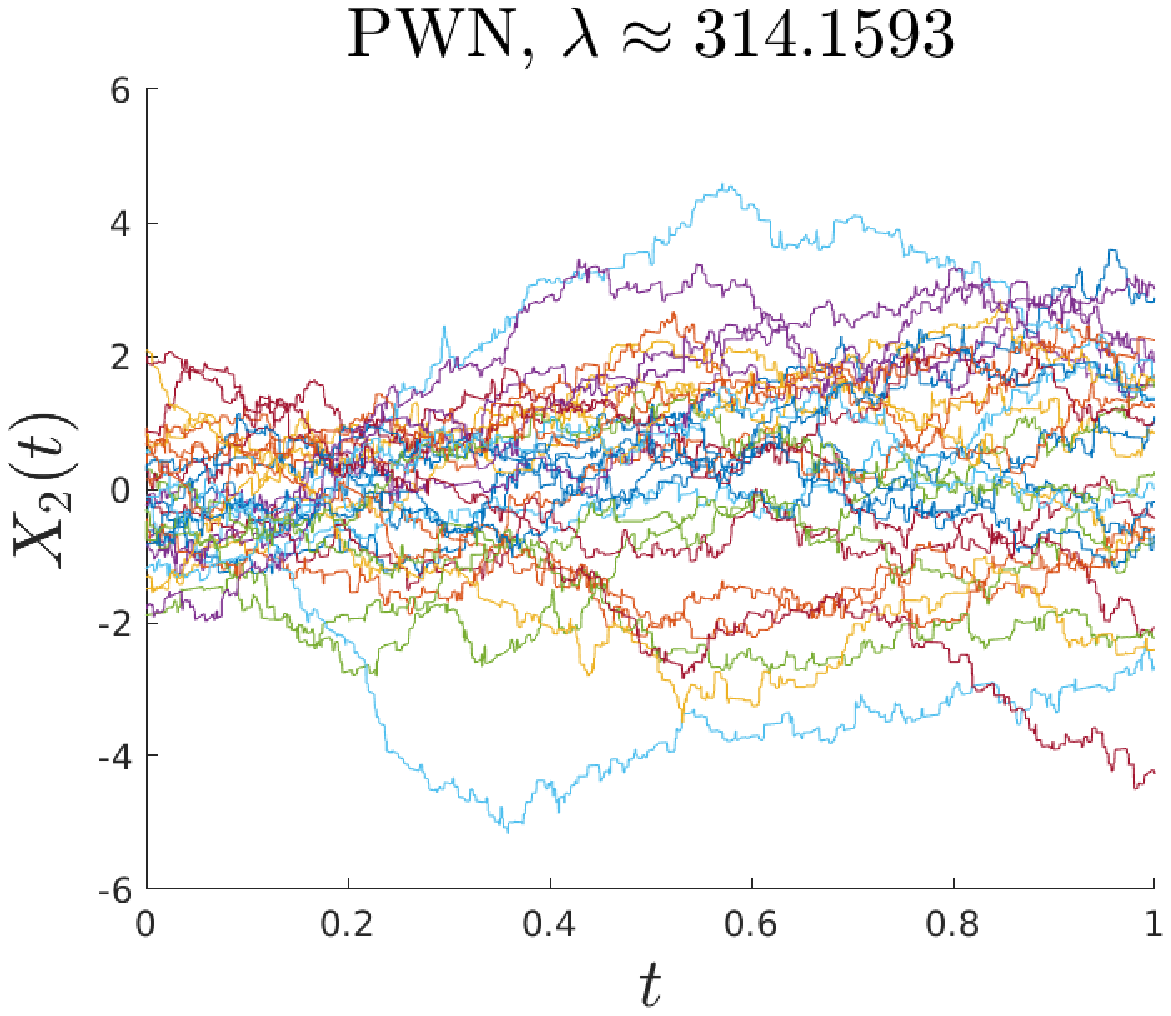}
		\caption{Sample paths of $X^{\lambda}_1(t)$ (left) and $X^{\lambda}_2(t)$ (right) for the Duffing oscillator subject to Poisson white noise in Section~\ref{subsec:DuffingPWNvsBM} with $E[Y^2] = 0.01$.}
		\label{fig:DuffPWNSamplePathAlmostBM}
\end{figure}


\begin{figure}[h!]
		\centering
		\includegraphics[width = 0.45\textwidth]		
						{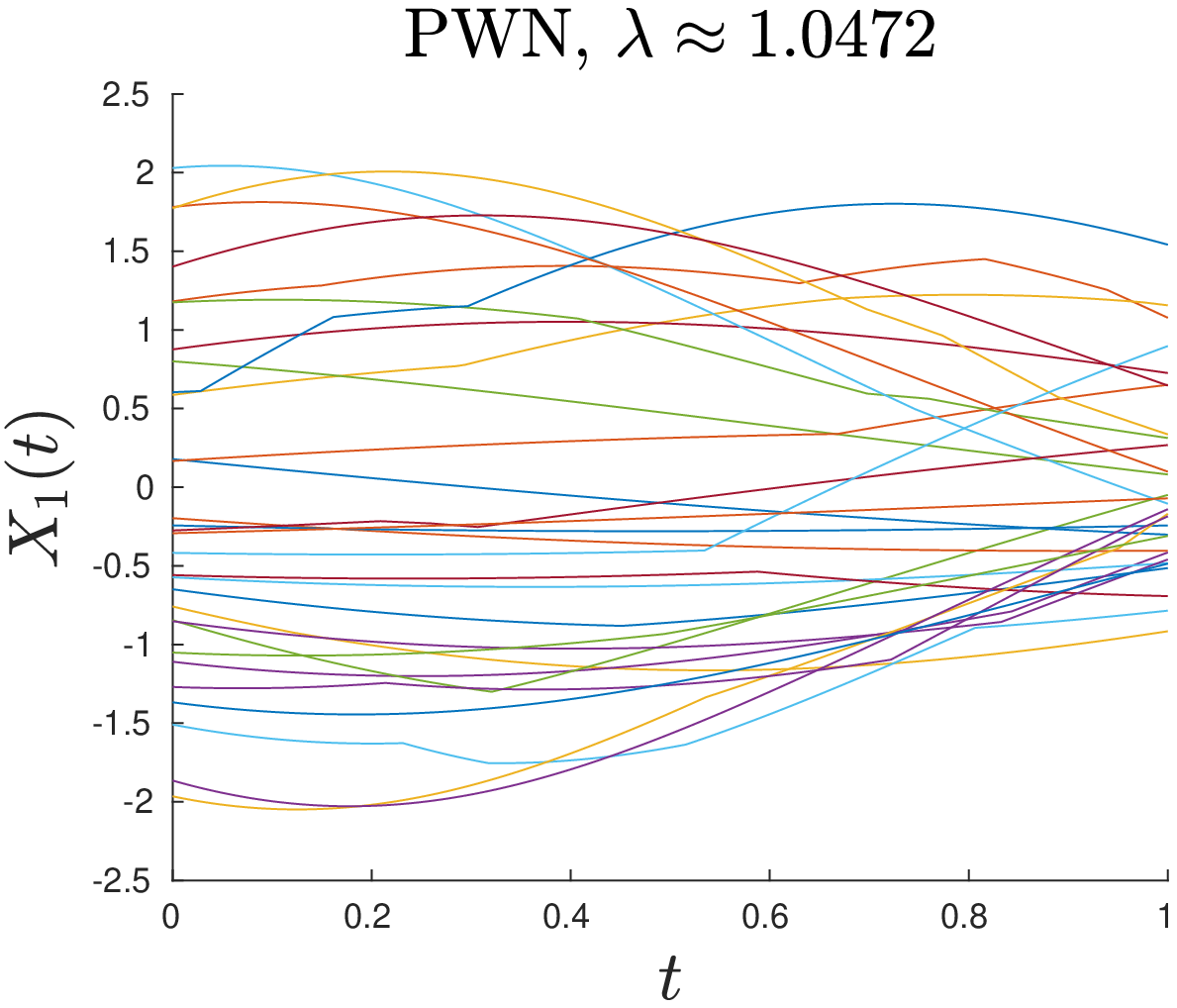} 
						\hspace{2em}
		\includegraphics[width = 0.43\textwidth]		
						{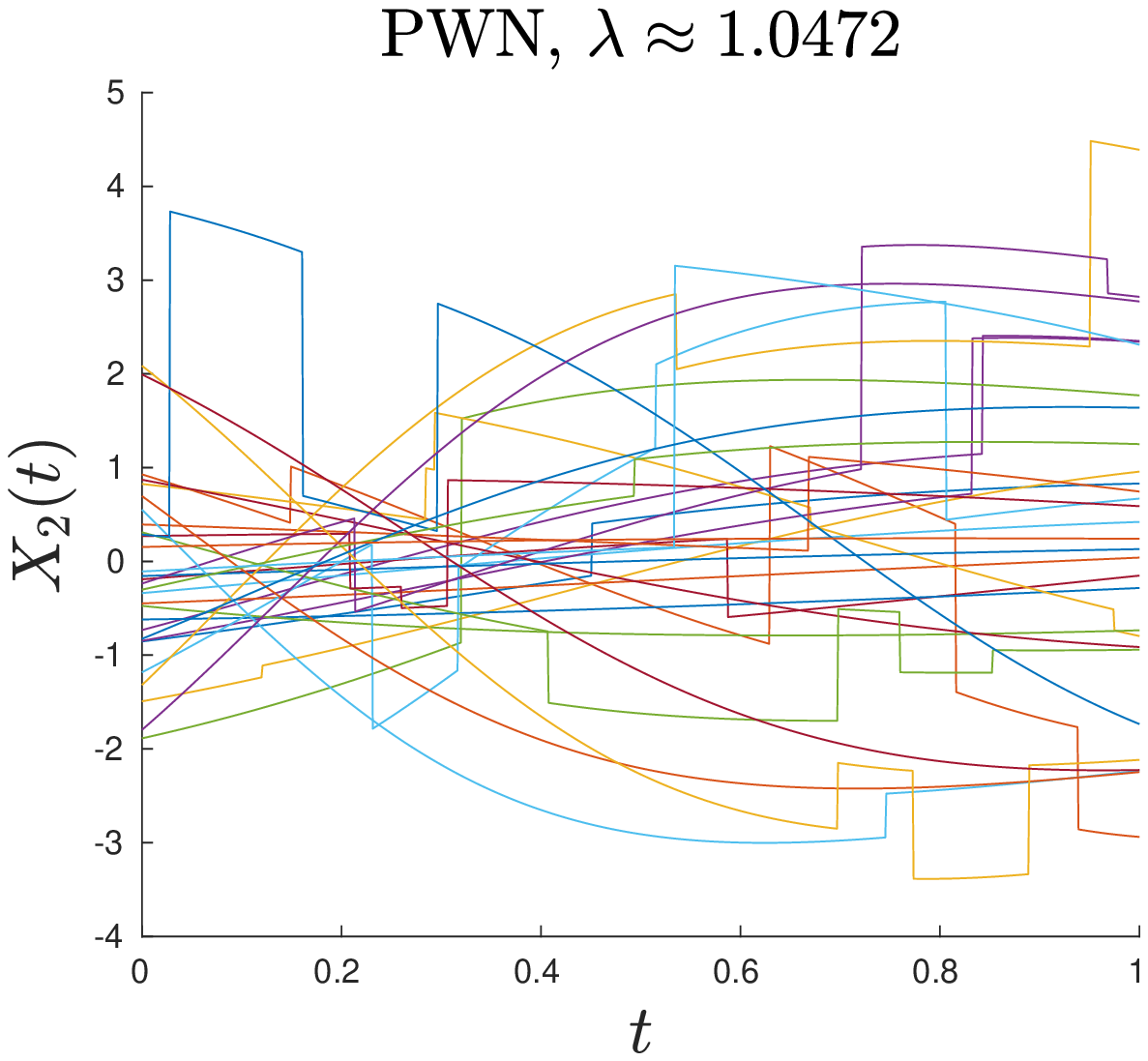}
		\caption{Sample paths of $X^{\lambda}_1(t)$ (left) and $X^{\lambda}_2(t)$ (right) for the Duffing oscillator subject to Poisson white noise in Section~\ref{subsec:DuffingPWNvsBM} with $E[Y^2] = 3$.}
		\label{fig:DuffPWNSamplePathSmallInt}
\end{figure}

Our objective is then to reproduce the same conclusion by investigating the characteristic function of $(X_1^{\lambda}(t),X_2^{\lambda}(t))$. From~\eqref{eq:DuffingCHFGeneralPDE}, $\varphi(\boldsymbol{u},t)$ satisfies the PDE 
\begin{align} \label{eq:DuffPWNCHFPDE}
\mathcal{Q}[\varphi(\boldsymbol{u},t)] = \frac{\partial \varphi(\boldsymbol{u},t)}{\partial t} - (u_1 + 2 \zeta \nu u_2 ) \frac{\partial \varphi(\boldsymbol{u},t)}{\partial u_2} - \nu^2 u_2 \frac{\partial \varphi(\boldsymbol{u},t)}{\partial u_1} + & \nu^2 \alpha u_2 \frac{\partial^3 \varphi(\boldsymbol{u},t)}{\partial u_1^3}  \notag\\ & -\lambda \varphi(\boldsymbol{u},t)  \,[\phi (u_2)-1] = 0
\end{align}
where $\phi(\cdot)$ is the chf of $Y$. Similar calculations as above highlight that $\varphi(\boldsymbol{u},t)$ is real-valued since $Y$ and $-Y$ and consequently, $C^{\lambda}(t)$ and $-C^{\lambda}(t)$ have the same distribution.

We sought a neural network approximation $\widetilde{\varphi}(\boldsymbol{u},t)$ on the truncated domain $\boldsymbol{u} \in [-6,6]^2$ using an architecture comprised of an input layer with 3 neurons, an output layer with 1 neuron, and 5 hidden  layers with 50 neurons each. We generated exactly the same set of collocation points as in Section~\ref{subsubsec:DuffGWNCharFun} to evaluate the loss function. Figures~\ref{fig:DuffPWNCHFAlmostBM} and~\ref{fig:DuffPWNCHFSmallInt} depict $\widetilde{\varphi}(\boldsymbol{u},t)$ for the case in which $E[Y^2] = 0.01 \,\, (\lambda \approx 314.1593)$ and $E[Y^2] = 3  \,\,(\lambda \approx 1.0472)$, respectively, at times $t = 0.25,0.75$. The loss value for the trained neural network associated with former scenario was $3.274475\times 10^{-5}$ while that of the latter was $1.2586042\times 10^{-5}$. It is clear from scrutinizing the plots in Figure~\ref{fig:DuffPWNCHFAlmostBM} and that of Figures~\ref{fig:DuffGWNCHFTime1} and~\ref{fig:DuffGWNCHFTime2} that the chf of $(X_1^{\lambda}(t),X_2^{\lambda}(t))$ for large $\lambda$ appears identical to that of $(X_1(t),X_2(t))$ subject to Gaussian white noise. In contrast, the plots in Figure~\ref{fig:DuffPWNCHFSmallInt}  differ from the previously mentioned figures which was expected owing to the asymptotic results discussed above.

\begin{figure}[h!]
		\centering
		\includegraphics[width = 0.45\textwidth]		
						{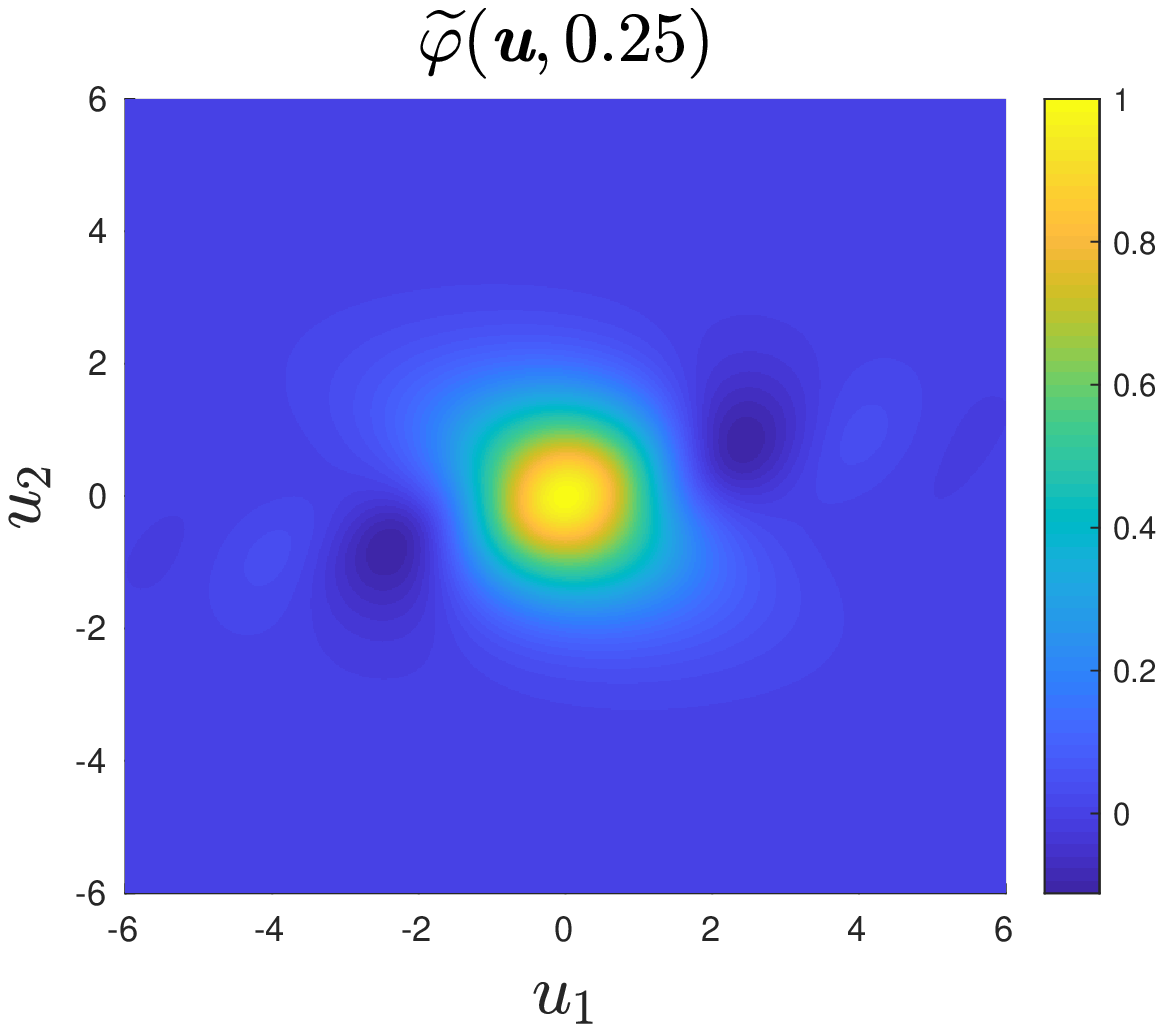} 
						\hspace{3em}
		\includegraphics[width = 0.45\textwidth]		
						{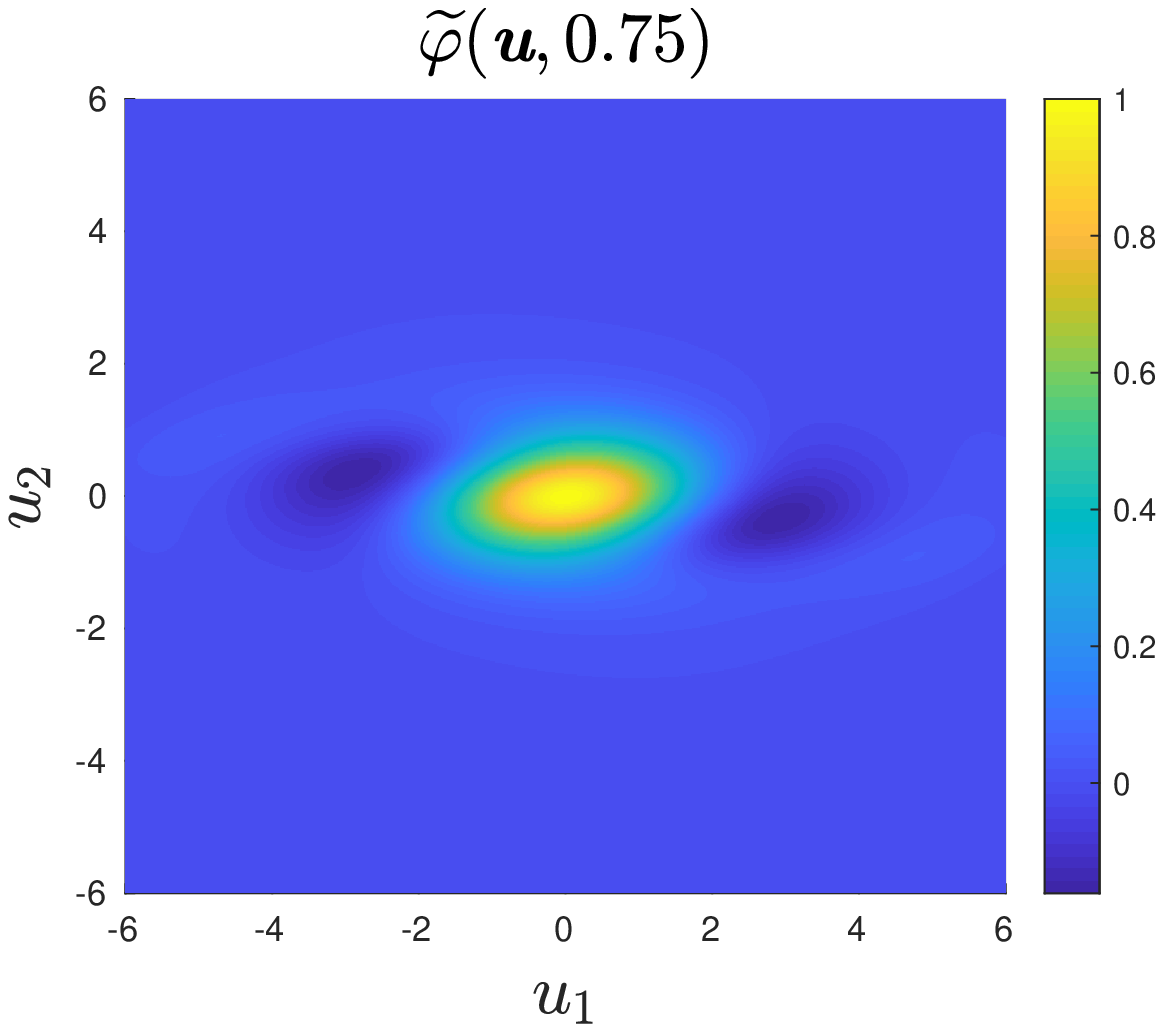}
		\caption{Plots of the neural network approximation $\widetilde{\varphi}(\boldsymbol{u},t)$ for Section~\ref{subsec:DuffingPWNvsBM} with $E[Y_k^2] = 0.01$ at $t=0.25$ (left) and $t=0.75$ (right).}
		\label{fig:DuffPWNCHFAlmostBM}
\end{figure}

\begin{figure}[h!]
		\centering
		\includegraphics[width = 0.43\textwidth]		
						{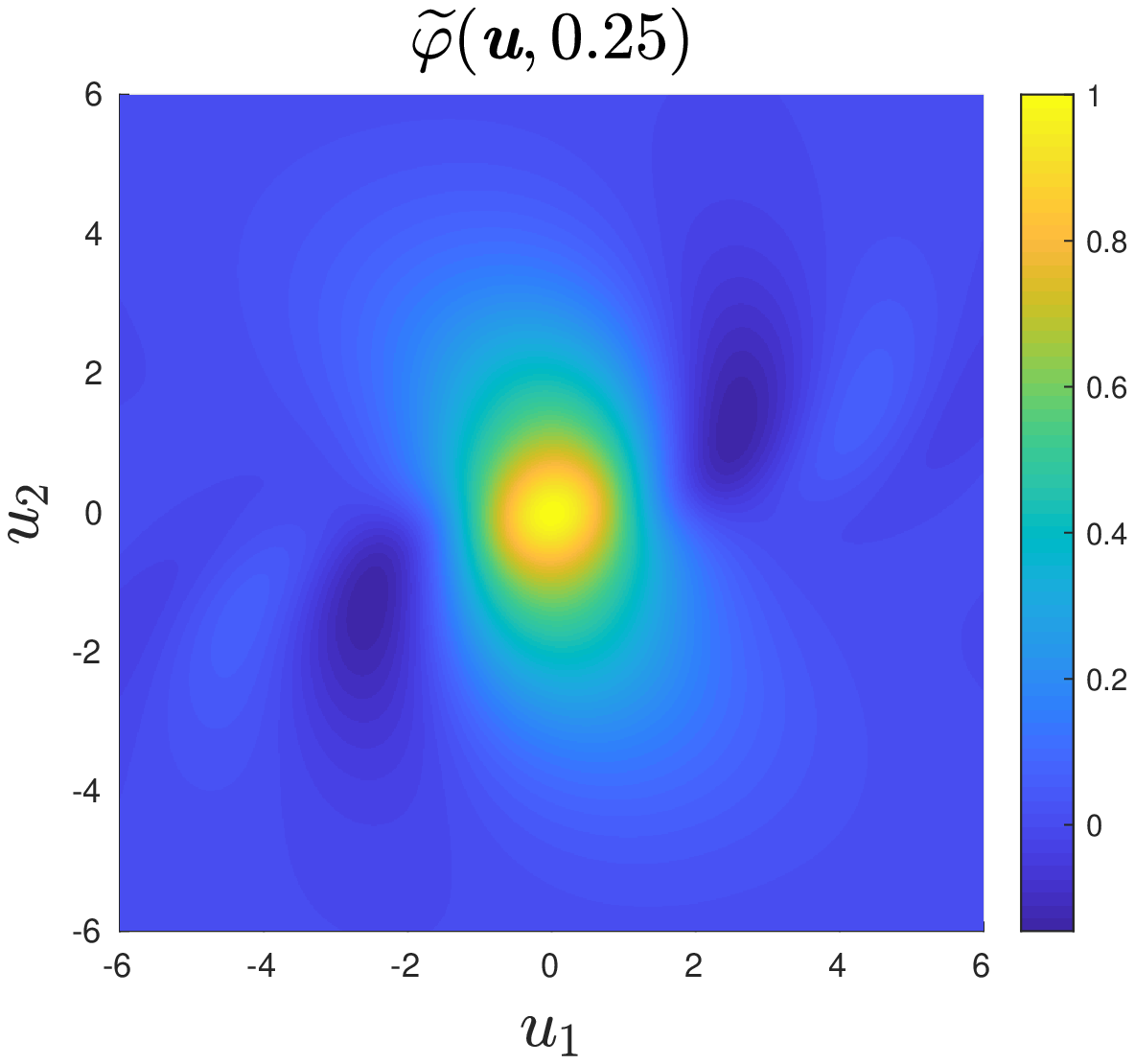} 
						\hspace{3em}
		\includegraphics[width = 0.46\textwidth]		
						{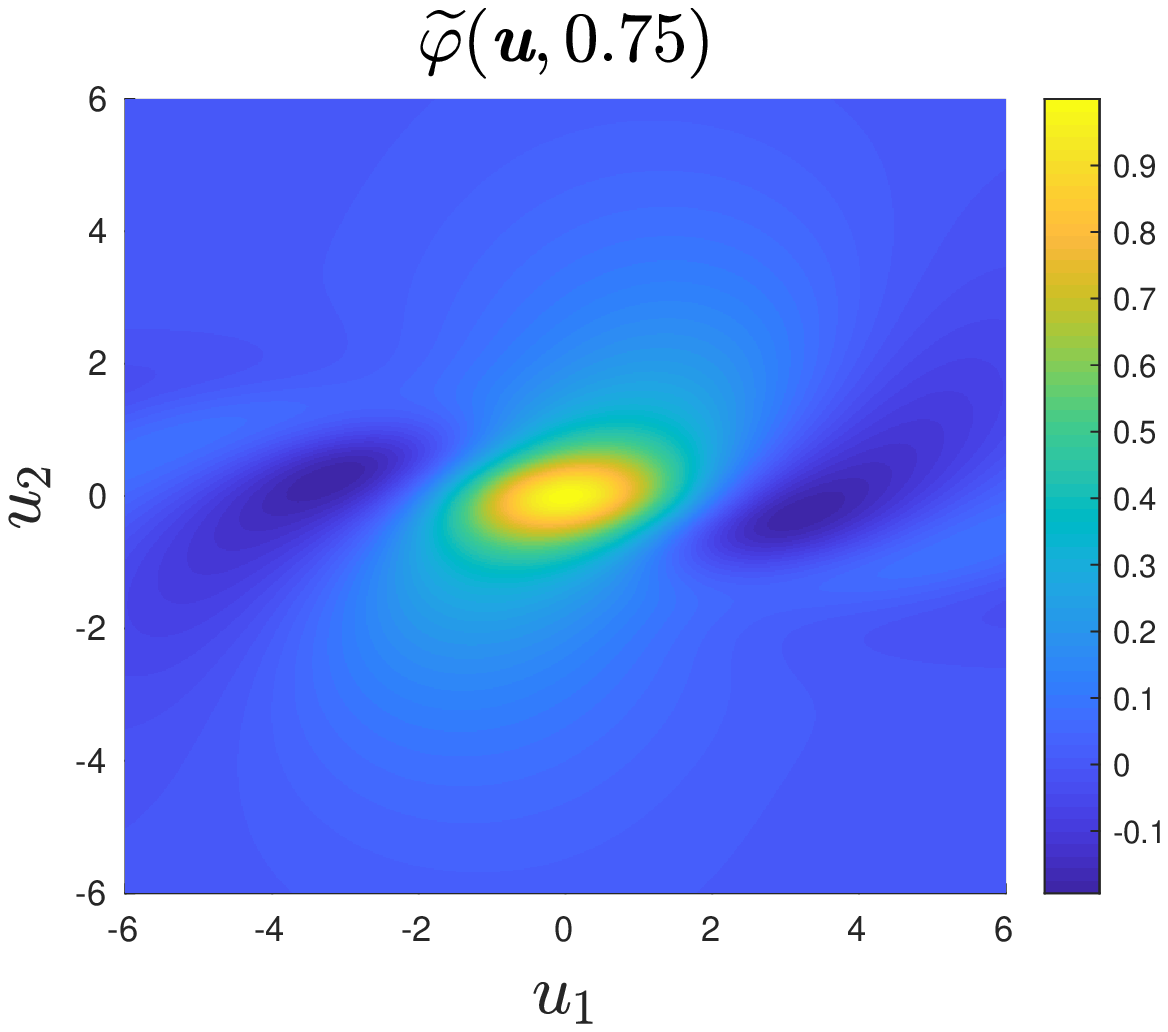}
		\caption{Plots of the neural network approximation $\widetilde{\varphi}(\boldsymbol{u},t)$ for Section~\ref{subsec:DuffingPWNvsBM} with $E[Y_k^2] = 3$ at $t=0.25$ (left) and $t=0.75$ (right).}
		\label{fig:DuffPWNCHFSmallInt}
\end{figure}

\subsection{Characteristic function for a 3-dimensional state vector} \label{subsec:3Dexample}


Suppose that $X(t)$, $t \in [0,1]$, represents the displacement of a damped harmonic oscillator in which the external forcing is a non-Gaussian process with continuous samples. Let $X(t)$ satisfy the SDE
\begin{align} \label{eq:3DSDEEx}
& \ddot{X}(t) + \beta \dot{X}(t) + \nu^2 X(t) = S(t)^3, \notag \\
& dS(t) = -\alpha S(t) \,dt + \sigma \sqrt{2 \alpha} \, dB(t)
\end{align}
or in system form,
\begin{align} \label{eq:3SSDEExSystemForm}
d \begin{bmatrix}
X_1(t) \\ X_2(t) \\ X_3(t)
\end{bmatrix}
 = 
\begin{bmatrix}
X_2(t) \\
-(\nu^2 + \beta) X_1(t)  + S(t)^3 \\
-\alpha S(t)
\end{bmatrix} \, dt +
\begin{bmatrix}
0\\
0 \\
\sigma \sqrt{2 \alpha}
\end{bmatrix} \, dB(t)
\end{align}
where $B(t)$ is the Brownian motion. It will be demonstrated that even when the frequency domain is 3-dimensional, the neural network approximation to the chf of $\boldsymbol{X}(t) = [X_1(t), X_2(t), X_3(t)]^T $ can adequately match the Monte Carlo solution.

For our simulation, we consider $\boldsymbol{X}(0) \sim N(\boldsymbol{0}_{3 \times 1},\boldsymbol{I}_{3 \times 3})$ with $\boldsymbol{I}$ denoting the $3 \times 3$ identity matrix so that $\varphi(\boldsymbol{u},0) = \exp \left( -\frac{1}{2} \boldsymbol{u}' \boldsymbol{u} \right), \boldsymbol{u} = (u_1,u_2,u_3)$. The parameters we set are $\beta = 0.5, \nu = 3, \alpha = 0.12$. The target chf $\varphi^{MC}(\boldsymbol{u},t)$ was estimated through 200000 Monte Carlo samples of $\boldsymbol{X}(t)$ produced using forward Euler with time step $0.005$.

According to~\eqref{eq:CharFunPDE}, the PDE satisfied by $  \varphi(\boldsymbol{u},t)$ is \begin{align} \label{eq:FP3Dex}
\mathcal{Q}[\varphi(\boldsymbol{u},t)] =
\frac{\partial \varphi (\boldsymbol{u},t)}{\partial t} - (u_1 - \beta u_2) \frac{\partial \varphi (\boldsymbol{u},t)}{\partial u_2} + \nu^2 u_2 \frac{\partial \varphi (\boldsymbol{u},t)}{\partial u_1} & + \alpha u_3 \frac{\partial \varphi (\boldsymbol{u},t)}{\partial u_3} \notag \\ 
& + u_2  \frac{\partial^3 \varphi(\boldsymbol{u},t)}{\partial u_3^3} + \sigma^2 \alpha u_3^2 \varphi( \boldsymbol{u},t) = 0
\end{align}
whose solution we approximate on the truncated domain $(u_1,u_2,u_3) \in [-5.5,5.5]^2 \times [-3,3].$ The architecture we employed included an input layer with 4 neurons, an output layer with 1 neuron, and 5 hidden layers with 50 neurons each. Following similar arguments above, the chf of $\boldsymbol{X}(t)$ is real-valued since both $\boldsymbol{X}(t)$ and $-\boldsymbol{X}(t)$
satisfy the dynamics~\eqref{eq:3SSDEExSystemForm}. To train the neural network, latin hypercube samples were simulated for the collocation points which constituted $N_{IC} = 2000$ points in the truncated domain of $\boldsymbol{u}$, $N_{0} =100$ points in $t \in [0,1]$, and $N_{Op} = 100000$ points in $(\boldsymbol{u},t) \in [-5.5,5.5]^2 \times [-3,3] \times [0,1].$ The trained network yielded a loss value of $7.966772 \times 10^{-6}$. 

We assess the performance of $\widetilde{\varphi}(\boldsymbol{u},t)$ with respect to the target Monte Carlo solution in Figures~\ref{fig:3Dtime1},~\ref{fig:3Dtime2}, and~\ref{fig:3Dtime3}. Each figure displays $\varphi^{MC}(\boldsymbol{u},t)$ (left panel) and $\widetilde{\varphi}(\boldsymbol{u},t)$ (right panel)  evaluated at three pairs of values of $(u_3,t)$. The plots suggest that $\widetilde{\varphi}(\boldsymbol{u},t)$ can sufficiently capture the target chf. Finally, we query $\widetilde{\varphi}(\boldsymbol{u},t)$ at $u_1=u_2=0$ and various values of $t$ to acertain that it can recover the analytical expression $\varphi(0,0,u_3,t) = e^{-\frac{1}{2} u_3^2}$ which is due to the fact that $S(t) \sim N(0,1)$. In Figure~\ref{fig:3DOU}, we compile plots of $\varphi(0,0,u_3,t)$ and $\widetilde{\varphi}(0,0,u_3,t)$ which are almost indistinguishable. This further supports the approximation quality of $\widetilde{\varphi}(\boldsymbol{u},t)$. 

\begin{figure}[h!]
		\centering
		\includegraphics[width = 0.46\textwidth]		
						{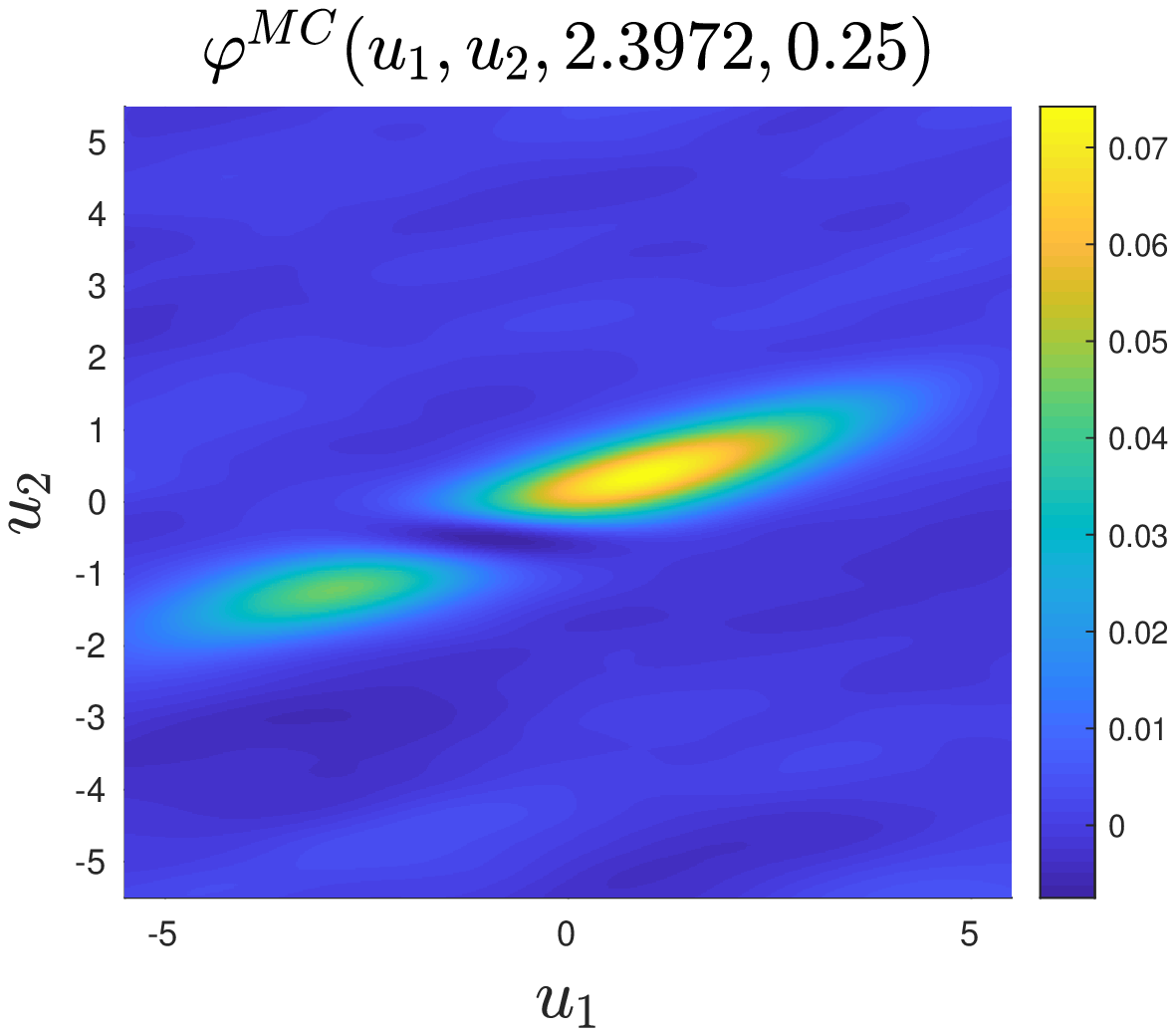} 
						\hspace{3em}
		\includegraphics[width = 0.435\textwidth]		
 						{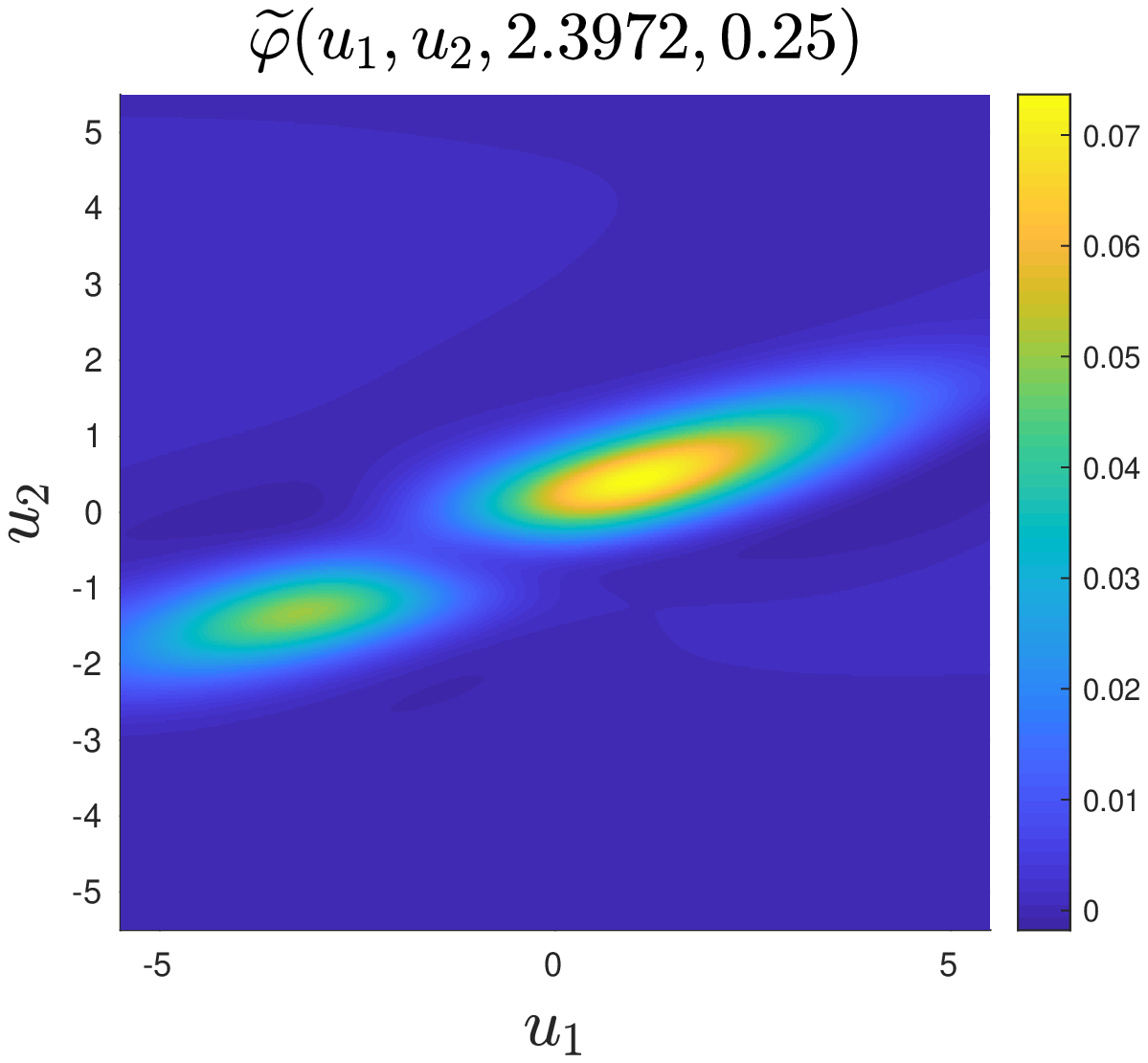}
		\caption{Comparison between the chf $\varphi^{MC}(\boldsymbol{u},t)$ obtained from Monte Carlo samples (left) and the neural network approximation $\widetilde{\varphi}(\boldsymbol{u},t)$ (right) for Section~\ref{subsec:3Dexample} at $u_3 = 2.3972, t=0.25$.}
		\label{fig:3Dtime1}
\end{figure}

\begin{figure}[h!]
		\centering
		\includegraphics[width = 0.47\textwidth]		
						{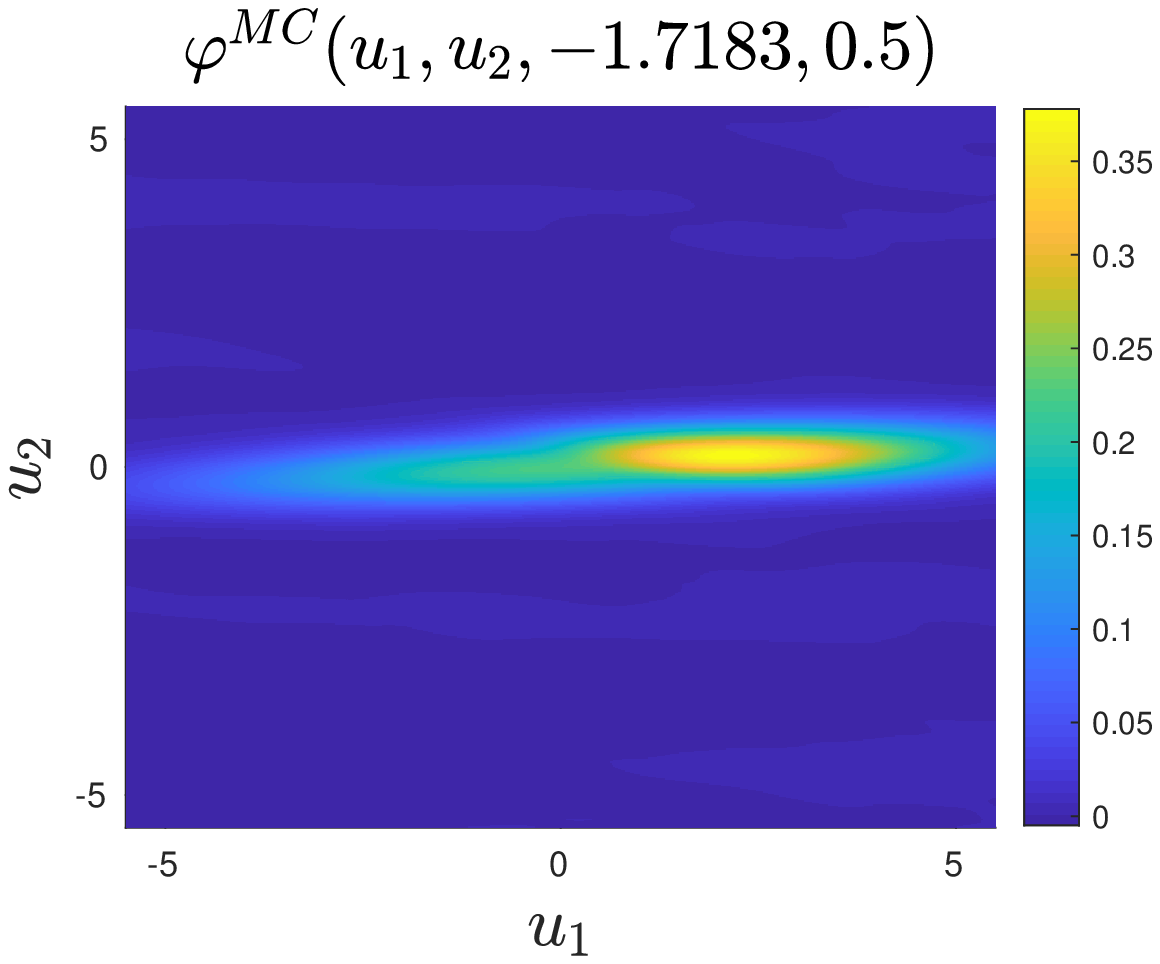} 
						\hspace{3em}
		\includegraphics[width = 0.43\textwidth]		
						{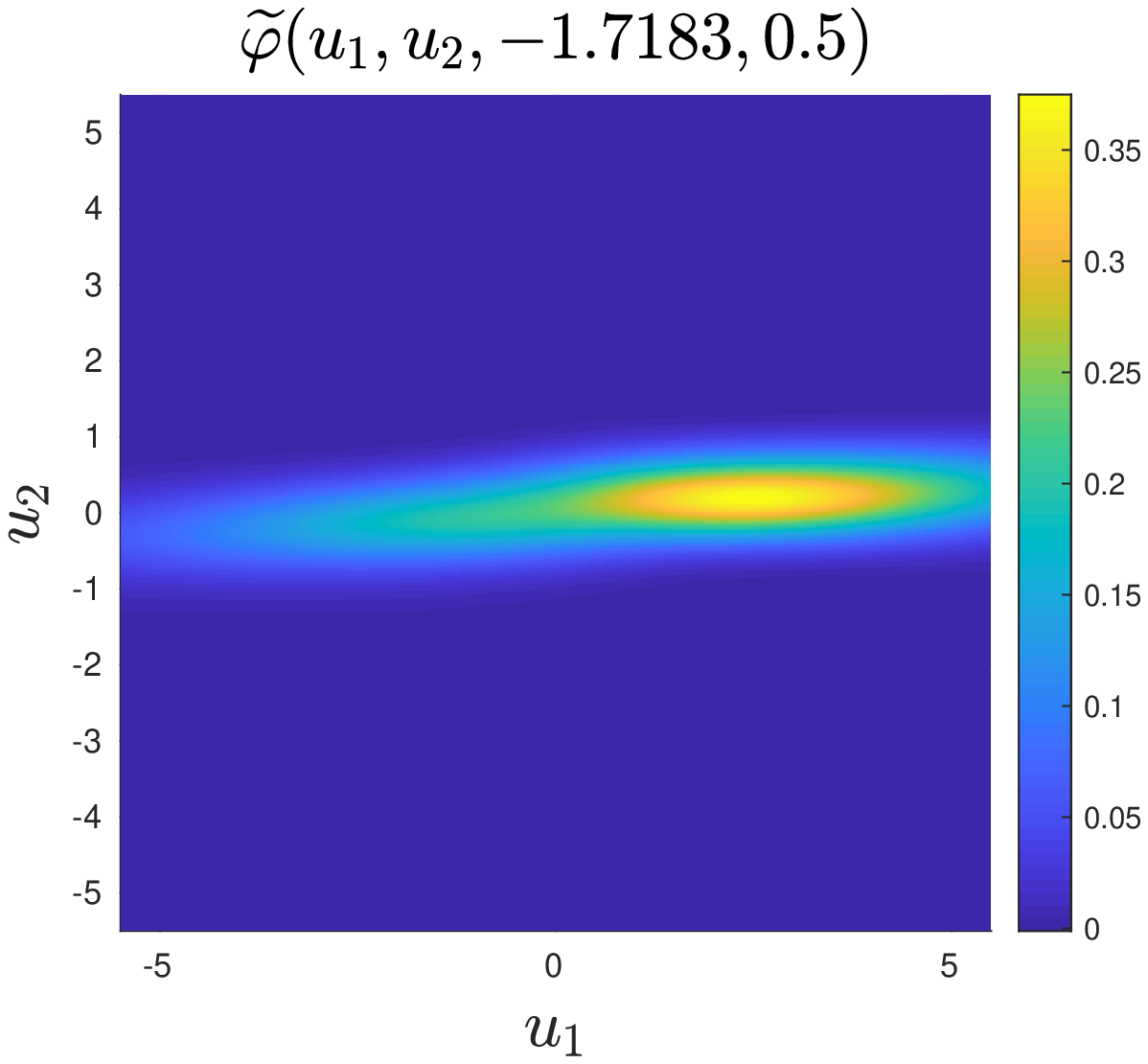}
		\caption{Comparison between the chf $\varphi^{MC}(\boldsymbol{u},t)$ obtained from Monte Carlo samples (left) and the neural network approximation $\widetilde{\varphi}(\boldsymbol{u},t)$ (right) for Section~\ref{subsec:3Dexample} at $u_3 = -1.7183, t=0.5$.}
		\label{fig:3Dtime2}
\end{figure}

\begin{figure}[h!]
		\centering
		\includegraphics[width = 0.45\textwidth]		
						{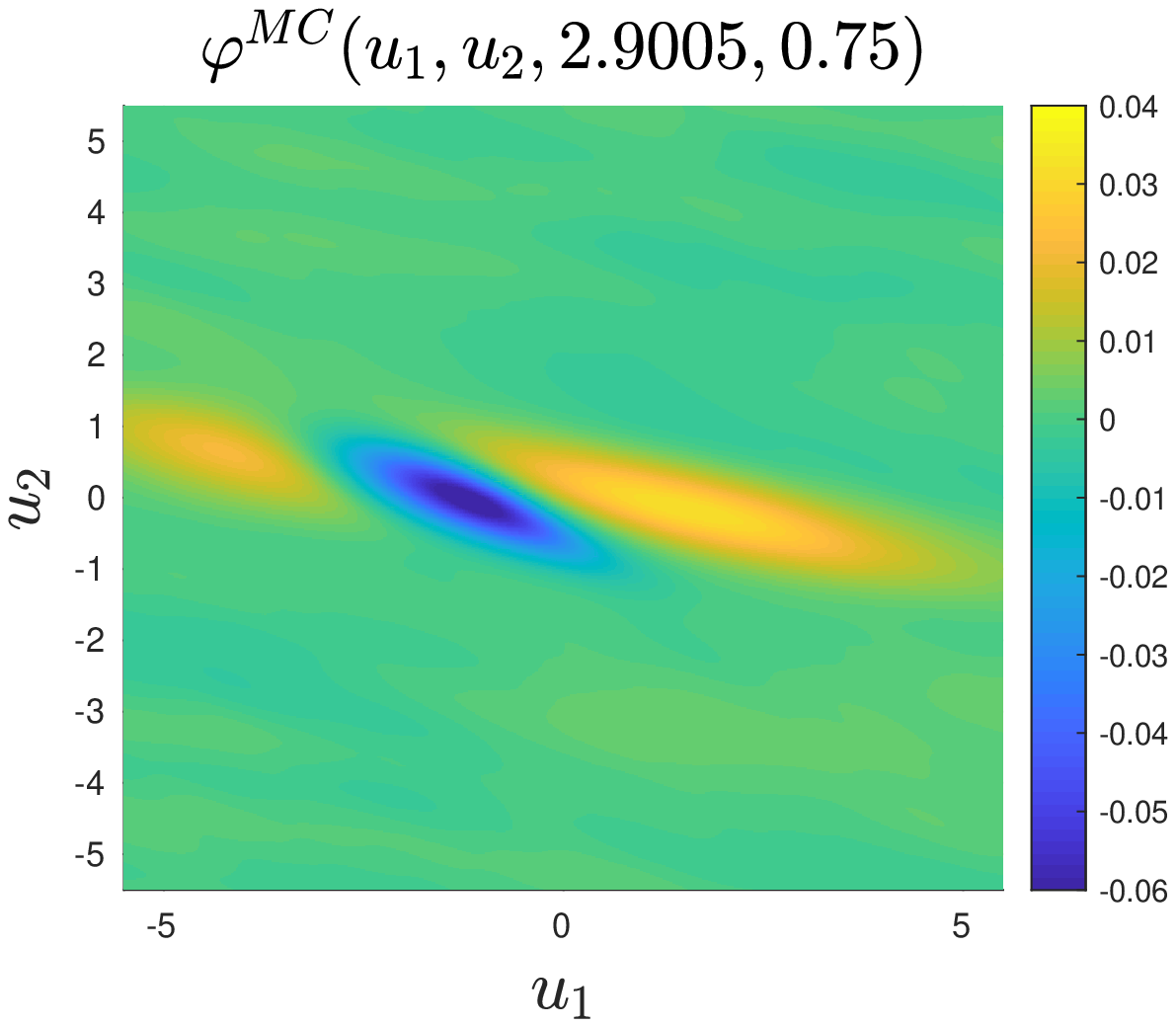} 
						\hspace{3em}
		\includegraphics[width = 0.45\textwidth]		
						{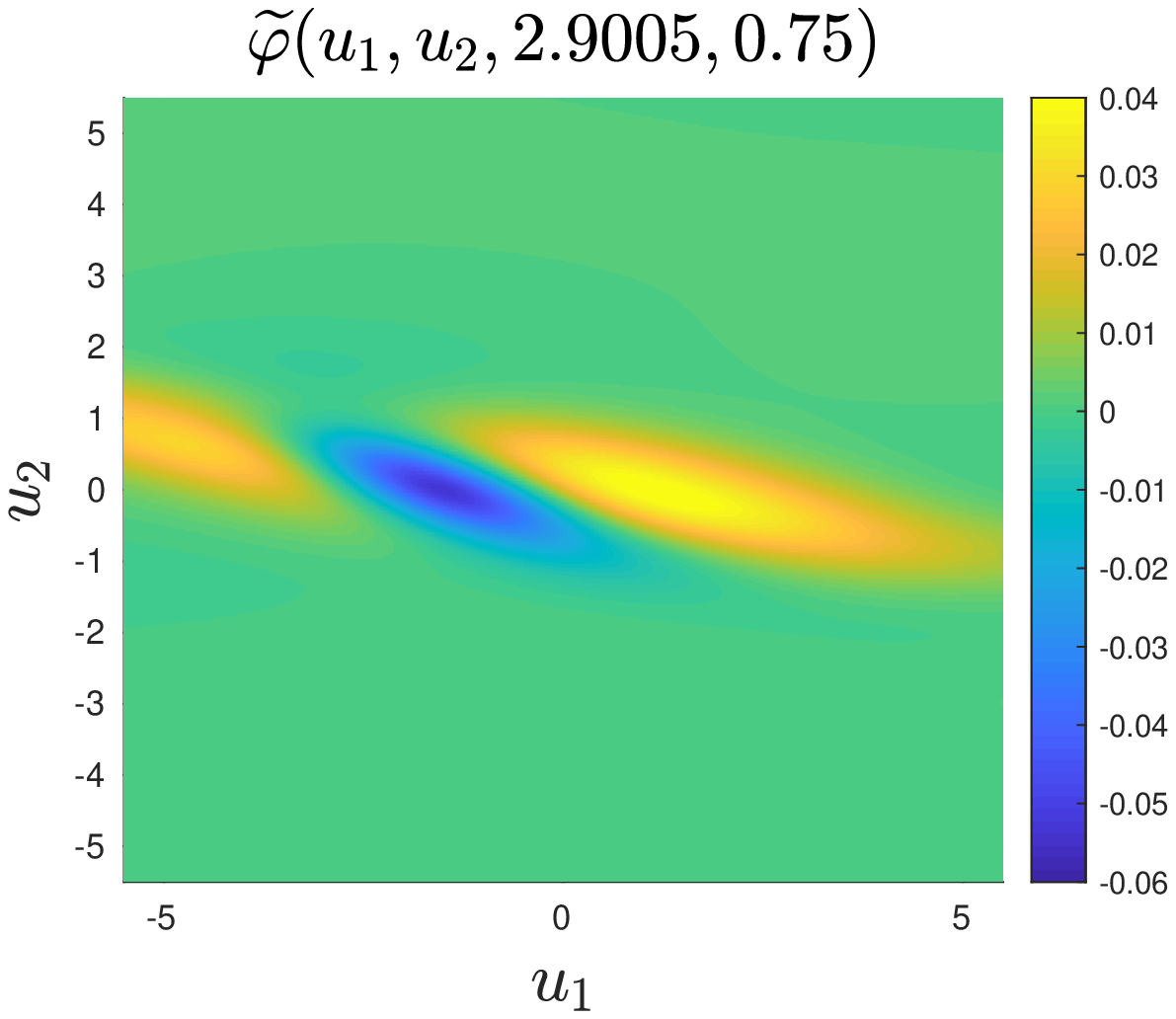}
		\caption{Comparison between the chf $\varphi^{MC}(\boldsymbol{u},t)$ obtained from Monte Carlo samples (left) and the neural network approximation $\widetilde{\varphi}(\boldsymbol{u},t)$ (right) for Section~\ref{subsec:3Dexample} at $u_3 = 2.9005, t=0.75$.}
		\label{fig:3Dtime3}
\end{figure}

\begin{figure}[h!]
		\centering
		\includegraphics[width = 0.45\textwidth]		
						{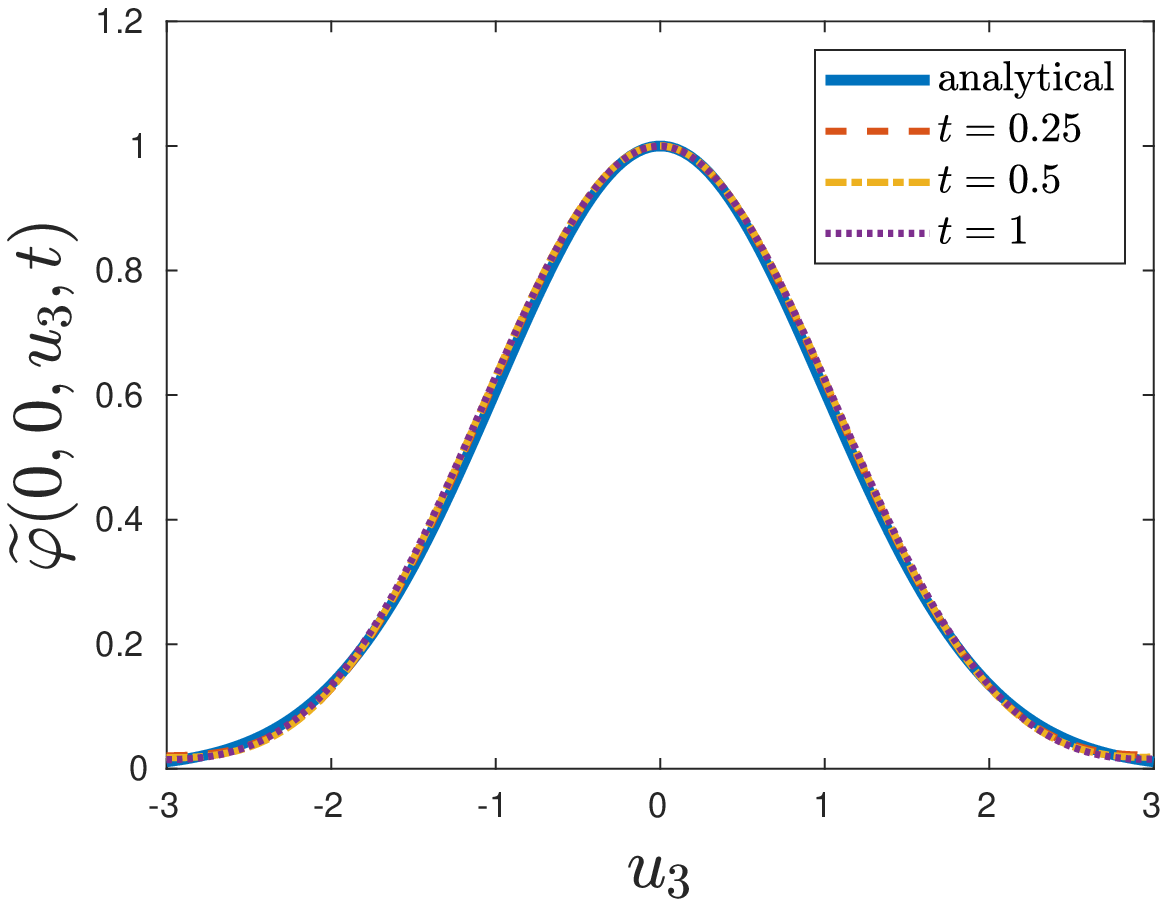} 
		\caption{Comparison of the analytical expression for $\varphi(0,0,u_3,t)$ (solid) and the neural network approximation $\widetilde{\varphi}(0,0,u_3,t)$ at various times (different styles of dashed lines) for Section~\ref{subsec:3Dexample}.}
		\label{fig:3DOU}
\end{figure}

\section{Conclusion}

This work is concerned with estimating the pdf of the state vector whose dynamics are described by a diffusion process. This is traditionally accomplished by solving the Fokker-Planck equation which describes the time evolution of the pdf.
Since solving this PDE through standard numerical methods may be infeasible for spatial dimensions larger than 3, the use of physics-informed neural networks to approximate the solution to this PDE was investigated. In addition, we sought a neural network solution to the differential equation for the characteristic of the state from which the pdf of the state can be deduced. By incorporating probabilistic constraints on the pdf and chf, we outlined strategies for designing the loss function to train a neural network representation for these quantities.

Through a wide variety of applications, it was demonstrated that the neural network approximation to the pdf or chf can match the analytical or Monte Carlo solution even for integro-differential equations and systems of PDEs. They also highlighted the advantages and disadvantages of solving one type of differential equation over another.  For example, while the differential equation for the chf is complex-valued with high-order derivatives, there are instances for which a PDE for the pdf is unavailable. Imposing the normalization constraint of the Fokker-Planck equation may also present numerical challenges in the neural network solution. Nevertheless, the applications underscore that solving either type of differential equation with neural networks offers consistent information on the pdf of the state and that the neural network representation of the pdf or the chf may be useful for a wide variety of applications. The ideas developed here can be readily extended to high-dimensional problems wherein training the neural network solution to the differential equation has to be performed in a gridless manner.

\section*{Acknowledgements}
We are grateful to Prof. Maziar Raissi and Ali Al-Aradi for sharing their code from which we based our python scripts for the neural network simulations. The work reported in this paper has been partially supported by the National Science Foundation under grant CMMI-1639669. This support is gratefully acknowledged.

\section*{Data availability statement}
The data that support the findings of this study are available from the corresponding author upon reasonable request.


\bibliographystyle{abbrv}
\bibliography{references}

\begin{thebibliography}{10}

\bibitem{paper:TensorFlow}
M.~Abadi, A.~Agarwal, P.~Barham, E.~Brevdo, Z.~Chen, C.~Citro, G.~S. Corrado,
  A.~Davis, J.~Dean, M.~Devin, S.~Ghemawat, I.~Goodfellow, A.~Harp, G.~Irving,
  M.~Isard, Y.~Jia, R.~Jozefowicz, L.~Kaiser, M.~Kudlur, J.~Levenberg,
  D.~Man\'{e}, R.~Monga, S.~Moore, D.~Murray, C.~Olah, M.~Schuster, J.~Shlens,
  B.~Steiner, I.~Sutskever, K.~Talwar, P.~Tucker, V.~Vanhoucke, V.~Vasudevan,
  F.~Vi\'{e}gas, O.~Vinyals, P.~Warden, M.~Wattenberg, M.~Wicke, Y.~Yu, and
  X.~Zheng.
\newblock {TensorFlow}: Large-scale machine learning on heterogeneous systems,
  2015.
\newblock Software available from tensorflow.org.

\bibitem{paper:AlAradiCNJS2020}
A.~Al-Aradi, A.~Correia, D.~de~Frietas~Naiff, G.~Jardim, and Y.~Saporito.
\newblock Applications of the deep galerkin method to solving partial
  integro-differential and hamilton-jacobi-bellman equations, 2020.
\newblock arXiv:1912.01455v2.

\bibitem{paper:AlAradiCNJS2018}
A.~Al-Aradi, A.~Correia, D.~Naiff, G.~Jardim, and Y.~Saporito.
\newblock Solving nonlinear and high-dimensional partial differential equations
  via deep learning, 2018.
\newblock arXiv:1811.08782.

\bibitem{paper:ChenR2018}
J.~Chen and Z.~Rui.
\newblock Dimension-reduced {FPK} equation for additive white-noise excited
  nonlinear structures.
\newblock {\em Probabilistic Engineering Mechanics}, 53:1--13, June 2018.

\bibitem{paper:ChenM2018}
N.~Chen and A.~J. Majda.
\newblock Efficient statistically accurate algorithms for the
  fokker{\textendash}planck equation in large dimensions.
\newblock {\em Journal of Computational Physics}, 354:242--268, Feb. 2018.

\bibitem{paper:ChenMT2018}
N.~Chen, A.~J. Majda, and X.~T. Tong.
\newblock Rigorous analysis for efficient statistically accurate algorithms for
  solving fokker--planck equations in large dimensions.
\newblock {\em {SIAM}/{ASA} Journal on Uncertainty Quantification},
  6(3):1198--1223, Jan. 2018.

\bibitem{paper:ChoVK2016}
H.~Cho, D.~Venturi, and G.~Karniadakis.
\newblock Numerical methods for high-dimensional probability density function
  equations.
\newblock {\em Journal of Computational Physics}, 305:817--837, Jan. 2016.

\bibitem{paper:Cybenko1989}
G.~Cybenko.
\newblock Approximation by superpositions of a sigmoidal function.
\newblock {\em Mathematics of Control, Signals, and Systems}, 2(4):303--314,
  Dec. 1989.

\bibitem{paper:DektorV2019}
A.~Dektor and D.~Venturi.
\newblock Dynamically orthogonal tensor methods for high-dimensional nonlinear
  pdes, 2019.
\newblock arXiv:1907.05924.

\bibitem{book:GoodfellowBC2016}
I.~Goodfellow, Y.~Bengio, and A.~Courville.
\newblock {\em Deep Learning}.
\newblock MIT Press, 2016.
\newblock \url{http://www.deeplearningbook.org}.

\bibitem{book:Grigoriu2002}
M.~Grigoriu.
\newblock {\em Stochastic calculus. Applications in science and engineering.}
\newblock Birkh\"auser, Boston, 2002.

\bibitem{paper:Grigoriu2004}
M.~Grigoriu.
\newblock Characteristic function equations for the state of dynamic systems
  with gaussian, poisson, and l{\'{e}}vy white noise.
\newblock {\em Probabilistic Engineering Mechanics}, 19(4):449--461, 2004.

\bibitem{paper:Grigoriu2009}
M.~Grigoriu.
\newblock Reliability of linear systems under poisson white noise.
\newblock {\em Probabilistic Engineering Mechanics}, 24(3):397--406, July 2009.

\bibitem{paper:HornikSW1989}
K.~Hornik, M.~Stinchcombe, and H.~White.
\newblock Multilayer feedforward networks are universal approximators.
\newblock {\em Neural Networks}, 2(5):359--366, Jan. 1989.

\bibitem{book:KloedenP1992}
P.~E. Kloeden and E.~Platen.
\newblock {\em Numerical Solution of Stochastic Differential Equations}.
\newblock Springer Berlin Heidelberg, 1992.

\bibitem{paper:MasudB2005}
A.~Masud and L.~A. Bergman.
\newblock Solution of the four dimensional fokker--planck equation: still a
  challenge.
\newblock In G.~Augusti, G.~Schu\"{e}ller, and M.~Ciampoli, editors, {\em
  ICOSSAR}, Millpress, Rotterdam, 2005.

\bibitem{paper:PichlerMB2013}
L.~Pichler, A.~Masud, and L.~A. Bergman.
\newblock Numerical solution of the fokker{\textendash}planck equation by
  finite difference and finite element methods{\textemdash}a comparative study.
\newblock In {\em Computational Methods in Stochastic Dynamics}, pages 69--85.
  Springer Netherlands, 2013.

\bibitem{paper:RaissiPK2019}
M.~Raissi, P.~Perdikaris, and G.~Karniadakis.
\newblock Physics-informed neural networks: A deep learning framework for
  solving forward and inverse problems involving nonlinear partial differential
  equations.
\newblock {\em Journal of Computational Physics}, 378:686--707, Feb. 2019.

\bibitem{book:Risken1989}
H.~Risken.
\newblock {\em The Fokker-Planck Equation}.
\newblock Springer Berlin Heidelberg, 1989.

\bibitem{paper:SirignanoS2018}
J.~Sirignano and K.~Spiliopoulos.
\newblock {DGM}: A deep learning algorithm for solving partial differential
  equations.
\newblock {\em Journal of Computational Physics}, 375:1339--1364, Dec. 2018.

\bibitem{book:Skorohod1982}
A.~Skorohod.
\newblock {\em Studies in the theory of random processes}.
\newblock Dover Publications, Inc. New York, 1982.

\bibitem{paper:Stein1987}
M.~Stein.
\newblock Large sample properties of simulations using latin hypercube
  sampling.
\newblock {\em Technometrics}, 29(2):143--151, 1987.

\bibitem{paper:WojtkiewiczB2000}
S.~Wojtkiewicz and L.~Bergman.
\newblock Numerical solution of high dimensional fokker-planck equations.
\newblock In {\em 8th ASCE Specialty Conference on Probabilistic Mechanics and
  Structural Reliability}, 2000.

\bibitem{paper:XuZLZLK2020}
Y.~Xu, H.~Zhang, Y.~Li, K.~Zhou, Q.~Liu, and J.~Kurths.
\newblock Solving fokker-planck equation using deep learning.
\newblock {\em Chaos: An Interdisciplinary Journal of Nonlinear Science},
  30(1):013133, Jan. 2020.

\end{thebibliography}

\end{document}